\documentclass[10pt]{amsart}

\usepackage{amsmath,amsfonts,amssymb,amsthm}
\usepackage{xcolor}
\usepackage{tikz}

\newtheorem{theorem}{Theorem}[subsection]
\newtheorem{lemma}[theorem]{Lemma}
\newtheorem{proposition}[theorem]{Proposition}
\newtheorem{corollary}[theorem]{Corollary}

\theoremstyle{definition}
\newtheorem{definition}[theorem]{Definition}
\newtheorem{example}[theorem]{Example}

\theoremstyle{remark}
\newtheorem{remark}[theorem]{Remark}

\numberwithin{equation}{subsection}

\newcommand{\KD}{\ensuremath\mathrm{KD}}
\newcommand{\kd}{\ensuremath\mathrm{key}}
\newcommand{\SSYT}{\ensuremath\mathrm{SSYT}}

\newcommand{\wt}{\ensuremath\mathrm{\mathbf{wt}}}
\newcommand{\cwt}{\ensuremath\mathrm{\mathbf{cwt}}}

\newcommand{\comp}[1]{\ensuremath\mathbf{#1}}

\newcommand{\supp}{\ensuremath\mathbf{supp}}
\newcommand{\Row}{\ensuremath\mathrm{Row}}
\newcommand{\drop}{\ensuremath\mathbf{drop}}

\newcommand{\lswap}{\ensuremath\mathrm{lswap}}
\newcommand{\thread}{\ensuremath\boldsymbol{\theta}}

\newcommand{\rectify}{\ensuremath\mathrm{rectify}}

\newcommand{\key}{\ensuremath\kappa}
\newcommand{\schubert}{\ensuremath\mathfrak{S}}

\newcommand{\KDs}{\ensuremath\mathcal{D}}
\newcommand{\KDstr}{\ensuremath\overline{\mathcal{D}}(\comp{a},k)}
\newcommand{\KDstrm}{\ensuremath\overline{\mathcal{D}}^{(m)}(\comp{a},k)}

\newcommand{\setW}{\ensuremath U^{=k}}
\newcommand{\setV}{\ensuremath U^{<k}}
\newcommand{\setVV}{\ensuremath U^{\leq k}}
\newcommand{\setY}{\ensuremath U^{\mathrm{rect}}}

\newcommand{\col}{\ensuremath\mathrm{col}}
\newcommand{\droppable}{\ensuremath\mathfrak{d}}

\newcommand{\ins}{\Delta}
\newcommand{\Le}{\ensuremath\mathbf{Le}}
\newcommand{\Lbl}{\ensuremath\mathcal{L}}

\newcommand{\matchable}{\ensuremath\mathfrak{m}}
\newcommand{\Mtch}{\ensuremath\mathcal{M}}

\newcommand{\plength}{\ensuremath\mu}

\newcommand{\row}{\ensuremath\mathrm{row}}

\newcommand{\Tab}{\mathbb{D}^{-1}}
\newcommand{\thd}{\thread}

\newcommand{\rect}{\ensuremath\varrho}

\newlength\cellsize 
\newlength\smcellsize 

\savebox2{
\begin{picture}(8,8)
\put(0,0){\line(1,0){8}}
\put(0,0){\line(0,1){8}}
\put(8,0){\line(0,1){8}}
\put(0,8){\line(1,0){8}}
\end{picture}}

\newcommand\boxify[1]{\def\thearg{#1}\def\nothing{}%
\ifx\thearg\nothing\vrule width0pt height\cellsize depth0pt%
  \else\hbox to 0pt{\usebox2\hss}\fi%
  \vbox to \cellsize{\vss\hbox to \cellsize{\hss$_{#1}$\hss}\vss}}

\savebox3{
\begin{picture}(8,8)
\put(4,4){\circle{8}}
\end{picture}}

\newcommand{\circify}[1]{\def\thearg{#1}\def\nothing{}%
\ifx\thearg\nothing\vrule width0pt height\smcellsize depth0pt%
  \else\hbox to 0pt{\usebox3\hss}\fi%
  \vbox to \smcellsize{\vss\hbox to \smcellsize{\hss$_{#1}$\hss}\vss}}

\newcommand{\grayify}[1]{\def\thearg{#1}\def\nothing{}%
\ifx\thearg\nothing\vrule width0pt height\smcellsize depth0pt%
  \else\hbox to 0pt{\usebox3\hss}\fi%
  \vbox to \smcellsize{\vss\hbox to \smcellsize{\hss$\color{gray}_{#1}$\hss}\vss}}
\newcommand{\redify}[1]{\def\thearg{#1}\def\nothing{}%
\ifx\thearg\nothing\vrule width0pt height\smcellsize depth0pt%
  \else\hbox to 0pt{\usebox3\hss}\fi%
  \vbox to \smcellsize{\vss\hbox to \smcellsize{\hss$\color{red}_{#1}$\hss}\vss}}

\newcommand\nullify[1]{\def\thearg{#1}\def\nothing{}%
\ifx\thearg\nothing\vrule width0pt height\cellsize depth0pt%
  \else\hbox to 0pt{\hss}\fi%
  \vbox to \cellsize{\vss\hbox to \cellsize{\hss$_{#1}$\hss}\vss}}

\newcommand\tableau[1]{\vtop{\let\\=\cr
\setlength\baselineskip{-8000pt}
\setlength\lineskiplimit{8000pt}
\setlength\lineskip{0pt}
\halign{&\boxify{##}\cr#1\crcr}}}

\newcommand\cirtab[1]{\vtop{\let\\=\cr
\setlength\baselineskip{-8000pt}
\setlength\lineskiplimit{8000pt}
\setlength\lineskip{0pt}
\halign{&\circify{##}\cr#1\crcr}}}

\newcommand\nulltab[1]{\vtop{\let\\=\cr
\setlength\baselineskip{-8000pt}
\setlength\lineskiplimit{8000pt}
\setlength\lineskip{0pt}
\halign{&\nullify{##}\cr#1\crcr}}}

\definecolor{boxgray}{gray}{.7}

\newcommand{\cball}[1]{%
  \begin{tikzpicture}
    \filldraw[fill=#1!50,draw=black] circle (4pt);
  \end{tikzpicture}
}

\newcommand{\mball}{\makebox[0pt]{\raisebox{1.5pt}{$\leftarrow$}}\cball{red}}

\usetikzlibrary{calc}
\newcommand{\convexpath}[2]{
  [   
  create hullcoords/.code={
    \global\edef\namelist{#1}
    \foreach [count=\counter] \nodename in \namelist {
      \global\edef\numberofnodes{\counter}
      \coordinate (hullcoord\counter) at (\nodename);
    }
    \coordinate (hullcoord0) at (hullcoord\numberofnodes);
    \pgfmathtruncatemacro\lastnumber{\numberofnodes+1}
    \coordinate (hullcoord\lastnumber) at (hullcoord1);
  },
  create hullcoords
  ]
  ($(hullcoord1)!#2!-90:(hullcoord0)$)
  \foreach [
  evaluate=\currentnode as \previousnode using \currentnode-1,
  evaluate=\currentnode as \nextnode using \currentnode+1
  ] \currentnode in {1,...,\numberofnodes} {
    let \p1 = ($(hullcoord\currentnode) - (hullcoord\previousnode)$),
    \n1 = {atan2(\y1,\x1) + 90},
    \p2 = ($(hullcoord\nextnode) - (hullcoord\currentnode)$),
    \n2 = {atan2(\y2,\x2) + 90},
    \n{delta} = {Mod(\n2-\n1,360) - 360}
    in 
    {arc [start angle=\n1, delta angle=\n{delta}, radius=#2]}
    -- ($(hullcoord\nextnode)!#2!-90:(hullcoord\currentnode)$) 
  }
}

\begin{document}


\title[Pieri rule for Demazure characters]{A Pieri rule for Demazure characters \\ of the general linear group}  

\author[Assaf]{Sami Assaf}
\address{Department of Mathematics, University of Southern California, 3620 S. Vermont Ave., Los Angeles, CA 90089-2532, U.S.A.}
\email{shassaf@usc.edu}
\thanks{Work supported in part by NSF DMS-1763336.}

\author[Quijada]{Danjoseph Quijada}
\address{Department of Mathematics, University of Southern California, 3620 S. Vermont Ave., Los Angeles, CA 90089-2532, U.S.A.}
\email{dquijada@usc.edu}

\subjclass[2010]{Primary 05E05; Secondary 05E10, 14N10, 14N15}



\keywords{Demazure characters, key polynomials, Pieri rule, RSK insertion}

\begin{abstract}
  The Pieri rule is a nonnegative, multiplicity-free formula for the Schur function expansion of the product of an arbitrary Schur function with a single row Schur function. Key polynomials are characters of Demazure modules for the general linear group that generalize the Schur function basis of symmetric functions to a basis of the full polynomial ring. We prove a nonsymmetric generalization of the Pieri rule by giving a cancellation-free, multiplicity-free formula for the key polynomial expansion of the product of an arbitrary key polynomial with a single part key polynomial. Our proof is combinatorial, generalizing the Robinson--Schensted--Knuth insertion algorithm on tableaux to an insertion algorithm on Kohnert diagrams.
\end{abstract}

\maketitle
\tableofcontents

%
\section{Introduction}
%
\label{sec:introduction}

Schur polynomials are ubiquitous throughout mathematics. They arise in representation theory as characters for irreducible polynomial representations of the general linear group and as Frobenius characters for irreducible representations of the symmetric group; they arise in geometry as polynomial representatives for the cohomology classes of Schubert cycles in Grassmannians. Thus the Schur polynomials are naturally indexed by integer partitions. In these contexts, combinatorial rules for expanding the product of Schur polynomials in the Schur basis has deep meaning, in the former cases giving the irreducible decomposition of tensor products or restrictions of modules, and in the latter case giving intersection multiplicities for Schubert varieties. The celebrated Pieri rule \cite{Pie93}, originally stated in the context of Schubert Calculus \cite{Sch79}, gives a manifestly positive combinatorial formula for computing these structure constants when one of the factors has a special form, namely consists of a single nonzero part. In this case, the Schur expansion is multiplicity-free, meaning that the only nonzero coefficient appearing is $1$.

Demazure \cite{Dem74} generalized the Weyl character formula to certain submodules of irreducible modules generated by extremal weight spaces under the action of a Borel subalgebra of a Lie algebra. The \emph{Demazure modules} originally arose in connection with Schubert calculus \cite{Dem74a} and have since been shown to have deep connections with certain specializations of nonsymmetric Macdonald polynomials \cite{San00,Ion03,Ass18}. A \emph{Demazure character} is naturally indexed by a highest weight and an element of the Weyl group, which, in the case of the general linear group, becomes an integer partition and a permutation. When the permutation is taken to be the longest element, the Demazure module is the full irreducible module, and so the corresponding Demazure character is a Schur polynomial. In general, the set of Demazure characters of the general linear group form a basis of the polynomial ring often called \emph{key polynomials} and are indexed by weak compositions obtained by acting on the given partition with the given permutation. As the key polynomials form a basis of the polynomial ring containing the Schur polynomials, it is natural to consider the expansion of a product of key polynomials into the key basis. However the coefficients appearing are not, in general, nonnegative.

We prove, in Theorem~\ref{thm:monkey}, a cancellation-free combinatorial formula for these structure constants when one of the factors has a special form parallel to the Pieri case for Schur polynomials. Moreover, our \emph{nonsymmetric Pieri rule} is multiplicity-free in the sense that the only nonzero coefficients appearing are $1$ and $-1$.

Our proof of this new rule is combinatorial, arising via a bijection, stated in Theorem~\ref{thm:RSKD}, on the combinatorial model of \emph{Kohnert diagrams} \cite{Koh91} that generates key polynomials. This bijection generalizes the beautiful Robinson--Schensted--Knuth insertion algorithm \cite{Rob38,Sch61,Knu70} that can be used to prove Pieri's formula for Schur polynomials. We give simple, direct proofs of the bijection for two extreme cases in Section~\ref{sec:bijection}, and we develop new tools for studying the combinatorics of Kohnert diagrams to prove the general case in Section~\ref{sec:stratify}. 

Related to our nonsymmetric Pieri rule, Haglund, Luoto, Mason and van Willigenburg \cite{HLMvW11-2} give a \emph{nonnegative} combinatorial formula for the key expansion of a product of a key polynomial and a Schur polynomial with a certain number of variables. Our formula is more general in that we do not restrict the number of variables for the Schur polynomial, though it is less general in that we consider only Schur polynomials indexed by partitions with one part. When both formulas apply, we are in one of the extremal cases where our proof simplifies greatly to give a simpler proof and formula than what is given in \cite{HLMvW11-2}.

Assaf and McNamara give a \emph{skew Pieri rule} \cite{AM11} involving signs that is also proved by generalizing Robinson--Schensted--Knuth insertion, but the proof there uses a sign-reversing involution to cancel terms. In contrast, the signs in our nonsymmetric Pieri formula arise from the fact that the image of the bijection induced by insertion is to a union of objects that is, in general, not disjoint. As a corollary, in Section~\ref{sec:applications-pos}, we characterize when the union in the image is disjoint, and thus characterize when the key expansion of the product of key polynomial and a single-part key polynomial is nonnegative. We apply these characterizations to the \emph{Schubert polynomials} of Lascoux and Sch{\"u}tzenberger \cite{LS82}, providing geometric motivation for these results.

%
\section{Key combinatorics}
%
\label{sec:kohnert}

Schur polynomials may be defined combinatorially as the generating polynomials of \emph{semistandard Young tableaux}, certain positive integer fillings of \emph{Young diagrams}, graphical representations of integer partitions as diagrams of unit cells in the plane. The Demazure characters for the general linear group, studied combinatorially as \emph{standard bases} by Lascoux and Sch{\"u}tzenberger \cite{LS90} and Kohnert \cite{Koh91}, then under the name \emph{key polynomials} by Reiner and Shimozono \cite{RS95}, Mason \cite{Mas09}, Assaf and Searles \cite{AS18}, Assaf \cite{Ass-W} and others, can be characterized in many equivalent ways. In Section~\ref{sec:kohnert-diagrams}, we review Kohnert's \cite{Koh91} elegant combinatorial algorithm for computing a key polynomial based on diagrams, which lies at the heart of our Pieri rule for key polynomials. In Section~\ref{sec:kohnert-lswap}, we review a partial order on weak compositions studied by Assaf and Searles \cite{AS18} that allows us to characterize when diagrams corresponding to one key polynomial also correspond to another. 

\subsection{Young tableaux}
\label{sec:kohnert-tableaux}

Throughout, we fix a positive integer $n$ and consider polynomials in variables $x_1, x_2, \ldots, x_n$.

A \emph{partition} $\lambda$ of $m$ is sequence $(\lambda_1 \geq \cdots \geq \lambda_{\ell} > 0)$ of non-negative integers that sum to $m$. We insist $\ell \leq n$ and set $\lambda_i=0$ for all $\ell<i\leq n$. We draw the \emph{Young diagram} of a partition $\lambda$ in English notation as the set of left justified unit cells with $\lambda_{n-i+1}$ cells in row $i$ indexed from the bottom as depicted in Fig.~\ref{fig:5441}.

\begin{figure}[ht]
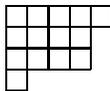

  \begin{displaymath}
    \tableau{ \ & \ & \ & \ & \ \\
      \ & \ & \ & \ \\
      \ & \ & \ & \ \\
      \ }
  \end{displaymath}
  \caption{\label{fig:5441} The Young diagram for the partition $(5,4,4,1)$.}
\end{figure}

A \emph{semistandard Young tableau of shape $\lambda$} is a filling of the cells of the Young diagram of $\lambda$ with positive integers such that entries weakly increase left to right along rows and strictly increase top to bottom down columns. Let $\SSYT_n(\lambda)$ denote the set of semistandard Young tableaux of shape $\lambda$ with image in $\{1,2,\ldots,n\}$. Here we emphasize that entries may not exceed $n$. For example, the semistandard Young tableaux of shape $(3,2)$ with largest entry $3$ are given in Fig.~\ref{fig:SSYT}.

\begin{figure}[ht]
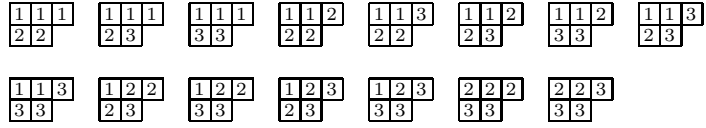

  \begin{displaymath}
    \begin{array}{cccccccc}
      \tableau{1 & 1 & 1 \\ 2 & 2 } &
      \tableau{1 & 1 & 1 \\ 2 & 3 } &
      \tableau{1 & 1 & 1 \\ 3 & 3 } &
      \tableau{1 & 1 & 2 \\ 2 & 2 } &
      \tableau{1 & 1 & 3 \\ 2 & 2 } &
      \tableau{1 & 1 & 2 \\ 2 & 3 } &
      \tableau{1 & 1 & 2 \\ 3 & 3 } &
      \tableau{1 & 1 & 3 \\ 2 & 3 } \\ \\
      \tableau{1 & 1 & 3 \\ 3 & 3 } &
      \tableau{1 & 2 & 2 \\ 2 & 3 } &
      \tableau{1 & 2 & 2 \\ 3 & 3 } &
      \tableau{1 & 2 & 3 \\ 2 & 3 } &
      \tableau{1 & 2 & 3 \\ 3 & 3 } &
      \tableau{2 & 2 & 2 \\ 3 & 3 } &
      \tableau{2 & 2 & 3 \\ 3 & 3 } 
    \end{array}
  \end{displaymath}
  \caption{\label{fig:SSYT}The set $\SSYT_3(3,2)$ of semistandard Young tableaux of shape $(3,2)$ with largest entry $3$.}
\end{figure}

A \emph{weak composition} $\comp{a} = (a_1,\ldots,a_{n})$ is a sequence $n$ of nonnegative integers. To each semistandard Young tableau $T$, we associate the weak composition $\wt(T)$ whose $i$th component is equal to the number of occurrences of $i$ in $T$. For example, the weights of the first column of Fig.~\ref{fig:SSYT} are $(3,2,0),(2,0,3)$, from top to bottom.

Classically, Schur polynomials may be defined combinatorially as the generating polynomials of semistandard Young tableaux as follows.

\begin{definition}
  The \emph{Schur polynomial} $s_{\lambda}(x_1,\ldots,x_n)$ is given by
  \begin{equation}
    s_{\lambda}(x_1,\ldots,x_n) = \sum_{T \in \SSYT_n(\lambda)} x_1^{\wt(T)_1} \cdots x_n^{\wt(T)_n}.
    \label{e:schur}
  \end{equation}
  \label{def:schur}
\end{definition}

For example, from Fig.~\ref{fig:SSYT} we compute the Schur polynomial
\begin{eqnarray*}
  s_{(3,2)}(x_1,x_2,x_3) & = & x_2^2 x_3^3 + x_1 x_2 x_3^3 + x_1^2 x_3^3 + x_2^3 x_3^2 + 2 x_1 x_2^2 x_3^2 + 2 x_1^2 x_2 x_3^2 \\
  & & + x_1^3 x_3^2 + x_1 x_2^3 x_3 + 2 x_1^2 x_2^2 x_3 + x_1^3 x_2 x_3 + x_1^2 x_2^3 + x_1^3 x_2^2.
\end{eqnarray*}

\subsection{Kohnert diagrams}
\label{sec:kohnert-diagrams}

A \emph{diagram} is any finite collection of unit cells in the first quadrant. To distinguish between generic diagrams and Young diagrams, we draw cells of generic diagrams as unit circles.

The \emph{key diagram} of a weak composition $\comp{a}$, denoted by $\kd_{\comp{a}}$ as the set of left justified cells with $a_i$ in row $i$ indexed in Cartesian coordinates. For example, Fig.~\ref{fig:40514} shows the key diagram for the weak composition $(4,1,5,0,4)$. 

\begin{figure}[ht]
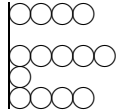

  \begin{displaymath}
    \vline\cirtab{ \ & \ & \ & \ \\
      \\
      \ & \ & \ & \ & \ \\
      \ \\
      \ & \ & \ & \  } 
  \end{displaymath}
  \caption{\label{fig:40514} The key diagram for the weak composition $(4,1,5,0,4)$.}
\end{figure}

\begin{definition}[\cite{Koh91}]
  A \emph{Kohnert move} on a diagram selects the rightmost cell of a given row and moves the cell down within its column to the first available position below, if it exists, jumping over other cells in its way as needed. 
\label{def:move}
\end{definition}

Denote the set of diagrams that can be obtained by Kohnert moves from the key diagram $\kd_{\comp{a}}$ by $\KD(\comp{a})$. Note that there might be multiple ways to obtain a diagram from different Kohnert moves of a given diagram, but each resulting diagram is included in the set exactly once.

\begin{example}
  To construct the set $\KD(0,3,2)$ of Kohnert diagrams for $(0,3,2)$, we begin with the key diagram $\kd_{(0,3,2)}$ show at the top of Fig.~\ref{fig:kohnert}. Selecting the second row, we may move the rightmost cell down to the first. Selecting the third row, we must also move the rightmost cell down to the first row, jumping over the cell in the same column and second row. Fig.~\ref{fig:kohnert} shows all diagrams that can be obtained via Kohnert moves from the key diagram of $(0,3,2)$. 
  \label{ex:KD}
\end{example}

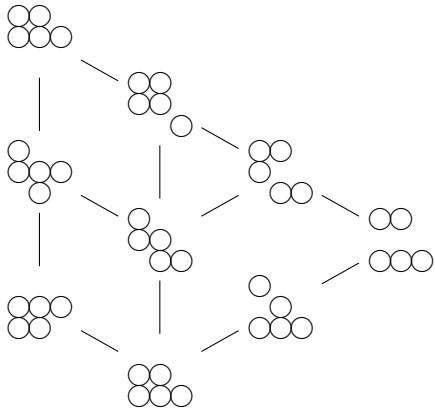
\begin{figure}[ht]
  \begin{center}
    \begin{tikzpicture}[xscale=1.6,yscale=0.9]
      \node at (0,5) (A) {$\vline\cirtab{  \ &  \ \\  \ &  \ &  \ \\ & }$};
      \node at (1,4) (B) {$\vline\cirtab{  \ &  \ \\  \ &  \ \\ & & \ }$};
      \node at (2,3) (C) {$\vline\cirtab{  \ &  \ \\  \ \\ &  \ &  \ }$};
      \node at (3,2) (D) {$\vline\cirtab{  \ &  \ \\ \\  \ &  \ &  \ }$};
      \node at (1,2) (E) {$\vline\cirtab{  \ \\  \ &  \ \\ &  \ &  \ }$};
      \node at (2,1) (F) {$\vline\cirtab{  \ \\ &  \ \\  \ &  \ &  \ }$};
      \node at (1,0) (G) {$\vline\cirtab{ \\  \ &  \ \\  \ &  \ &  \ }$};
      \node at (0,3) (H) {$\vline\cirtab{  \ \\  \ &  \ &  \ \\ &  \ }$};
      \node at (0,1) (I) {$\vline\cirtab{ \\  \ &  \ &  \ \\  \ &  \ }$};
      \draw[thin] (A) -- (H) ;
      \draw[thin] (A) -- (B) ;
      \draw[thin] (H) -- (I) ;
      \draw[thin] (H) -- (E) ;
      \draw[thin] (B) -- (E) ;
      \draw[thin] (B) -- (C) ;
      \draw[thin] (I) -- (G) ;
      \draw[thin] (E) -- (G) ;
      \draw[thin] (C) -- (E) ;
      \draw[thin] (C) -- (D) ;
      \draw[thin] (D) -- (F) ;
      \draw[thin] (F) -- (G) ;
    \end{tikzpicture}
    \caption{\label{fig:kohnert}Construction of Kohnert diagrams for $(0,3,2)$, where southeast edges indicate Kohnert moves on the second row, and south or southwest edges indicate Kohnert moves on the third row.}
  \end{center}
\end{figure}

To each diagram $T$ of cells in the first quadrant we associate the weak composition $\wt(T)$ whose $i$th component is equal to the number of cells in row $i$ of $T$. For example, the weights of diagrams in the leftmost column of Fig.~\ref{fig:kohnert} are $(0,3,2),(1,3,1),(2,3,0)$, from top to bottom.

We take as our definition the following result of Kohnert \cite{Koh91} that characterizes key polynomials as the generating polynomials for Kohnert diagrams.

\begin{definition}
  The \emph{key polynomial} $\key_{\comp{a}}$ is given by
  \begin{equation}
    \key_{\comp{a}} = \sum_{T \in \KD(\comp{a})} x_1^{\wt(T)_1} \cdots x_n^{\wt(T)_n} .
    \label{e:key}
  \end{equation}
  \label{def:key}  
\end{definition}

For example, from Fig.~\ref{fig:kohnert} we compute
\begin{eqnarray*}
  \key_{(0,3,2)} & = & x_2^3 x_3^2 + x_1 x_2^2 x_3^2 + x_1^2 x_2 x_3^2 + x_1^3 x_3^2 + x_1^2 x_2^2 x_3 \\ & & + x_1^3 x_2 x_3 + x_1^3 x_2^2 + x_1 x_2^3 x_3 + x_1^2 x_2^3 .
\end{eqnarray*}

Assaf and Searles \cite[Definition~4.5]{AS18} give an explicit map between Kohnert diagrams for $\comp{a}$ and semistandard Young tableaux of shape $\comp{\mathrm{sort}(a)}$ that is always injective and is surjective if and only if $\comp{a}$ is weakly increasing \cite[Theorem~4.6]{AS18}.

\begin{proposition}[\cite{AS18}]
  The map that places entry $n-i+1$ in each cell of row $i$ and raises cells of the diagram of $\comp{a}$ to partition shape $\lambda$ is a well-defined, weight-reversing, injective map $\KD(\comp{a}) \hookrightarrow \SSYT_n(\lambda)$, where $\lambda = \comp{\mathrm{sort}(a)}$. Moreover, this map is a bijection if and only if $\comp{a}$ is weakly increasing.
  \label{prop:sort}
\end{proposition}

\begin{example}
  Consider the diagram $T$ on the left of Fig.~\ref{fig:colsort}, which is a Kohnert diagram for the weak composition $(4,1,5,0,4)$. We label the cells in row $i$ with label $5-i+1$ as shown, then raise the cells to the semistandard Young tableau of shape $(5,4,4,1)$ shown on the right of Fig.~\ref{fig:colsort}. 
\end{example}

\begin{figure}[ht]
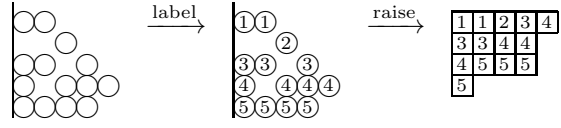

  \begin{displaymath}
    \begin{array}{ccccc}
      \vline\cirtab{%
        \ & \ \\
        & & \ \\
        \ & \ & & \ \\
        \ & & \ & \ & \ \\
        \ & \ & \ & \ }
      & \xrightarrow{\mathrm{label}} &
      \vline\cirtab{%
        1 & 1 \\
        &   & 2 \\
        3 & 3 &   & 3 \\
        4 &   & 4 & 4 & 4 \\
        5 & 5 & 5 & 5 }
      & \xrightarrow{\mathrm{raise}} &
      \tableau{%
        1 & 1 & 2 & 3 & 4 \\
        3 & 3 & 4 & 4 \\
        4 & 5 & 5 & 5 \\
        5 }
    \end{array}
  \end{displaymath}
  \caption{\label{fig:colsort}An example of the injective map from Kohnert diagrams to semistandard Young tableaux.}
\end{figure}

In particular, this immediately gives a combinatorial proof of the following.

\begin{corollary}[\cite{Mac91}]
  For $\comp{a}$ weakly increasing of length $n$, we have
  \begin{equation}
    \key_{\comp{a}} = s_{\mathrm{rev}(\comp{a})}(x_1,\ldots,x_n),
  \end{equation}
  where $\mathrm{rev}(\comp{a})$ is the partition $(a_n,a_{n-1},\ldots,a_1)$.
\end{corollary}

Note that $\key_{(0,3,2)}$ is not equal to $s_{(3,2)}(x_1,x_2,x_3)$, as the map in this case is not surjective. However, every Kohnert diagram for $(0,3,2)$ is also a Kohnert diagram for $(0,2,3)$, i.e. $\KD(0,3,2) \subset \KD(0,2,3)$. Containment of Kohnert diagrams leads to a useful partial order on weak compositions that we explore next.

\subsection{Left swap order}
\label{sec:kohnert-lswap}

Not all diagrams, for instance the diagram with two cells in positions $(1,1)$ and $(2,2)$, can result from Kohnert moves on key diagrams. We will often need to distinguish between diagrams that can arise from Kohnert moves on a key diagram with those that cannot. 

\begin{definition}
  A diagram $T$ is a \emph{generic Kohnert diagram} if there exists a weak composition $\comp{a}$ for which $T \in \KD(\comp{a})$.
  \label{def:generic}
\end{definition}

Assaf and Searles \cite[Lemma~2.2]{AS18} give a useful criterion to determine if $T$ is a generic Kohnert diagram.

\begin{proposition}[\cite{AS18}]
  A diagram $T$ is a Kohnert diagram for some weak composition if and only if for every position $(c,r)\in\mathbb{N} \times \mathbb{N}$ with $c>1$, we have
  \begin{equation}
    \#\{ (c-1,s) \in T \mid s \geq r \} \geq \#\{ (c,s) \in T \mid s \geq r \}.
    \label{e:rectified}
  \end{equation}
  \label{prop:lemma2.2}
\end{proposition}

Given a diagram $T$, the \emph{column weight of $T$}, denote by $\cwt(T)$, the weak composition whose $i$th part is the number of cells in the $i$th column of $T$. Abusing notation,  given a weak composition $\comp{a}$, the \emph{column weight of $\comp{a}$} is $\cwt(\comp{a}) = \cwt(\kd_{\comp{a}})$.

We have the following immediate consequence of Proposition~\ref{prop:lemma2.2}.

\begin{corollary}
  The column weight of a generic Kohnert diagram is a partition.
  \label{cor:cwt}
\end{corollary}

Since Kohnert moves preserve the column weight, we may define a partial order on generic Kohnert diagrams with fixed column weight $\mu$ by $S \prec T$ whenever $S$ can be obtained from $T$ by a sequence of Kohnert moves. For example, Fig.~\ref{fig:kohnert} gives the Hasse diagram for this partial order restricted to $\KD(0,3,2)$. This partial order is neither ranked nor is it a lattice, though it does have a unique bottom element $\kd_{\mu^{\prime}}$, where $\mu^{\prime}$ is the \emph{conjugate} of $\mu$ determined by $\mu^{\prime}_i = \#\{ k \mid \mu_k \geq i \}$.

Related to this order, Assaf and Searles \cite{AS18} considered a partial order on weak compositions that sort to a given partition defined as follows.

\begin{definition}[\cite{AS18}]
  A \emph{left swap} on a weak composition $\comp{a} = (a_1, a_2, \cdots )$ exchanges two parts $a_i < a_j$ with $i < j$. The \emph{left swap order} on weak compositions is the transitive closure of the relations $\comp{b} \preceq \comp{a}$ whenever $\comp{b}$ is a left swap of $\comp{a}$. 
  \label{def:lswap}
\end{definition}

\begin{figure}[ht]
  \begin{center}
    \begin{tikzpicture}[xscale=2.25,yscale=1]
      \node at (1,5.25)   (b5) {$(0,2,2,3)$};
      \node at (0.5,4.25) (a4) {$(2,0,2,3)$};
      \node at (1.5,4.25) (c4) {$(0,2,3,2)$}; 
      \node at (0,3.25)   (a3) {$(2,2,0,3)$};
      \node at (1,3.25)   (b3) {$(2,0,3,2)$};
      \node at (2,3.25)   (c3) {$(0,3,2,2)$};
      \node at (0,2)      (a2) {$(2,2,3,0)$};
      \node at (1,2)      (b2) {$(2,3,0,2)$}; 
      \node at (2,2)      (c2) {$(3,0,2,2)$}; 
      \node at (0.5,1)    (a1) {$(2,3,2,0)$};
      \node at (1.5,1)    (c1) {$(3,2,0,2)$}; 
      \node at (1,0)      (b0) {$(3,2,2,0)$}; 
      \draw[thin] (b5) -- (a4) ;
      \draw[thin] (b5) -- (c4) ;
      \draw[thin] (a4) -- (a3) ;
      \draw[thin] (a4) -- (b3) ;
      \draw[thin] (c4) -- (b3) ;
      \draw[thin] (c4) -- (c3) ;
      \draw[thin] (a3) -- (a2) ;
      \draw[thin] (a3) -- (b2) ;
      \draw[thin] (b3) -- (a2) ;
      \draw[thin] (b3) -- (b2) ;
      \draw[thin] (b3) -- (c2) ;
      \draw[thin] (c3) -- (b2) ;
      \draw[thin] (c3) -- (c2) ;
      \draw[thin] (a2) -- (a1) ;
      \draw[thin] (b2) -- (a1) ;
      \draw[thin] (b2) -- (c1) ;
      \draw[thin] (c2) -- (c1) ;
      \draw[thin] (a1) -- (b0) ;
      \draw[thin] (c1) -- (b0) ;
    \end{tikzpicture}
    \caption{\label{fig:lswap}The left swap order on weak compositions of length $4$ that sort to the partition $(3,2,2)$.}
  \end{center}
\end{figure}
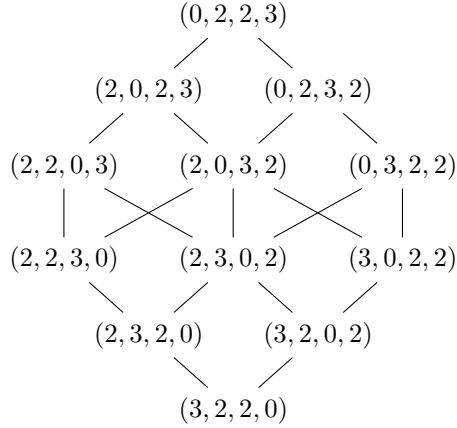

\begin{example}
  Setting $n=4$ and taking $\lambda=(3,2,2)$, the Hasse diagram for the left swap order on weak compositions of length $n$ that sort to $\lambda$ is shown in Fig.~\ref{fig:lswap}. Notice unlike the partial order in Fig.~\ref{fig:kohnert}, this order is \emph{ranked} by \emph{co-inversions}, the number of pairs $(i,j)$ such that $i<j$ and $a_i < a_j$.
\end{example}

Given a weak composition $\comp{a}$, let $\lswap(\comp{a})$ denote the set of weak compositions $\comp{b}$ for which $\comp{b} \preceq \comp{a}$. For example, we have
\[ \lswap(0,3,2) = \{(0,3,2), (3,0,2), (3,2,0), (2,3,0)\}. \]
Notice the remaining two weak compositions that sort to the partition $(3,2)$, namely $(2,0,3)$ and $(0,2,3)$, are not included in this set. Comparing with Fig.~\ref{fig:kohnert}, the weak compositions in $\lswap(0,3,2)$ are precisely those whose key diagrams are Kohnert diagrams for $(0,3,2)$, and this observation holds in general. To prove this, we require the following definition from \cite[Definition~3.5]{AS18}.

\begin{definition}[\cite{AS18}]
  The \emph{thread decomposition} of a generic Kohnert diagram partitions the cells into \emph{threads} as follows. Beginning with the rightmost column, select the lowest available cell to begin the thread. After threading a cell in column $j+1$, thread the lowest available cell in column $j$ that is weakly above the threaded cell in column $j+1$. Continue the thread until all columns are threaded or no choices remain. Continue threading until all cells are part of some thread.
  \label{def:thread}
\end{definition}

As noted in \cite{AS18}, Proposition~\ref{prop:lemma2.2} ensures each thread of the thread decomposition ends in the first column. Following \cite[Lemma~3.6]{AS18}, we may define the \emph{thread weight} of a Kohnert diagram $T$ to be the weak composition $\thread(T)$ whose $i$th part is the number of cells in the thread occupying row $i$ in the first column of $T$.

\begin{example}
  Consider the generic Kohnert diagram $T$ shown in Fig.~\ref{fig:thread}. We begin threading in column $5$ as shown in the second diagram where we label the cells of the thread with entry $3$ since the thread terminates in row $3$. The next thread begins in column $4$ with the cell in row $1$, and the cells of this thread are labeled by the terminal row $1$. We next thread in column $4$, row $3$, and we label this thread with $5$. The final thread consists of the single remaining cell in row $2$, which gets label $2$. The resulting thread weight of this diagram is $\thread(T) = (4,1,5,0,4)$.
    \label{ex:thread}
\end{example}

\begin{figure}[ht]
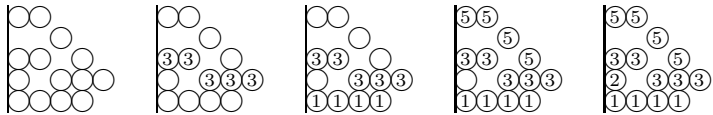

  \begin{displaymath}
    \arraycolsep=\cellsize
    \begin{array}{ccccc}
      \vline\cirtab{%
        \ & \ \\
        & & \ \\
        \ & \ & & \ \\
        \ & & \ & \ & \ \\
        \ & \ & \ & \ } &
      \vline\cirtab{%
        \ & \ \\
        & & \ \\
        3 & 3 & & \ \\
        \ & & 3 & 3 & 3 \\
        \ & \ & \ & \ } &
      \vline\cirtab{%
        \ & \ \\
        & & \ \\
        3 & 3 & & \ \\
        \ & & 3 & 3 & 3 \\
        1 & 1 & 1 & 1 } &
      \vline\cirtab{%
        5 & 5 \\
        & & 5 \\
        3 & 3 & & 5 \\
        \ & & 3 & 3 & 3 \\
        1 & 1 & 1 & 1 } &
      \vline\cirtab{%
        5 & 5 \\
        & & 5 \\
        3 & 3 & & 5 \\
        2 & & 3 & 3 & 3 \\
        1 & 1 & 1 & 1 } 
    \end{array}
  \end{displaymath}
  \caption{\label{fig:thread}An example of thread decomposition of a generic Kohnert diagram, where the cells of a given thread are labeled the same.}
\end{figure}

Implicit in \cite[Theorem~3.7]{AS18}, we have the following useful fact.

\begin{lemma}[\cite{AS18}]
  For a generic Kohnert diagram $T$, we have $T \in \KD(\comp{a})$ if and only if $\thread(T) \preceq \comp{a}$.
  \label{lem:3.7}
\end{lemma}

An easy and exceedingly useful consequence of this fact is the following.

\begin{proposition}
  Given weak compositions $\comp{a}$ and $\comp{b}$, we have $\comp{b} \preceq \comp{a}$ if and only if $\kd_{\comp{b}} \in \KD(\comp{a})$.
  \label{prop:lswap}
\end{proposition}

\begin{proof}
  Suppose $\comp{b}$ is obtained from $\comp{a}$ by a left swap for $i<j$. Beginning with $\kd_{\comp{a}}$, we may apply a Kohnert move to row $j$ until the cell lands in row $i$ below it, jumping over cells in rows $k$ if $a_k \geq a_j$. Repeating this for the rightmost $a_j-a_i$ cells in row $j$, we obtain $\kd_{\comp{b}}$. Thus $\kd_{\comp{b}}$ is a Kohnert diagram for $\comp{a}$.

  Conversely, suppose $\kd_{\comp{b}} \in \KD(\comp{a})$. By Definition~\ref{def:thread}, the threads of a key diagram necessarily consist of all cells in a given row, and so $\thread(\kd_{\comp{b}}) = \comp{b}$. Thus by Lemma~\ref{lem:3.7}, we have $\comp{b} \preceq \comp{a}$.  
\end{proof}

%
\section{Key formula}
%
\label{sec:formula}

The Pieri rule for Schur polynomials has an elegant description in terms of adding cells to Young diagrams. A beautiful combinatorial proof uses the insertion algorithm of Schensted \cite{Sch61} based on ideas of Robinson \cite{Rob38} and later generalized by Knuth \cite{Knu70} that takes a pair of semistandard Young tableaux, the latter being a single row, and maps it bijectively to another single semistandard Young tableau. Analogous to this, in Section~\ref{sec:formula-main}, we state our first main result giving a similar bijection that takes a pair of Kohnert diagrams, the latter being a single box, and maps it bijectively to another Kohnert diagram. Our expression for the union in the image of our bijection has redundancy, and in Section~\ref{sec:formula-addable} we reduce the indexing set to the maximal elements which can be described in terms of addable cells, parallel to the classical case. The image of the tableaux under RSK form a \emph{disjoint union}, from which the Schur expansion of the product of Schur polynomials can be deduced easily via generating polynomials. In Section~\ref{sec:formula-drop}, we state our second main result giving an inclusion-exclusion formula for the key expansion of the product of key polynomials from our bijection for which the unions in the image are not disjoint.

\subsection{Main result}
\label{sec:formula-main}

An elegant combinatorial proof of the Pieri rule for Schur polynomials \cite{Pie93} uses the following consequence of the  Robinson--Schensted--Knuth insertion algorithm \cite{Rob38,Sch61,Knu70} on semistandard Young tableaux.

\begin{theorem}[\cite{Sch61}]
  Given any partition $\lambda$ of length at most $n$, there exists a weight-preserving bijection
  \begin{equation}
    \SSYT_n(\lambda) \times \SSYT_n(1) \stackrel{\sim}{\longrightarrow} \bigsqcup_{\substack{  \mu \supset \lambda \\ |\mu/\lambda| = 1 }} \SSYT_n(\mu) .
    \label{e:RSK}
  \end{equation}
  \label{thm:RSK}
\end{theorem}

The image of the bijection in Theorem~\ref{thm:RSK} induced by RSK insertion is \emph{disjoint}. Therefore we may take generating polynomials to obtain Pieri's rule for multiplying Schur polynomials \cite{Pie93} as an immediate corollary.

\begin{theorem}[\cite{Pie93}]
  Given any partition $\lambda$ of length at most $n$, we have
  \begin{equation}
    s_{\lambda}(x_1,\ldots,x_n) \cdot s_{(1)}(x_1,\ldots,x_n) = \sum_{\substack{  \mu \supset \lambda \\ |\mu/\lambda| = 1 }} s_{\mu}(x_1,\ldots,x_n) .
    \label{e:Pieri}
  \end{equation}
  \label{thm:Pieri}
\end{theorem}

The partitions $\mu$ appearing on the right-hand side of Eq.~\eqref{e:RSK} can be described in terms of \emph{addable cells} for the Young diagram of the partition $\lambda$. A cell not in the Young diagram of $\lambda$ is an \emph{addable cell for $\lambda$} if the union of $\lambda$ and the cell is the Young diagram of a partition. 

\begin{figure}[ht]
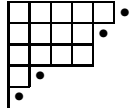

  \begin{displaymath}
      \vline\nulltab{
        \boxify{\ } & \boxify{\ } & \boxify{\ } & \boxify{\ } & \boxify{\ } & \bullet \\
        \boxify{\ } & \boxify{\ } & \boxify{\ } & \boxify{\ } & \bullet \\
        \boxify{\ } & \boxify{\ } & \boxify{\ } & \boxify{\ } \\
        \boxify{\ } & \bullet \\
        \bullet } 
  \end{displaymath}
  \caption{\label{fig:add-part}The four addable cells ($\bullet$) for the partition $(5,4,4,1)$.}
\end{figure}

\begin{example}\label{ex:add-part}
  The partition $\lambda = (5,4,4,1)$ has four addable cells illustrated in Fig.~\ref{fig:add-part}. Therefore Theorem~\ref{thm:RSK} ensures the existence of a bijection
  \begin{multline*}
    \SSYT_n(5,4,4,1) \times \SSYT_n(1) \stackrel{\sim}{\longrightarrow} \\ \SSYT_n(5,4,4,1,1) \sqcup \SSYT_n(5,4,4,2) \sqcup \SSYT_n(5,5,4,1) \sqcup \SSYT_n(6,4,4,1) ,
  \end{multline*}
  assuming $n \geq 5$. If $n=4$, then the first term on the right is omitted.
  Since the union on the right is \emph{disjoint}, taking generating polynomials immediately gives
  \[ s_{(5,4,4,1)} s_{(1)} = s_{(5,4,4,1,1)} + s_{(5,4,4,2)} + s_{(5,5,4,1)} + s_{(6,4,4,1)}. \]
\end{example}

The analogous statement for key polynomials is more subtle. Now the images we want must be key diagrams of weak compositions, but we also allow downward movement within the left swap order before adding a cell. 

Furthermore, the key polynomial product requires an additional parameter $k$ that governs the number of variables in the right term of the product. Given a positive integer $k\leq n$, let $\comp{e}_k$ denote the weak composition with a single nonzero part of value $1$ in position $k$, i.e. $\comp{e}_k = (0^{k-1},1,0^{n-k})$.

Our bijection, whose proof comprises Section~\ref{sec:bijection}, is stated succinctly as follows.

\begin{theorem}
  Given any weak composition $\comp{a}$ and any positive integer $k \leq n$, there exists a weight-preserving bijection
  \begin{equation}
    \KD(\comp{a}) \times \KD(\comp{e}_k) \stackrel{\sim}{\longrightarrow} \bigcup_{\substack{ \comp{b} \preceq \comp{a} \\ 1 \leq j \leq k }} \KD(\comp{b} + \comp{e}_j) ,
    \label{e:RSKD}
  \end{equation}
  where the addition of weak compositions on the right is coordinate-wise.
  \label{thm:RSKD}
\end{theorem}

Notice Theorem~\ref{thm:RSKD} generalizes Theorem~\ref{thm:RSK} in the following way.

\begin{corollary}
  For $\comp{a}$ a weakly increasing weak composition, we have a bijection
  \begin{equation}
    \KD(\comp{a}) \times \KD(\comp{e}_n) \stackrel{\sim}{\longrightarrow} \bigsqcup_{\substack{1 \leq j \leq n \\ a_j < a_{j+1} }} \KD(\comp{a} + \comp{e}_j) .
  \end{equation}
  In particular, by Proposition~\ref{prop:sort}, Theorem~\ref{thm:RSKD} implies Theorem~\ref{thm:RSK}.
  \label{cor:RSKD}
\end{corollary}

\begin{proof}
  We consider the right-hand side of Eq.~\eqref{e:RSKD} when $\comp{a}$ is weakly increasing and $k=n$. We claim first that
  \[ \bigcup_{\substack{ \comp{b} \preceq \comp{a} \\ 1 \leq j \leq n }} \KD(\comp{b} + \comp{e}_j) = \bigcup_{\substack{ 1 \leq j \leq n }} \KD(\comp{a} + \comp{e}_j) . \]
  To see this, suppose $\comp{b} \prec \comp{a}$. Then for any $j \leq n$, we have $\comp{b} + \comp{e}_j \prec \comp{a} + \comp{e}_i$ where $i \geq j$ is the largest index for which $b_j = a_i$. Therefore by Proposition~\ref{prop:lswap}, $\KD(\comp{b} + \comp{e}_j) \subset \KD(\comp{a} + \comp{e}_i)$. Therefore the pair $(\comp{b}, j)$ can be removed from the indexing set for all $j$, proving the claimed equality. Similarly, if $a_j = a_{j+1}$, then $\comp{a} + \comp{e}_j \prec \comp{a} + \comp{e}_{j+1}$, so again by Proposition~\ref{prop:lswap}, $\KD(\comp{a} + \comp{e}_j) \subset \KD(\comp{a} + \comp{e}_{j+1})$ and so the index $j$ may be removed from the union on the right. Since $\comp{a}$ is weakly increasing, the indexing is as stated. Moreover, the set of pairs $(\comp{a},j)$ for $j$ such that $a_j < a_{j+1}$ result in weak compositions $\comp{a} + \comp{e}_j$ with different column weights, and so the union must be disjoint. The result now follows from Theorem~\ref{thm:RSKD}.
\end{proof}

We make note of two key differences between Theorem~\ref{thm:RSKD} and Theorem~\ref{thm:RSK}. First, by Proposition~\ref{prop:lswap}, the union on the right-hand side of Eq.~\eqref{e:RSKD} has redundancy whenever $\comp{b}' + \comp{e}_{j'} \prec \comp{b} + \comp{e}_j$ for some $\comp{b}',\comp{b} \preceq \comp{a}$ and some $j',j \leq k$. Second, the image of the bijection in Theorem~\ref{thm:RSKD} is not, in general, \emph{disjoint}, even after accounting for this redundancy. Therefore when taking generating polynomials, we must use inclusion--exclusion to account for the nontrivial intersections.

\begin{example}\label{ex:41504}
  Consider the weak composition $(4,1,5,0,4)$ and the integer $k=3$. Theorem~\ref{thm:RSKD} states that there exists a bijection that maps pairs of Kohnert diagrams $(T,U) \in \KD(4,1,5,0,4) \times \KD(0,0,1,0,0)$ onto the set of Kohnert diagrams obtained from weak compositions of the form $\comp{b} + \comp{e}_j$, where $\comp{b} \in \lswap(4,1,5,0,4)$ and $j \in \{ 1,2,3 \}$. The 14 weak compositions in $\lswap(4,1,5,0,4)$ give up to 42 terms on the right hand side. However, after filtering those terms by Proposition~\ref{prop:lswap}, only four maximal terms remain. Thus, for this case, Theorem~\ref{thm:RSKD} is equivalent to the existence of a bijection
  \begin{multline*}
    \KD(4,1,5,0,4) \times \KD(0,0,1,0,0) \stackrel{\sim}{\longrightarrow} \\
    \KD(4,2,5,0,4) \cup \KD(4,5,5,0,1) \cup \KD(5,1,5,0,4) \cup \KD(4,1,6,0,4) .
  \end{multline*}
  Of these four sets, all have distinct column weights and so are pairwise disjoint \emph{except} for the middle two, which have nontrivial intersection given by
  \[ \KD(4,5,5,0,1) \cap \KD(5,1,5,0,4) = \KD(5,4,5,0,1) . \]
  Therefore taking generating polynomials gives the \emph{signed} expansion
  \[ \key_{(4,1,5,0,4)} \key_{(0,0,1,0,0)} = \key_{(4,2,5,0,4)} + \key_{(4,5,5,0,1)} + \key_{(5,1,5,0,4)} - \key_{(5,4,5,0,1)} + \key_{(4,1,6,0,4)} . \]
\end{example}

\subsection{Addable cells}
\label{sec:formula-addable}

As suggested by Corollary~\ref{cor:RSKD}, the maximal, in the sense of Proposition~\ref{prop:lswap}, weak compositions appearing on the right-hand side of Eq.~\eqref{e:RSKD} can be described in terms of \emph{addable cells} for key diagrams.

\begin{definition}
  Given a weak composition $\comp{a}$, the cell in row $r$ and column $c$ is an \emph{addable cell for $\comp{a}$} if $a_r < c$ and there exists $s \geq r$ such that $a_s = c-1$.
  \label{def:addable}
\end{definition}

For partitions, we may add a cell in row $r$ and column $c$ only if row $r$ has length $c-1$. For weak compositions, this condition is relaxed so that row $r$ has length \emph{at most} $c-1$ and some row $s$ weakly above row $r$ has length \emph{exactly} $c-1$ since the excess cells of row $s$ can be dropped down to row $r$ via Kohnert moves. 

\begin{example}\label{ex:addable}
  Consider $\comp{a} = (4,1,5,0,4)$ and $k=3$. The three addable cells corresponding to $\comp{a} + \comp{e}_j$ for $j=1,2,3$ are depicted on the left side of Fig.~\ref{fig:add-comp}. These are the cells for which row $r$ has length exactly $c-1$. In addition to this, the cell in row $2$ and column $5$ is also addable since row $5>2$ has length $4 = c-1$, as indicated in the right side of Fig.~\ref{fig:add-comp}. In order to realize this addition in the form $\comp{b} + \comp{e}_2$, we must drop the supporting cells in row $5$ down to row $2$, giving $\comp{b} = (4,4,5,0,1) \prec \comp{a}$. Notice these four additions are precisely the four maximal terms found in Ex.~\ref{ex:41504}.
\end{example}

\begin{figure}[ht]
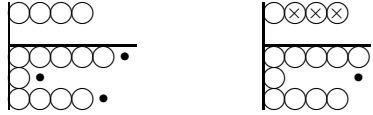

  \begin{displaymath}
    \vline\nulltab{ \circify{\ } & \circify{\ } & \circify{\ } & \circify{\ } \\
      \\\hline
      \circify{\ } & \circify{\ } & \circify{\ } & \circify{\ } & \circify{\ } & \bullet \\
      \circify{\ } & \bullet & \\
      \circify{\ } & \circify{\ } & \circify{\ } & \circify{\ } & \bullet }
    \hspace{6\cellsize}
    \vline\nulltab{ \circify{\ } & \circify{\times } & \circify{\times } & \circify{\times } \\
      \\\hline
      \circify{\ } & \circify{\ } & \circify{\ } & \circify{\ } & \circify{\ }  \\
      \circify{\ } & & & & \bullet \\
      \circify{\ } & \circify{\ } & \circify{\ } & \circify{\ } }
  \end{displaymath}
  \caption{\label{fig:add-comp}The four $3$-addable cells ($\bullet$) for the weak composition $(4,1,5,0,4)$, where for the right diagram we use $(4,4,5,0,1) \prec (4,1,5,0,4)$ to support the addable cell.}
\end{figure}

When the cell in row $r$ and column $c$ is an addable cell for $\comp{a}$, we also require a way to construct the weak composition $\comp{b} \preceq \comp{a}$ such that $\comp{b} + \comp{e}_r$ appears as a term on the right-hand side of Eq.~\eqref{e:RSKD} corresponding to this addition.

\begin{definition}
  Given a weak composition $\comp{a}$ and an addable position $(c,r)$, the \emph{maximal support composition for $\comp{a}$ at $(c,r)$} is the weak composition
  \begin{equation}
    \supp_{\comp{a}}^{(c,r)} = t_{r_0,r_1} \cdots t_{r_{q-1},r_{q}} \cdot \comp{a},
    \label{e:top-support}
  \end{equation}
  where $t_{i,j}$ is the transposition interchanging parts in positions $i$ and $j$, and $r = r_0 < r_1 < \cdots < r_q$ is the unique increasing sequence of row indices such that $a_{r_{i-1}} < a_{r_i}$ and if  $r_{i-1} < s < r_i$, then either $a_s \leq a_{r_{i-1}}$ or $a_s > a_{r_i}$.
  \label{def:top-support}
\end{definition}

\begin{example}\label{ex:support}
  Consider the weak composition $\comp{a} = (4,6,4,3,0,1,1,2,5,4)$ and $k=6$. To find addable cells in column $c = 5$, we look for a row index $r\leq k$ for which $a_r < c$, giving options of $r = 1,3,4,5,6$ depicted in Fig.~\ref{fig:support}. By Definition~\ref{def:addable}, rows $4,5,6$ require the existence of a row index $l>r$ for which $a_l = c-1$, and since $a_{10} = c-1$, this condition is met. Considering the cell in row $5$, the maximal support composition is constructed using the sequence $r_0 = 5 < 6 < 8 < 10$ so that the cells in each of these rows marked by $\circify{\times}$ drop down to row $5$ to create $\supp_{\comp{a}}^{(5,5)} = (4,6,4,3,4,1,0,1,5,2)$. The maximal support compositions for the other addable cells in column $c$ are similarly marked in Fig.~\ref{fig:support}.
\end{example}

\begin{figure}[ht]
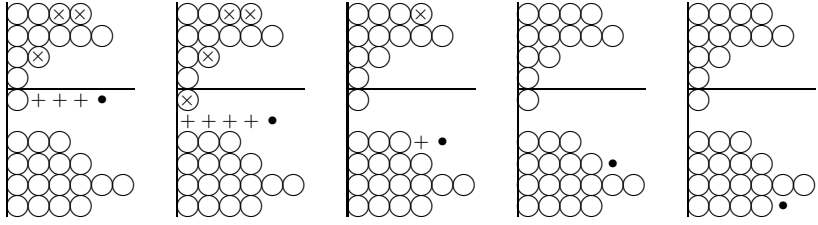

  \begin{displaymath}
    \arraycolsep=\cellsize
    \begin{array}{ccccc}
    \vline\nulltab{\circify{\ } & \circify{\ } & \circify{\times} & \circify{\times} \\
      \circify{\ } & \circify{\ } & \circify{\ } & \circify{\ } & \circify{\ } \\
      \circify{\ } & \circify{\times } \\
      \circify{\ } \\\hline
      \circify{\ } & + & + & + & \bullet \\
      \\
      \circify{\ } & \circify{\ } & \circify{\ } \\
      \circify{\ } & \circify{\ } & \circify{\ } & \circify{\ } \\
      \circify{\ } & \circify{\ } & \circify{\ } & \circify{\ } & \circify{\ }  & \circify{\ } \\
      \circify{\ } & \circify{\ } & \circify{\ } & \circify{\ } 
    } &
    \vline\nulltab{\circify{\ } & \circify{\ } & \circify{\times } & \circify{\times} \\
      \circify{\ } & \circify{\ } & \circify{\ } & \circify{\ } & \circify{\ } \\
      \circify{\ } & \circify{\times } \\
      \circify{\ } \\\hline
      \circify{\times } \\
      + & + & + & + & \bullet\\
      \circify{\ } & \circify{\ } & \circify{\ } \\
      \circify{\ } & \circify{\ } & \circify{\ } & \circify{\ } \\
      \circify{\ } & \circify{\ } & \circify{\ } & \circify{\ } & \circify{\ }  & \circify{\ } \\
      \circify{\ } & \circify{\ } & \circify{\ } & \circify{\ } 
    } &
    \vline\nulltab{\circify{\ } & \circify{\ } & \circify{\ } & \circify{\times} \\
      \circify{\ } & \circify{\ } & \circify{\ } & \circify{\ } & \circify{\ } \\
      \circify{\ } & \circify{\ } \\
      \circify{\ } \\\hline
      \circify{\ } \\
      & & & \\
      \circify{\ } & \circify{\ } & \circify{\ } & + & \bullet \\
      \circify{\ } & \circify{\ } & \circify{\ } & \circify{\ } \\
      \circify{\ } & \circify{\ } & \circify{\ } & \circify{\ } & \circify{\ }  & \circify{\ } \\
      \circify{\ } & \circify{\ } & \circify{\ } & \circify{\ } 
    } &
    \vline\nulltab{\circify{\ } & \circify{\ } & \circify{\ } & \circify{\ }  \\
      \circify{\ } & \circify{\ } & \circify{\ } & \circify{\ } & \circify{\ } \\
      \circify{\ } & \circify{\ } \\
      \circify{\ } \\\hline
      \circify{\ } \\
      \\
      \circify{\ } & \circify{\ } & \circify{\ } \\
      \circify{\ } & \circify{\ } & \circify{\ } & \circify{\ } & \bullet \\
      \circify{\ } & \circify{\ } & \circify{\ } & \circify{\ } & \circify{\ }  & \circify{\ } \\
      \circify{\ } & \circify{\ } & \circify{\ } & \circify{\ } 
    } &
    \vline\nulltab{\circify{\ } & \circify{\ } & \circify{\ } & \circify{\ } \\
      \circify{\ } & \circify{\ } & \circify{\ } & \circify{\ } & \circify{\ } \\
      \circify{\ } & \circify{\ } \\
      \circify{\ } \\\hline
      \circify{\ } \\
      \\
      \circify{\ } & \circify{\ } & \circify{\ } \\
      \circify{\ } & \circify{\ } & \circify{\ } & \circify{\ } \\
      \circify{\ } & \circify{\ } & \circify{\ } & \circify{\ } & \circify{\ }  & \circify{\ } \\
      \circify{\ } & \circify{\ } & \circify{\ } & \circify{\ } & \bullet 
    }
    \end{array}
  \end{displaymath}
  \caption{\label{fig:support}The five addable cells $(\bullet)$ for $(4,6,4,3,0,1,1,2,5,4)$ in column $5$. Here the marked cells $(\otimes)$ drop to positions $(+)$ in row $r$ in creating the maximal support composition. The first, third and fourth are $6$-addable but the second and fifth are not.}
\end{figure}

We can now re-characterize the right-hand side of Eq.~\eqref{e:RSKD} in terms of addable cells, making our first reduction in the terms in the union.

\begin{lemma}
  Given a weak composition $\comp{a}$ and positive integer $k \leq n$, we have
  \begin{equation}
    \bigcup_{\substack{ \comp{b} \preceq \comp{a} \\ 1 \leq j \leq k }} \KD(\comp{b} + \comp{e}_j)  = 
    \bigcup_{\substack{ 1 \leq j \leq k \\ (c,j) \text{ addable for } \comp{a} }} \KD(\supp_{\comp{a}}^{(c,j)} + \comp{e}_j) .
    \label{e:addable}
  \end{equation}
  \label{lem:addable}
\end{lemma}

\begin{proof}
  By Definition~\ref{def:top-support}, we have $\supp_{\comp{a}}^{(c,j)} \preceq \comp{a}$, and so we have containment of the right-hand side of Eq.~\eqref{e:addable} in the left-hand side. For the other direction, suppose $\comp{b} \preceq \comp{a}$ and $j \leq k$ satisfies $b_j = c-1$. Then by the choice of $r_i$'s in Definition~\ref{def:top-support}, $\comp{b} \preceq \supp_{\comp{a}}^{(c,j)}$, and so $\comp{b} + \comp{e}_j \preceq \supp_{\comp{a}}^{(c,j)} + \comp{e}_j$. Therefore, by Proposition~\ref{prop:lswap}, $\KD(\comp{b} + \comp{e}_j) \subseteq \KD(\supp_{\comp{a}}^{(c,j)} + \comp{e}_j)$, and so we may eliminate these terms from the union on the left, giving the desired equality.
\end{proof}

Lemma~\ref{lem:addable} reduces the index set for the union in Eq.~\eqref{e:RSKD} to addable cells for $\comp{a}$. In order to characterize the maximal terms for the union, we strengthen Definition~\ref{def:addable} to the following notion of \emph{$k$-addable cells}.

\begin{definition}
  Given a weak composition $\comp{a}$ and positive integer $k$, the cell in row $r \leq k$ and column $c$ is a \emph{$k$-addable cell for $\comp{a}$} if 
  \begin{enumerate}
  \item $a_r < c$ and if $a_r < c-1$, then there exists some $l>k$ such that $a_l = c-1$;
  \item for all $r < i \leq k$, either $a_i < a_r$ or $a_i \geq c$.
  \end{enumerate}
  \label{def:k-addable}
\end{definition}

Notice Definition~\ref{def:k-addable}(1) is stronger than Definition~\ref{def:addable} since in this case we require the supporting row to be above row $k$, not just above row $r$. Taken together, the conditions of Definition~\ref{def:k-addable} imply a $k$-addable cell is an addable cell such that none of the row indices in Definition~\ref{def:top-support} except the first lies below row $k$.

\begin{example}\label{ex:k-addable}
  Continuing with Ex.~\ref{ex:support}, Definition~\ref{def:k-addable}(1) is met for all rows since $a_{10} = c-1$. Checking Definition~\ref{def:k-addable}(2) eliminates rows $1$ and $5$, and so the only $k$-addable cells in column $c$ occur in rows $3,4,6$. Notice above each eliminated row sits another that is $k$-addable.
\end{example}

\begin{lemma}
  Given a weak composition $\comp{a}$ and a positive integer $k$, if $c$ is a column index for which there exists some row index $r\leq k$ such that $(c,r)$ is an addable cell for $\comp{a}$, then there exists a row index $r \leq s \leq k$ such that $(c,s)$ is $k$-addable for $\comp{a}$ and $\supp_{\comp{a}}^{(c,r)} \preceq \supp_{\comp{a}}^{(c,s)}$. 
  \label{lem:k-addable}
\end{lemma}

\begin{proof}
  If $(c,r)$ is $k$-addable for $\comp{a}$, then we may take $s=r$. Otherwise, we consider three cases based on how $(c,r)$ fails to be $k$-addable for $\comp{a}$.

  Suppose Definition~\ref{def:k-addable}(1) fails for $(c,r)$. Since $(c,r)$ is addable, there exists a maximal row index $s$ such that $r<s\leq k$ and $a_s = c-1$. Then $(c,s)$ is $k$-addable, satisfying Definition~\ref{def:k-addable}(1) since $a_s = c-1$ and Definition~\ref{def:k-addable}(2) since, by choice of $s$, no row index $i>s$ has $a_i = c-1$. Moreover, the final index $t$ appearing in Definition~\ref{def:top-support} for $(c,r)$ satisfies $t \leq s$ and $a_t = a_s$, ensuring we have $\supp_{\comp{a}}^{(c,r)} \preceq \supp_{\comp{a}}^{(c,t)} \preceq \supp_{\comp{a}}^{(c,s)}$ as desired.

  Suppose Definition~\ref{def:k-addable}(1) holds but Definition~\ref{def:k-addable}(2) fails for $(c,r)$ with some index $r < i \leq k$ for which $a_r < a_i \leq c-1$. Then there exists a row index $j \leq k$ appearing in Definition~\ref{def:top-support} for $(c,r)$, and we may take $s'$ to be the maximum such index. Define $s$ to be the largest row index such that $s \leq k$ and $a_s = a_{s'}$. Then $(c,s)$ is $k$-addable, satisfying Definition~\ref{def:k-addable}(1) since $(c,r)$ does and Definition~\ref{def:k-addable}(2) since by Definition~\ref{def:top-support} for $(c,r)$, no row index $i>s$ has $a_s < a_i \leq c-1$ and by choice of $s$ no row index $i>s$ has $a_s = a_i$. Moreover, since $s' \leq s$ and $a_{s'} = a_s$, we have $\supp_{\comp{a}}^{(c,r)} \preceq \supp_{\comp{a}}^{(c,s')} \preceq \supp_{\comp{a}}^{(c,s)}$ as desired.

  Suppose Definition~\ref{def:k-addable}(1) holds but Definition~\ref{def:k-addable}(2) fails for $(c,r)$ only for some index $r < i \leq k$ for which $a_i = a_r$. We may take $s$ to be the maximal index such that $r<s\leq k$ and $a_s = a_r$. By Definition~\ref{def:top-support}, the sequences of row indices for the cells $(c,r)$ and $(c,s)$ differ only for the first index $r_0$, and no other index is weakly less than $k$. In particular, $(c,s)$ is $k$-addable, satisfying Definition~\ref{def:k-addable}(1) since $(c,r)$ does and Definition~\ref{def:k-addable}(2) since by Definition~\ref{def:top-support} for $(c,r)$, no row index $i>s$ has $a_s < a_i \leq c-1$ and by choice of $s$ no row index $i>s$ has $a_s = a_i$. Once again, since $r \leq s$ and $a_{r} = a_s$, we have $\supp_{\comp{a}}^{(c,r)} \preceq \supp_{\comp{a}}^{(c,s)}$ as desired.
\end{proof}

Lemma~\ref{lem:k-addable} allows us to reduce the indexing set on the right-hand side of Eq.~\eqref{e:addable} still further, and indeed this is the most we can reduce it.


\begin{lemma}
  Given a weak composition $\comp{a}$ and a positive integer $k$, if both $(c,r)$ and $(c,s)$ are $k$-addable cells for $\comp{a}$ with $r<s$, then $a_r > a_s$ and $\supp_{\comp{a}}^{(c,r)}+\comp{e}_r$ and $\supp_{\comp{a}}^{(c,s)}+\comp{e}_s$ are incomparable in left swap order.
  \label{lem:incomparable}
\end{lemma}

\begin{proof}
  Since both cells $(c,r)$ and $(c,s)$ are addable, we must have $a_r, a_s < c$. Since $(c,r)$ is $k$-addable, by Definition~\ref{def:k-addable}(2) we must have $a_r > a_s$. The first $r-1$ parts of $\supp_{\comp{a}}^{(c,r)}+\comp{e}_r$ and $\supp_{\comp{a}}^{(c,s)}+\comp{e}_s$ must agree, and the $r$th part of the former is strictly larger, ensuring it cannot be above the latter in left swap order by Proposition~\ref{prop:lswap}. On the other hand, there are $a_r - a_s>0$ fewer cells above row $k$ in the latter than in the former, ensuring the latter cannot be above the former in left swap order. Thus the two are incomparable. 
\end{proof}

We may now state the minimal indexing set for the union in Eq.~\eqref{e:RSKD}.

\begin{theorem}
  For a weak composition $\comp{a}$ and positive integer $k$, we have
  \begin{equation}
    \bigcup_{\substack{ \comp{b} \preceq \comp{a} \\ 1 \leq j \leq k }} \KD(\comp{b} + \comp{e}_j) = 
    \bigcup_{\substack{ 1 \leq j \leq k \\ (c,j) \text{ $k$-addable for } \comp{a} }} \KD(\supp_{\comp{a}}^{(c,j)} + \comp{e}_j) ,
    \label{e:k-addable}
  \end{equation}
  where no term on the right is strictly contained in another.
  \label{thm:k-addable}
\end{theorem}

\begin{proof}
  By Lemma~\ref{lem:k-addable}, if $(c,r)$ is an addable cell for $\comp{a}$ that it not $k$-addable, then there exists a $k$-addable cell $(c,s)$ such that $\supp_{\comp{a}}^{(c,r)} \preceq \supp_{\comp{a}}^{(c,s)}$. Since $r<s$, we also have $\supp_{\comp{a}}^{(c,r)} + \comp{e}_r \preceq \supp_{\comp{a}}^{(c,s)} + \comp{e}_s$. Therefore, by Proposition~\ref{prop:lswap}, the set of Kohnert diagrams of the former is contained in the set of Kohnert diagrams of the latter, so the addable cell $(c,r)$ may be removed from the indexing set on the right side of Eq.~\eqref{e:addable}. Thus the equality follows from Lemma~\ref{lem:addable}. 

  To see that the terms on the right side of Eq.~\eqref{e:k-addable} are pairwise not contained in one another, note that for $c' \neq c$, the column weights of $\supp_{\comp{a}}^{(c,j')}+\comp{e}_{j'}$ and $\supp_{\comp{a}}^{(c,j)}+\comp{e}_j$ are different, regardless of the values for $j',j$, and so the sets of Kohnert diagrams in this case are disjoint. For cells added within the same column, Lemma~\ref{lem:incomparable} ensures there is no pairwise containment.
\end{proof}

\subsection{Drop sets}
\label{sec:formula-drop}

The image of the bijection in Theorem~\ref{thm:RSK} induced by RSK insertion is \emph{disjoint}. Therefore we may take generating polynomials to obtain Pieri's rule for multiplying Schur polynomials as an immediate corollary.

In contrast with this, the image of the bijection in Theorem~\ref{thm:RSKD} is not, in general, disjoint. Therefore when taking generating polynomials to obtain our key analog of Pieri's rule for multiplying key polynomials, we must use inclusion--exclusion to account for the nontrivial intersections.

As remarked in the proof of Theorem~\ref{thm:k-addable}, if $c' \neq c$, then the column weights of $\supp_{\comp{a}}^{(c,j')}+\comp{e}_{j'}$ and $\supp_{\comp{a}}^{(c,j)}+\comp{e}_j$ are different, regardless of the values for $j',j$, and so the sets of Kohnert diagrams in this case are disjoint. Therefore we focus our attention on cells added within a given column. 

\begin{definition}
  Given a weak composition $\comp{a}$, a positive integer $k$, and a column index $c$, the \emph{$k$-addable row set for $\comp{a}$ in column $c$} is given by
  \begin{equation}
    \Row_{\comp{a},k}^c = \{ r \leq k \mid (c,r) \text{ is $k$-addable for $\comp{a}$} \} .
    \label{e:col-set}
  \end{equation}
  We say that $c$ is a \emph{$k$-addable column for $\comp{a}$} whenever the set $\Row_{\comp{a},k}^c$ is nonempty.
  \label{def:col-set}
\end{definition}

By Lemma~\ref{lem:incomparable}, if we take elements of $\Row_{\comp{a},k}^c$ as increasing, $r_1 < \cdots < r_p$, then the corresponding parts of $\comp{a}$ are decreasing, $a_{r_1} > \cdots > a_{r_p}$. 

Similar to Definition~\ref{def:top-support}, we can construct the weak compositions that index the intersections of the sets $\KD(\supp_{\comp{a}}^{(c,r)} + \comp{e}_r)$ using left swap order.

\begin{definition}
  Let $\comp{a}$ be a weak composition, $k$ a positive integer, and $c$ a $k$-addable column for $\comp{a}$. Given a nonempty subset $R \subseteq \Row_{\comp{a},k}^c$, the \emph{maximal drop composition for $\comp{a}$ in column column $c$ at rows $R$} is the weak composition
  \begin{equation}
    \drop_{\comp{a}}^{(c,R)} = t_{r_{-p},r_{-p+1}} \cdots t_{r_{-1},r_{0}} \cdot \supp_{\comp{a}}^{(c,r_0)},
    \label{e:drop}
  \end{equation}
  where $R = \{r_{-p} < \cdots < r_{-1} < r_{0}\}$.
  \label{def:drop}
\end{definition}

In particular, for singleton sets we have the equivalence $\drop_{\comp{a}}^{(c,\{r\})} = \supp_{\comp{a}}^{(c,r)}$.

\begin{example}\label{ex:drop}
  Consider again the weak composition $\comp{a} = (4,6,4,3,0,1,1,2,5,4)$ and $k=6$. In Ex.~\ref{ex:k-addable} we found three $k$-addable cells for $\comp{a}$ in column $5$ giving $\Row_{\comp{a},k}^c = \{3,4,6\}$. There are three doubleton subsets as well as the entire set to consider for $R$, as indicated in Fig.~\ref{fig:drop}. Notice the cells that fall from above row $k$ are the same as for the maximal support composition of the highest row of $R$, but now not all cells fall to the same row.  
\end{example}

\begin{figure}[ht]
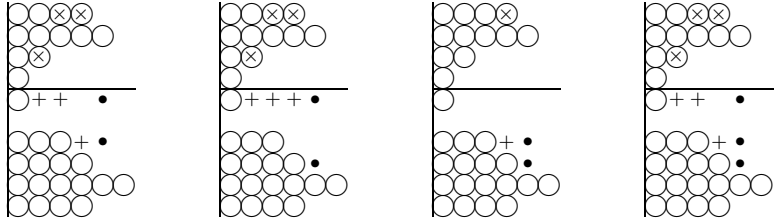

  \begin{displaymath}
    \arraycolsep=2\cellsize
    \begin{array}{cccc}
    \vline\nulltab{\circify{\ } & \circify{\ } & \circify{\times} & \circify{\times} \\
      \circify{\ } & \circify{\ } & \circify{\ } & \circify{\ } & \circify{\ } \\
      \circify{\ } & \circify{\times } \\
      \circify{\ } \\\hline
      \circify{\ } & + & + & & \bullet \\
      \\
      \circify{\ } & \circify{\ } & \circify{\ } & + & \bullet \\
      \circify{\ } & \circify{\ } & \circify{\ } & \circify{\ } \\
      \circify{\ } & \circify{\ } & \circify{\ } & \circify{\ } & \circify{\ }  & \circify{\ } \\
      \circify{\ } & \circify{\ } & \circify{\ } & \circify{\ } 
    } &
    \vline\nulltab{\circify{\ } & \circify{\ } & \circify{\times } & \circify{\times} \\
      \circify{\ } & \circify{\ } & \circify{\ } & \circify{\ } & \circify{\ } \\
      \circify{\ } & \circify{\times } \\
      \circify{\ } \\\hline
      \circify{\ } & + & + & + & \bullet \\
      \\
      \circify{\ } & \circify{\ } & \circify{\ } & \\
      \circify{\ } & \circify{\ } & \circify{\ } & \circify{\ } & \bullet\\
      \circify{\ } & \circify{\ } & \circify{\ } & \circify{\ } & \circify{\ }  & \circify{\ } \\
      \circify{\ } & \circify{\ } & \circify{\ } & \circify{\ } 
    } &
    \vline\nulltab{\circify{\ } & \circify{\ } & \circify{\ } & \circify{\times }  \\
      \circify{\ } & \circify{\ } & \circify{\ } & \circify{\ } & \circify{\ } \\
      \circify{\ } & \circify{\ } \\
      \circify{\ } \\\hline
      \circify{\ } \\
      \\
      \circify{\ } & \circify{\ } & \circify{\ } & + & \bullet \\
      \circify{\ } & \circify{\ } & \circify{\ } & \circify{\ } & \bullet \\
      \circify{\ } & \circify{\ } & \circify{\ } & \circify{\ } & \circify{\ }  & \circify{\ } \\
      \circify{\ } & \circify{\ } & \circify{\ } & \circify{\ } 
    } &
    \vline\nulltab{\circify{\ } & \circify{\ } & \circify{\times } & \circify{\times }  \\
      \circify{\ } & \circify{\ } & \circify{\ } & \circify{\ } & \circify{\ } \\
      \circify{\ } & \circify{\times } \\
      \circify{\ } \\\hline
      \circify{\ } & + & + & & \bullet \\
      \\
      \circify{\ } & \circify{\ } & \circify{\ } & + & \bullet \\
      \circify{\ } & \circify{\ } & \circify{\ } & \circify{\ } & \bullet \\
      \circify{\ } & \circify{\ } & \circify{\ } & \circify{\ } & \circify{\ }  & \circify{\ } \\
      \circify{\ } & \circify{\ } & \circify{\ } & \circify{\ } 
    } 
    \end{array}
  \end{displaymath}
  \caption{\label{fig:drop}The four nonempty, non-singleton subsets of rows $(\bullet)$ of the $6$-addable cells for $(4,6,4,3,0,1,1,2,5,4)$ in column $5$. Here marked cells $(\otimes)$ will drop down to the indicated position $(+)$ below in creating the maximal drop composition.}
\end{figure}

\begin{lemma}
  Let $\comp{a}$ be a weak composition, $k$ a positive integer, and $c$ a $k$-addable column for $\comp{a}$. Given a nonempty subset $R \subseteq \Row_{\comp{a},k}^c$ and a row index $s\in\Row_{\comp{a},k}^c$ such that $s > \max{R}$, we have
  \begin{multline}
    \KD\left( \supp_{\comp{a}}^{(c,s)} + \comp{e}_s \right) \cap
    \KD\left( \drop_{\comp{a}}^{(c,R)} + \comp{e}_{\min(R)} \right) = \\
    \KD\left( \drop_{\comp{a}}^{(c,R \cup \{s\})} + \comp{e}_{\min(R)} \right) .
    \label{e:intersect}
  \end{multline}
  \label{lem:intersect}
\end{lemma}

\begin{proof}
  It is immediate from Definition~\ref{def:drop} that $\drop_{\comp{a}}^{(c,R)} \preceq \supp_{\comp{a}}^{(c,\max(R))}$ for any nonempty $R \subseteq \Row_{\comp{a},k}^c$. Since $s = \max(R\cup\{s\}) > \min(R)$, it follows from Proposition~\ref{prop:lswap} that
  \[ \KD\left( \drop_{\comp{a}}^{(c,R \cup \{s\})} + \comp{e}_{\min(R)} \right) \subseteq \KD\left( \supp_{\comp{a}}^{(c,s)} + \comp{e}_s \right) .\]
  Similarly, adding a new maximum row index to $R$ expands the set of cells above row $k$ that are dropped but does not affect the resulting positions of cells from $\max(R)$ down, so $\drop_{\comp{a}}^{(c,R\cup\{s\})} \preceq \drop_{\comp{a}}^{(c,R)}$. Since $\min(R\cup\{s\}) = \min(R)$, by Proposition~\ref{prop:lswap} again we have
  \[ \KD\left( \drop_{\comp{a}}^{(c,R \cup \{s\})} + \comp{e}_{\min(R)} \right) \subseteq \KD\left( \drop_{\comp{a}}^{(c,R)} + \comp{e}_{\min(R)} \right) .\]
  Therefore the right-hand side of Eq.~\eqref{e:intersect} is contained in the left-hand side.

  For brevity, let $\comp{b} = \drop_{\comp{a}}^{(c,R \cup \{s\})} + \comp{e}_{\min(R)}$. Given a generic Kohnert diagram $T$ such that $\cwt(T) = \cwt(\comp{b})$ but $T \not\in\KD(\comp{b})$, we aim to show $T$ is not contained in the left-hand side of Eq.~\eqref{e:intersect}. By Lemma~\ref{lem:3.7}, the thread weight of $T$ also satisfies $\kd_{\thread(T)} \not\in \KD(\comp{b})$, and so we may assume $T = \kd_{\thread(T)}$. Let $j$ denote the largest index such that no weak composition $\comp{c} \preceq \comp{b}$ satisfies $c_i = \wt(T)_i$ for all $i \geq j$. If $j>s$, then since $\drop_{\comp{a}}^{(c,R\cup\{s\})}$ and $\supp_{\comp{a}}^{(c,s)}$ have the same values beyond index $s$, we must have $\wt(T) \not\preceq \supp_{\comp{a}}^{(c,s)} + \comp{e}_s$, and so $T$ does not appear in the set on the left-hand side of Eq.~\eqref{e:intersect}. If $j \leq s$, then since $\drop_{\comp{a}}^{(c,R\cup\{s\})}$ and $\drop_{\comp{a}}^{(c,R)}$ have the same values before index $s$, we must have $\wt(T) \not\preceq \drop_{\comp{a}}^{(c,R)} + \comp{e}_{\min{R}}$, and so again $T$ does not appear in the set on the left-hand side of Eq.~\eqref{e:intersect}. 
\end{proof}

\begin{example}
  Beginning with Theorem~\ref{thm:k-addable} our running example of the weak composition $\comp{a} = (4,6,4,3,0,1,1,2,5,4)$ with $k=6$ in column $c=5$ gives
  \begin{eqnarray*}
    \bigcup_{\substack{ r \in \Row_{\comp{a},k}^{c} }} \KD(\supp_{\comp{a}}^{(c,r)} + \comp{e}_r) & = &
    \KD(4,6,4,3,0,5,1,1,5,2) \\[-12pt] & & \cup \KD(4,6,4,5,0,1,1,2,5,1) \\ & & \cup \KD(4,6,5,3,0,1,1,2,5,4).
  \end{eqnarray*}
  Taking the generating polynomial by iteratively applying Lemma~\ref{lem:intersect} gives
  \begin{multline*}
    \key_{(4,6,4,3,0,5,1,1,5,2)} + \key_{(4,6,4,5,0,1,1,2,5,1)} + \key_{(4,6,5,3,0,1,1,2,5,4)} \\
    - \key_{(4,6,4,5,0,3,1,1,5,2)} - \key_{(4,6,5,3,0,4,1,1,5,2)} - \key_{(4,6,5,4,0,1,1,2,5,3)} \\
    +  \key_{(4,6,5,4,0,3,1,1,5,2)}
  \end{multline*}
\end{example}

We finally have all the ingredients needed to state the key analog of Pieri's rule.

\begin{theorem}
  Given a weak composition $\comp{a}$ and positive integer $k$, we have
  \begin{equation}
    \key_{\comp{a}} \cdot s_{(1)}(x_1,\ldots,x_k) = \sum_{\substack{ c \text{ $k$-addable for $\comp{a}$} \\ \varnothing \neq R \subseteq \Row_{\comp{a},k}^{c}}} (-1)^{\#R-1} \key_{\drop_{\comp{a}}^{(c,R)}+\comp{e}_{\min(R)}} .
    \label{e:monkey}
  \end{equation}
  Moreover, the terms on the right-hand side are pairwise distinct.
  \label{thm:monkey}
\end{theorem}

\begin{proof}
  Combining Theorems~\ref{thm:RSKD} and \ref{thm:k-addable}, we have a weight-preserving bijection
  \[ \KD(\comp{a}) \times \KD(\comp{e}_k) \stackrel{\sim}{\longrightarrow}
  \bigcup_{\substack{ 1 \leq j \leq k \\ (c,j) \text{ $k$-addable for } \comp{a} }} \KD(\supp_{\comp{a}}^{(c,j)} + \comp{e}_j). \]
  The generating polynomial on the left-hand side is $\key_{\comp{a}} \cdot s_{(1)}(x_1,\ldots,x_k)$. The generating polynomial on the right-hand side can be computed by first noting the sets are disjoint for different columns $c$, then using Lemma~\ref{lem:intersect} iteratively to compute intersections. The equality now follows from the inclusion--exclusion formula for intersections of sets.
\end{proof}

In particular, notice the right hand side of Eq.~\eqref{e:monkey} is nonnegative if and only if the $k$-addable row set for $\comp{a}$ for each $k$-addable column $c$ is a singleton. 

%
\section{Key bijections}
%
\label{sec:bijection}

We now prove the key Pieri rule for certain extremal cases via a reversible insertion of a single box into a generic Kohnert diagram. This gives an explicit weight-preserving bijection as asserted in Theorem~\ref{thm:RSKD}. To ease notation, given a weak composition $\comp{a}$ and a positive integer $k$, we denote the \emph{target space} of the bijection by $\KDs(\comp{a},k)$, that is
\begin{equation}
    \KDs(\comp{a},k) = \bigcup_{\substack{ \comp{b} \preceq \comp{a} \\ 1 \leq j \leq k }} \KD(\comp{b} + \comp{e}_j).
\end{equation}
In Section~\ref{sec:bijection-bottom}, we give a simple insertion algorithm for the case $k=1$, developing along the way several tools for understanding Kohnert diagrams that are essential for all cases. In Section~\ref{sec:bijection-rectify}, we review a generalization of the classical RSK insertion algorithm on tableaux to an insertion algorithm on diagrams \cite{Ass-K,AG}. In Section~\ref{sec:bijection-top}, we use rectification to construct a more subtle insertion algorithm for the case $k \geq \ell(\comp{a})$, where $\ell(\comp{a})$ denotes the largest index $i$ for which $a_i > 0$. 

\subsection{Bottom insertion}
\label{sec:bijection-bottom}

The bijection of Theorem~\ref{thm:RSKD} for the case $k=1$ is simple to state, though the proof requires several additional tools.

\begin{definition}
  Let $T$ be a generic Kohnert diagram, and let 
  \[ c = \min\{i \mid (i,1) \notin T\} \]
  be the column of the left-most empty position of $T$ in the first row. Then the \emph{bottom insertion map} $\ins_1$ sends $T$ to the diagram
  \begin{equation}
    \ins_1(T)= T \sqcup \{(c,1)\} .
    \label{e:insert-bot}  
  \end{equation}
  \label{def:insert-bot}  
\end{definition}

\begin{example}
  Consider the weak composition $\comp{a} = (0,3,2)$ and $k=1$. Refining the right-hand side of Theorem~\ref{thm:RSKD} using Theorem~\ref{thm:k-addable}, we expect a bijection
  \[ \KD{(0,3,2)} \times \KD{(1,0,0)} \stackrel{\sim}{\rightarrow} \KD{(1,3,2)} \cup \KD{(3,3,0)} \cup \KD{(4,0,2)} . \]
  Indeed, Fig.~\ref{fig:insert-bot} shows the images of the Kohnert diagrams in $\KD(0,3,2)$ (see Fig.~\ref{fig:kohnert}) under the bottom insertion map $\ins_1$, which is precisely the union on the right.
  \label{ex:bottom}
\end{example}

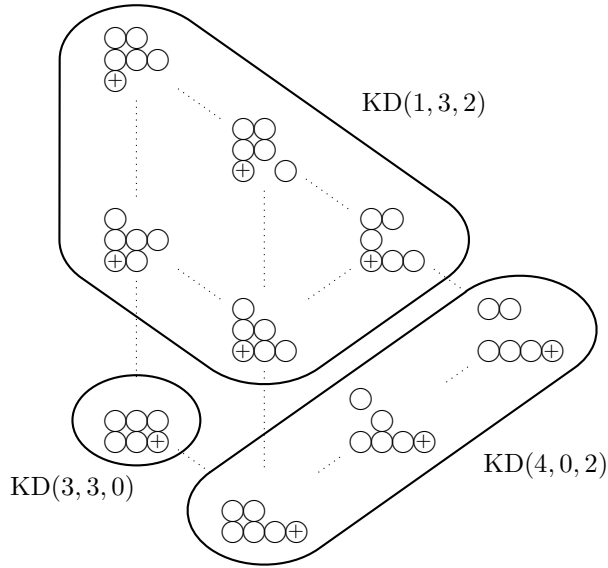
\begin{figure}[ht]
  \begin{center}
    \begin{tikzpicture}[xscale=1.7,yscale=1.2]
      \node at (0,5) (A) {$\vline\cirtab{  \ &  \ \\  \ &  \ &  \ \\ + & }$};
      \node at (1,4) (B) {$\vline\cirtab{  \ &  \ \\  \ &  \ \\ + & & \ }$};
      \node at (2,3) (C) {$\vline\cirtab{  \ &  \ \\  \ \\ + &  \ &  \ }$};
      \node at (3,2) (D) {$\vline\cirtab{  \ &  \ \\ \\  \ &  \ &  \ & + }$};
      \node at (1,2) (E) {$\vline\cirtab{  \ \\  \ &  \ \\ + &  \ &  \ }$};
      \node at (2,1) (F) {$\vline\cirtab{  \ \\ &  \ \\  \ &  \ &  \ & + }$};
      \node at (1,0) (G) {$\vline\cirtab{ \\  \ &  \ \\  \ &  \ &  \ & + }$};
      \node at (0,3) (H) {$\vline\cirtab{  \ \\  \ &  \ &  \ \\ + &  \ }$};
      \node at (0,1) (I) {$\vline\cirtab{ \\  \ &  \ &  \ \\  \ &  \ & + }$};
      \draw[thick] \convexpath{D,G}{6mm};
      \draw[thick] \convexpath{H,A,C,E}{6mm};
      \draw[thick] (0,1) circle (5mm);
      \node at (-0.5,0.25) (KD330) {$\KD(3,3,0)$};
      \node at (3.2,0.5) (KD402) {$\KD(4,0,2)$};
      \node at (2.25,4.5) (KD132) {$\KD(1,3,2)$};
      \draw[dotted] (A) -- (H) ;
      \draw[dotted] (A) -- (B) ;
      \draw[dotted] (H) -- (I) ;
      \draw[dotted] (H) -- (E) ;
      \draw[dotted] (B) -- (E) ;
      \draw[dotted] (B) -- (C) ;
      \draw[dotted] (I) -- (G) ;
      \draw[dotted] (E) -- (G) ;
      \draw[dotted] (C) -- (E) ;
      \draw[dotted] (C) -- (D) ;
      \draw[dotted] (D) -- (F) ;
      \draw[dotted] (F) -- (G) ;      
    \end{tikzpicture}
    \caption{\label{fig:insert-bot}The Kohnert diagrams for $(0,3,2)$ along with the appended cell $(\oplus)$ under the bottom insertion map $\ins_1$.}
  \end{center}
\end{figure}

In order to show $\ins_1(T)$ is a generic Kohnert diagram, we reformulate the criterion given in Proposition~\ref{prop:lemma2.2} by generalizing the thread decomposition  given in Definition~\ref{def:thread} to \emph{matching sequences} defined as follows.

\begin{definition}
  Let $C$ (respectively, $D$) be a diagram consisting of cells in some column $i$ (respectively, $i+1$). A \emph{matching} from $D$ to $C$ is a directed graph with vertex set $D \sqcup C$ such that for every $x\in D$ and every $y\in C$ we have
  \begin{enumerate}
  \item $x$ has out-degree $1$ and in-degree $0$;
  \item $y$ has out-degree $0$ and in-degree at most $1$;
  \item if $y \leftarrow x$, then the row index of $y$ is weakly greater than that of $x$.
  \end{enumerate}
  For a given matching $M$, we say a cell $x \in D$ \emph{matches to} a cell $y \in C$, written $(y \leftarrow x)$ or $M(x)=y$, whenever $M$ has a directed edge from $x$ to $y$.
  \label{def:matching}
\end{definition}

\begin{definition}
  For $T$ an arbitrary diagram, a \emph{matching sequence on $T$} is a directed graph $M$ with vertex set the cells of $T$ such that for every pair of adjacent columns $i$ and $i+1$ for which column $i+1$ is nonempty in $T$, the restriction of $M$ to the cells of $T$ in columns $i$ and $i+1$ is a matching. 
  \label{def:matching_seq}
\end{definition}

If $T$ is a generic Kohnert diagram, then the thread decomposition of $T$ induces the matching sequence $\Mtch_{\thd}(T)$ on $T$ defined by $x$ matching to $y$ for cells $x \in T$ in column $i+1$ and $y \in T$ in column $i$ if and only if $x$ and $y$ are in the same thread. 

\begin{figure}[ht]
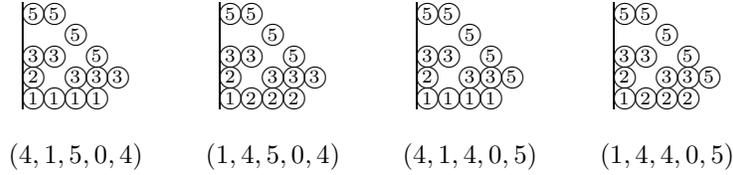

  \begin{displaymath}
    \arraycolsep=1.5\cellsize
    \begin{array}{cccc}
      \vline\cirtab{%
        5 & 5 \\
        & & 5 \\
        3 & 3 & & 5 \\
        2 & & 3 & 3 & 3 \\
        1 & 1 & 1 & 1 } &
      \vline\cirtab{%
        5 & 5 \\
        & & 5 \\
        3 & 3 & & 5 \\
        2 & & 3 & 3 & 3 \\
        1 & 2 & 2 & 2 } &
      \vline\cirtab{%
        5 & 5 \\
        & & 5 \\
        3 & 3 & & 5 \\
        2 & & 3 & 3 & 5 \\
        1 & 1 & 1 & 1 } &
      \vline\cirtab{%
        5 & 5 \\
        & & 5 \\
        3 & 3 & & 5 \\
        2 & & 3 & 3 & 5 \\
        1 & 2 & 2 & 2 } \\ \\
      (4,1,5,0,4) &
      (1,4,5,0,4) &
      (4,1,4,0,5) &
      (1,4,4,0,5)
    \end{array}
  \end{displaymath}
  \caption{\label{fig:matching}The four possible matching sequences of a generic Kohnert diagram along with their anchor weights (below), where the matched cells in adjacent columns are labeled the same.}
\end{figure}

\begin{example}
  Consider the generic Kohnert diagram from Example~\ref{ex:thread}. Considering columns $4$ and $5$, by Definition~\ref{def:matching}(3) there are two possible matchings since the cell in column $5$ cannot match to the cell in column $4$, row $1$. Continuing left, the matching is unique for columns $3$ and $4$ as well as for columns $2$ and $3$. However, we again have a choice for columns $1$ and $2$, with the cell in column $2$ row $1$ matching either to the cell in column $1$ row $1$ or the cell in column $1$ row $2$. The resulting four possible matching sequences are depicted in Fig.~\ref{fig:matching}. Notice the leftmost matches the thread decomposition in Fig.~\ref{fig:thread}.
  \label{ex:matching}
\end{example}

By the Hall Marriage Theorem, the characterization of generic Kohnert diagrams $T$ in Proposition~\ref{prop:lemma2.2} is equivalent to the existence of a matching sequence on $T$.

While matchings allow us to determine if an arbitrary diagram is a generic Kohnert diagram, in order to prove Theorem~\ref{thm:RSKD} we must be able to determine as well for which weak compositions $\comp{b}$ a generic Kohnert diagram lies in $\KD(\comp{b})$. 

\begin{definition}
  Given a matching $M$ on a diagram $T$, the \emph{anchor weight of $M$} is the weak composition $\wt(M)$ whose $i$th part is the number of cells along the path in $M$ that terminates in column $1$, row $i$.
  \label{def:wts}
\end{definition}

\begin{lemma}
  Let $T$ be a generic Kohnert diagram, and let $M$ be a matching sequence on $T$. Then $T \in \KD(\wt(M))$.
  \label{lem:matching-wt}
\end{lemma}

\begin{proof}
  We proceed by induction on the number of columns $c$ occupied by $T$. If $c=1$, then $T$ is a key diagram with $T = \kd_{\wt(M)} \in \KD(\wt(M))$.

  Suppose $c>1$ and assume the result for any generic Kohnert diagram occupying $c-1$ columns.
  Let $T'$ be the diagram obtained from $T$ by removing all the cells in the first column and pushing each cell in columns $2$ to $c$ to the left from position $(i,j)$ to position $(i-1,j)$.
  Let $M'$ be the matching sequence on $T'$ induced from $M$ by preserving all the existing matchings between cells in $T$ that were moved left to get $T'$. In particular, $T'$ is a generic Kohnert diagram by Proposition~\ref{prop:lemma2.2}, and by induction we have $T' \in \KD(\wt(M'))$.

  Let $T''$ be the diagram with the same first column as $T$ and the key diagram $\kd_{\wt(M')}$ in columns $2$ and beyond. Then $T \preceq T''$, so it suffices to show $T'' \in \KD(\wt(M))$. We do this by induction on the number of connected components of $M$.
  Observe the second column of $T''$ coincides with the first column of $T'$, and so $T$ and $T''$ coincide in the first two columns. Let $M''$ be the matching sequence on $T''$ defined by $M''(y)$ is the cell to the left of $y$ for $y$ strictly right of the second column, and $M''(y) = M(y)$ for $y$ in the second column. Then $\wt(M'') = \wt(M)$.

  If $M''$ has one component, then each column of $T''$ beyond the first has at most one cell. Letting $y$ denote the cell in column $2$, we have apply reverse Kohnert moves to columns $2,3,$ and so on until the cell lies in the same row as $M''(y)$. The corresponding matching is preserved, and so $T'' \in \KD(\wt(M))$ as desired. If $M''$ has more than one component, then let $x$ denote the highest cell in the first column of $T''$, and we may similarly apply reverse Kohnert moves to columns $2,3,\ldots$ to the cells on the component of $x$ until they lie in the same row as $x$. Having done this, we may remove the top row from the result, correspondingly removing one component of the matching. By induction, the remainder lifts as well, and so once again $T'' \in \KD(\wt(M))$.
\end{proof}



We may now strengthen Proposition~\ref{prop:lemma2.2} as follows.

\begin{theorem}
  For $T$ an arbitrary diagram, $T$ is a Kohnert diagram for $\comp{a}$ if and only if there exists a matching sequence $M$ on $T$ with $\wt(M) \preceq \comp{a}$.
  \label{thm:TFAE_kohnert}
\end{theorem}

\begin{proof}
  If a diagram $T$ is not a generic Kohnert diagram, then by Proposition~\ref{prop:lemma2.2} both statements are indeed false for all weak compositions $\comp{a}$. Suppose, then, $T$ is a generic Kohnert diagram. Statement (1) implies (2) using the thread decomposition by Lemma~\ref{lem:3.7}. Finally, to see (2) implies (1), we have $T \in \KD(\wt(M))$ by Lemma~\ref{lem:matching-wt} and $\wt(M) \preceq \comp{a}$, so, by Proposition~\ref{prop:lswap}, $T \in \KD(\comp{a})$. 
\end{proof}

Theorem~\ref{thm:TFAE_kohnert} yields the following useful characterization of the left swap order. 

\begin{corollary}
  For weak compositions $\comp{a}$ and $\comp{b}$, we have $\comp{b} \preceq \comp{a}$ if and only if $\wt(M) \preceq \comp{a}$ for some matching sequence $M$ on $\kd_{\comp{b}}$.
    \label{cor:TFAE_lswap}
\end{corollary}

We now have enough tools to show $\ins_1(T)$ is indeed a generic Kohnert diagram. 
Moreover, we can show $\ins_1$ sends $T$ to the appropriate target space. 
That is, if $T \in \KD(\comp{a})$, then $\ins_1(T) \in \KDs(\comp{a},1)$.
We will want to show $\ins_1$ is in fact an injective map, and the following lemma will be instrumental in helping us recover $T$ from $\ins_1(T)$. 
More generally, we use the following lemma for constructing maps in the reverse direction from the target space in subsequent sections.

\begin{lemma}
  Let $\comp{a}$ be a weak composition. For every diagram $U \in \KDs(\comp{a},n)$, there exists a unique column index $c$ such that $\cwt(U) = \cwt(\thread(U)) = \cwt(\comp{a}) + \comp{e}_c$. 
  Moreover, for every weak composition $\comp{b} \in \lswap(\comp{a})$ and positive integer $k \leq n$ such that $U \in \KD(\comp{b} + \comp{e}_k)$, we have $c = b_k + 1$.  
  \label{lem:col-added}
\end{lemma}

\begin{proof}
  Since $U \in \KDs(\comp{a},n)$, there exist a weak composition $\comp{b} \in \lswap(\comp{a})$ and positive integer $k \leq n$ such that $U \in \KD(\comp{b} + \comp{e}_k)$.
  Proposition~\ref{prop:lswap} implies $\kd_{\comp{b}} \in \KD(\comp{a})$, and since Kohnert moves preserve column weights, we have $\cwt(\comp{b}) = \cwt(\comp{a})$. Therefore the column index $b_k + 1$ satisfies the conditions of the proposition for the diagram $\kd_{\comp{b}+\comp{e}_k}$. 
  Since $U \in \KD(\comp{b} + \comp{e}_k)$, we have $\thread(U) \preceq \comp{b} + \comp{e}_k$ by Lemma~\ref{lem:3.7}, and hence $\kd_{\thread(U)} \in \KD(\comp{b} + \comp{e}_k)$ by Proposition~\ref{prop:lswap}. Again, since Kohnert moves preserve column weights, it follows that $b_k + 1$ satisfies the given conditions. 
  
  Notice $\kd_{\comp{b} + \comp{e}_k} \in \KDs(\comp{a},n)$, and since $\comp{b} \preceq \comp{a}$, we have $\cwt(\comp{b}) = \cwt(\comp{a})$. It follows that $b_k + 1$ is the unique column index satisfying the conditions above for $\kd_{\comp{b} + \comp{e}_k}$. Now, since $U \in \KD(\comp{b} + \comp{e}_k)$ and since Kohnert moves preserve column weights, we have $\cwt(U) = \cwt(\kd_{\comp{b} + \comp{e}_k}) = \cwt(\comp{b} + \comp{e}_k)$. Therefore, $c = b_k + 1$. 
\end{proof}

\begin{definition}
  For $\comp{a}$ a weak composition and $U \in \KDs(\comp{a},n)$, the \emph{added column} (with respect to $\comp{a}$) of $U$ is the unique column index satisfying Lemma~\ref{lem:col-added}.
  \label{def:added-column}
\end{definition}

\begin{figure}[ht]
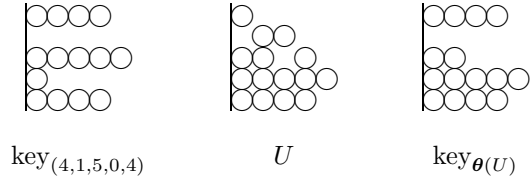

  \begin{displaymath}
    \arraycolsep=2\cellsize
    \begin{array}{ccc}
      \vline\cirtab{%
        \ & \ & \ & \ \\
        \\
        \ & \ & \ & \ & \ \\
        \ \\
        \ & \ & \ & \ } &        
      \vline\cirtab{%
        \ &   &   &   \\
        & \ & \ &   \\
        \ & \ &   & \ \\
        \ & \ & \ & \ & \ \\
        \ & \ & \ & \ } &
      \vline\cirtab{%
        \ & \ & \ & \ \\
        \\
        \ & \ \\
        \ & \ & \ & \ & \ \\
        \ & \ & \ & \ } \\ \\
      \kd_{(4,1,5,0,4)} & U & \kd_{\thread(U)}
    \end{array}
  \end{displaymath}
  \caption{\label{fig:added-column} A generic Kohnert diagram $U$ in $\KDs((4,1,5,0,4),n)$ for $n\geq 3$ with added column $2$.}
\end{figure}

\begin{example}
  Let $\comp{a}=(4,1,5,0,4)$, and consider the generic Kohnert diagram $U$ in Fig.~\ref{fig:added-column}. We have $\cwt(U) = (4,4,3,3,1) = \cwt(\comp{a})+\comp{e}_2$, and so $U$ has added column $c=2$. Furthermore, $\thread(U) = (4,5,2,0,4) = \comp{b}+\comp{e}_3$ where $\comp{b} = (4,5,1,0,4) \prec \comp{a}$. Thus we have $U \in \KD(\comp{b}+\comp{e}_3)$ and indeed $c = b_3+1$. 
  \label{ex:added-column}  
\end{example}

\begin{theorem}
  For each weak composition $\comp{a}$, the map $\ins_1$ induces a weight-preserving bijection
  \begin{equation}
    \KD(\comp{a}) \times \KD(\comp{e}_1) \stackrel{\sim}{\longrightarrow} \KDs(\comp{a},1).
  \end{equation}
  In particular, Theorem~\ref{thm:RSKD} is proved for $k=1$.
  \label{thm:insertion-bottom}
\end{theorem}

\begin{proof}
Let $T \in \KD(\comp{a})$, and consider the matching sequence $M = \Mtch_{\thd}T$ on $T$. Since $(c,1)$ is the left-most empty position of $T$ in row $1$, the thread decomposition algorithm implies that the path 
$$P = (1,1) \leftarrow (2,1) \leftarrow \cdots \leftarrow (c-1,1)$$
is a (weakly) connected component of $M$. 
We can extend $P$ by appending the matching $(c-1,1) \leftarrow (c,1)$ to it.
It follows that the directed graph 
$$M' = M \cup ((c-1,1) \leftarrow (c,1))$$ 
is a matching sequence with anchor weight $$\wt(M') = \thread(T) + \comp{e}_1.$$
Thus, $\ins_1(T) = T \sqcup \{(c,1)\} \in \KD(\thread(T) + \comp{e}_1)$ 
(by Theorem~\ref{thm:TFAE_kohnert}), and
since $\thread(T) \preceq \comp{a}$  (by Lemma~\ref{lem:3.7}),
we have $\ins_1(T) \in \KDs(\comp{a},1)$. 

Since $T \in \KD(\comp{a})$ and Kohnert moves preserve column weights, it follows that 
$\cwt(T) = \cwt(\comp{a})$.
So by Lemma~\ref{lem:col-added}, $c$ is the added column of $\ins_1(T)$, and we may recover $T$ from $\ins_1(T)$. 

On the other hand, for every diagram $U \in \KDs(\comp{a},1)$, we have $U \in \KD(\comp{b} + \comp{e}_1)$ for some weak composition $\comp{b} \in \lswap(\comp{a})$. 
Now consider the thread decomposition $M=\Mtch_{\thd}(U)$ on $U$.
By Theorem~\ref{thm:TFAE_kohnert}, $\thread(M) \preceq \comp{b} + \comp{e}_1$. In particular, $U$ must occupy every position in row $1$ at every column $1$ to $b_1 + 1$.  

Let $x \in U$ be the cell at position $(b_1 + 1,1)$. Then $x$ must be in the same thread at the cell in position $(1,1)$, and it must be the rightmost cell of that thread. In particular, removing $x$ from both $U$ and $M$ gives a matching sequence on $U \setminus \{x\}$ with anchor weight $\comp{b}$. By Theorem~\ref{thm:TFAE_kohnert}, we have $U \setminus \{x\} \in \KD(\comp{b}) \subset \KD(\comp{a})$. Since $\row(x) = 1$ and the diagram $U \setminus \{x\}$ occupies every position in row $1$ left of $x$, we have $\ins_1(U \setminus \{x\}) = U$. It also follows from Lemma~\ref{lem:col-added} that $\col(x) = b_1 + 1$ is the added column of $U$.

Hence, the map $\ins_1: \KD(\comp{a}) \longrightarrow \KDs(\comp{a},1)$ has an inverse that is well-defined on the target space. 
The desired bijection follows.
\end{proof}

\subsection{Rectification}
\label{sec:bijection-rectify}

Not every diagram is a generic Kohnert diagram. The criterion of Proposition~\ref{prop:lemma2.2} extends to a measurable way of identifying where and to what extent a diagram fails to be a generic Kohnert diagram.

\begin{definition}
  Let $T$ be an arbitrary diagram. For each position $(c,r)$ with $c > 1$, define 
  \begin{equation}
    \matchable_T(c,r) = \# \{(c-1,s) \in T \mid s \geq r\} - \# \{(c,s) \in T \mid s \geq r\}.
    \label{e:matchability-measure}
  \end{equation}
\end{definition}

Proposition~\ref{prop:lemma2.2} states $T$ is a generic Kohnert diagram if and only if $\matchable_T(c,r) \geq 0$ for all positions $(c,r)$ with $c>1$. When this fails, say in some column $c>1$, we may take $r$ to be the highest row index such that
$$
\matchable_T(c,r) = \min_{r'}\{\matchable_T(c,r')\} < 0.
$$
Then $T$ has a cell at position $(c,r)$ and no cell at position $(c-1,r)$. Using this observation, we can define a procedure, first used by Assaf and Gonz\'{a}lez \cite{Ass-K,AG}, to correct, or \emph{rectify}, $T$ by moving certain cells left.

\begin{definition}
  For $T$ an arbitrary diagram, define the diagram $\rect(T)$ by
  \begin{itemize}
  \item if $\matchable_T(c,r) \geq 0$ at every position $(c,r)$ with $c > 1$, then $\rect(T) = T$;
  \item otherwise, let $c > 1$ be the leftmost column index and let $r$ be the highest row index such that
    $$
    \matchable_T(c,r) = \min_{r'}\{\matchable_T(c,r')\} < 0,
    $$
    and set $\rect(T)$ to be the diagram obtained by pushing the cell in position $(c,r)$ left to the empty position $(c-1,r)$. 
  \end{itemize}
  \label{def:rect}
\end{definition}

\begin{figure}[ht]
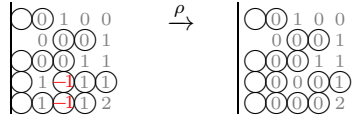

  \begin{displaymath}
    \vline\nulltab{\grayify{\ } & \grayify{0} & \color{gray}1 & \color{gray}0 & \color{gray}0 \\
      & \color{gray}0 & \grayify{0} & \grayify{0} & \color{gray}1 \\
      \grayify{\ } & \grayify{0} & \grayify{0} & \color{gray}1 & \color{gray}1\\
      \grayify{\ } & \color{gray}1 & \redify{-\!1} & \grayify{1} & \grayify{1} \\
      \grayify{\ } & \grayify{1} & \redify{-\!1} & \grayify{1} & \color{gray}2}
    \hspace{2\cellsize}
    \xrightarrow{\rho}
    \hspace{2\cellsize}
    \vline\nulltab{\grayify{\ } & \grayify{0} & \color{gray}1 & \color{gray}0 & \color{gray}0 \\
      & \color{gray}0 & \grayify{0} & \grayify{0} & \color{gray}1 \\
      \grayify{\ } & \grayify{0} & \grayify{0} & \color{gray}1 & \color{gray}1\\
      \grayify{\ } & \grayify{0} & \color{gray}0 & \grayify{0} & \grayify{1} \\
      \grayify{\ } & \grayify{0} & \grayify{0} & \grayify{0} & \color{gray}2}
  \end{displaymath}
  \caption{\label{fig:rho}A diagram $T$ (left) with each position $(c,r)$ with $c>1$ labeled by $\matchable_T(c,r)$, and its image under $\rho$ (right).}
  \end{figure}

\begin{example}
  Consider the diagram $T$ shown on the left side of Fig.~\ref{fig:rho}. For each position $(c,r)$, the index $\matchable_T(c,r)$ defined in Eq.~\eqref{e:matchability-measure} is indicated. Here the leftmost (only) column index $c$ for which $\matchable_T(c,r)<0$ for some $r$ is column $3$, and the highest row index $r$ for which $\matchable_T(3,r)<0$ is row $2$. Therefore $\rho$ acts on $T$ by moving the cell in position $(3,2)$ left to the empty position $(2,2)$, resulting in the diagram on the right side of Fig.~\ref{fig:rho}. Notice $\matchable_{\rho(T)}(c,r)\geq 0$ for all $c,r$, and so $\rho(T)$ is a Kohnert diagram. While one application of $\rho$ will not always result in a Kohnert diagram, we show below that repeated applications will.
  \label{ex:rho}
\end{example}

By our earlier observations $\rect$ is well-defined over all diagrams, and so it can be composed with itself repeatedly until the result is a generic Kohnert diagram. 

\begin{proposition}
  Given an arbitrary diagram $T$, there exists some integer $m \geq 0$ such that for all $N > m$, $\rect^N(T) = \rect^m(T)$. 
\end{proposition}

\begin{proof}
  Cells move only left under $\rect$, so since a diagram has finitely many cells and lies in the first quadrant, the procedure must ultimately terminate when the characterization of Kohnert diagrams in Proposition~\ref{prop:lemma2.2} is satisfied. Weight preservation is immediate since cells never change rows.
\end{proof}

\begin{definition}
  The \emph{rectification} of an arbitrary diagram $T$ is the result of iteratively applying $\rect$ as needed until we have a generic Kohnert diagram. That is,
  \begin{equation}
    \rectify(T) = \rect^m(T),
  \end{equation}
  for $m$ sufficiently large such that $\rect^{m+1}(T) = \rect^m(T)$.  
\end{definition}

\begin{figure}[ht]
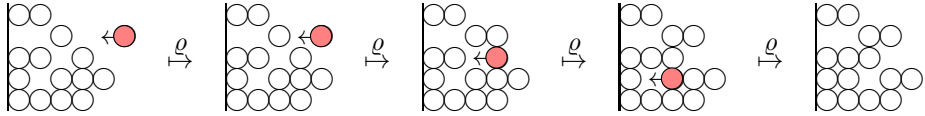

  \begin{displaymath}
    \arraycolsep=0.75\cellsize
    \begin{array}{ccccccccc}
      \vline\cirtab{\ & \ \\ & & \ & & & \mball \\ \ & \ & & \ \\ \ & & \ & \ & \ \\ \ & \ & \ & \ } &
      \raisebox{-2\cellsize}{$\stackrel{\displaystyle\rect}{\mapsto}$} &
      \vline\cirtab{\ & \ \\ & & \ & & \mball \\ \ & \ & & \ \\ \ & & \ & \ & \ \\ \ & \ & \ & \ } &
      \raisebox{-2\cellsize}{$\stackrel{\displaystyle\rect}{\mapsto}$} &
      \vline\cirtab{\ & \ \\ & & \ & \ \\ \ & \ & & \mball \\ \ & & \ & \ & \ \\ \ & \ & \ & \ } &
      \raisebox{-2\cellsize}{$\stackrel{\displaystyle\rect}{\mapsto}$} &
      \vline\cirtab{\ & \ \\ & & \ & \ \\ \ & \ & \ \\ \ & & \mball & \ & \ \\ \ & \ & \ & \ } &
      \raisebox{-2\cellsize}{$\stackrel{\displaystyle\rect}{\mapsto}$} &
      \vline\cirtab{\ & \ \\ & & \ & \ \\ \ & \ & \ \\ \ & \ & & \ & \ \\ \ & \ & \ & \ } 
    \end{array}
  \end{displaymath}
  \caption{\label{fig:rect}Rectification of a weak but not generic Kohnert diagram, where $\rect$ acts by moving the colored cell left.}
\end{figure}

\begin{example}
  Consider the diagram on the left side of Fig.~\ref{fig:rect}. This is the diagram from Fig.~\ref{fig:colsort}, which is a Kohnert diagram for the weak composition $(4,1,5,0,4)$, with an additional cell in row $4$ beyond the last occupied column. Here $\matchable(c,r) \geq 0$ for $c < 5$ or $r > 4$, and $\matchable(5,r) = -1$ for $r \leq 4$. Therefore $\rect$ will act on this additional cell in position $(5,4)$, moving it one column to the left within its row to position $(4,4)$. Iterating, $\rect$ acts by moving the colored cell ($\cball{red}$) left four times until arriving at the Kohnert diagram for $(4,2,5,0,4)$ on the right side of Fig.~\ref{fig:rect}.
  \label{ex:rect}
\end{example}

We use rectification in a limited sense in this paper, though it is worth noting the sense in which we use it gives a generalization of RSK insertion.

\begin{theorem}
  Let $\mathbb{D}:\SSYT_n(\lambda)\rightarrow\KD(\mathrm{rev}(\lambda))$ denote the bijection from semistandard Young tableaux to generic Kohnert diagrams obtained by moving the cells in column $c$ with entry $r$ to position $(c,n+1-r)$. Then for $T\in\SSYT_n(\lambda)$ and $1\leq j\leq n$, we have
  \begin{equation}
    \mathbb{D}(\mathrm{RSK}(T,j)) = \rectify(\mathbb{D}(T) \sqcup \{(\lambda_1+1,n+1-j)\}),
  \end{equation}
  where $\mathrm{RSK}(T,j)$ denotes the result of inserting $j$ into $T$ via RSK.
  \label{thm:RSK-rect}
\end{theorem}

\begin{proof}
  We assume familiarity with RSK insertion; see \cite{EC2} for details. Recall that when an entry $k$ bumps and entry $l$ in some column $c$, we have $k<l$ and there is no entry $k$ in column $c$. Suppose the RSK algorithm generates a sequence $(T,j) = (T_1,j_1),\ldots,(T_p,j_p),(T_{p+1},\varnothing)$ where $T_i \in \SSYT(\lambda)$ is obtained by inserting $j_{i-1}$ into row $i-1$ for $1 \leq i \leq p$, $j_i$ is the bumped entry for $i \leq p$, and $T_{p+1} = \mathrm{RSK}(T,j)$. Then we may consider the diagrams $U_i = \mathbb{D}(T_i) \sqcup \{(c_i,n+1-j_i)\}$ where $c_i$ is the column in which $j_{i}$ bumps $j_{i+1}$ in $T_i$ for $1 \leq i\leq p$ and $c_{p+1} = \lambda_{p+1}+1$. Also set $U_{p+1} = \mathbb{D}(T_{p+1})$, which is a generic Kohnert diagram. It suffices to show for each $i > 1$ there exists some integer $s_i \geq 0$ such that $U_i = \rect^{c_i-c_{i+1}}(U_{i-1})$.

  The difference between $U_{k}$ and $U_{k+1}$ is that the cell in row $n+1-j_k$ has moved left from column $c_k$ to column $c_{k+1}$. If $c_{k+1} = c_k$, then $U_{k+1} = U_{k}$, and we are done. If $c_{k+1} < c_k$, then since moving entries in column $c_{k}$ in rows $r \geq j_k$ will not result in a semistandard Young tableaux, by Proposition~\ref{prop:sort}, $U_k$ is not a generic Kohnert diagram. Moreover, this properties persist even as the cell $y$ in row $n+1-j_k$ column $c_k$ of $U_k$ moves left provided it stays strictly right of column $c_{k+1}$. Since any matching on $\mathbb{D}(T_k)$ gives a matching between all columns, the unmatched cell $y$ of $U_k$ must be the cell that moves under $\rect$, and again, this persists as it moves left until reaching column $c_{k+1}$. Thus $U_i = \rect^{c_i-c_{i+1}}(U_{i-1})$ as desired.
\end{proof}

\begin{figure}[ht]
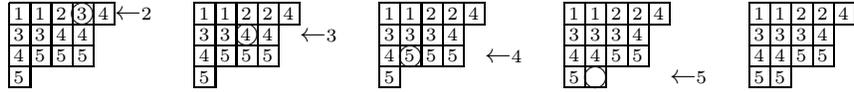

  \begin{displaymath}
    \arraycolsep=\cellsize
    \begin{array}{ccccc}
      \tableau{1 & 1 & 2 & \circify{3} & 4 \\ 3 & 3 & 4 & 4 \\ 4 & 5 & 5 & 5 \\ 5}\raisebox{0.2\cellsize}{$\leftarrow\scriptstyle 2$} &
      \tableau{1 & 1 & 2 & 2 & 4 \\ 3 & 3 & \circify{4} & 4 \\ 4 & 5 & 5 & 5 \\ 5}\raisebox{-0.8\cellsize}{$\leftarrow\scriptstyle 3$} &
      \tableau{1 & 1 & 2 & 2 & 4 \\ 3 & 3 & 3 & 4 \\ 4 & \circify{5} & 5 & 5 \\ 5}\raisebox{-1.8\cellsize}{$\leftarrow\scriptstyle 4$} &
      \tableau{1 & 1 & 2 & 2 & 4 \\ 3 & 3 & 3 & 4 \\ 4 & 4 & 5 & 5 \\ 5 & \circify{ \ }}\raisebox{-2.8\cellsize}{$\leftarrow\scriptstyle 5$} &
      \tableau{1 & 1 & 2 & 2 & 4 \\ 3 & 3 & 3 & 4 \\ 4 & 4 & 5 & 5 \\ 5 & 5}
    \end{array}
  \end{displaymath}
  \caption{\label{fig:RSK}An example of RSK insertion on an element of $\SSYT_5(5,4,4,1)$.}
\end{figure}

\begin{example}
  Consider the tableau $T\in\SSYT_5(5,4,4,1)$ on the left side of Fig.~\ref{fig:RSK} and the integer $j=2$. Under the RSK insertion algorithm, the circled entry is bumped by the external entry to the right of the tableau. Comparing this with the diagrams in Fig.~\ref{fig:rho}, $\mathbb{D}(T)$ is the leftmost diagram without the shaded cell being appended. The first tableau in Fig.~\ref{fig:RSK} maps under $\mathbb{D}$ to the third diagram in Fig.~\ref{fig:rho}, the second tableau maps to the fourth diagram, and the last three tableaux map to the last diagram, with order taken left to right in each figure.
  \label{ex:RSK}
\end{example}

\subsection{Top insertion}
\label{sec:bijection-top}

Rectification is the heart of our bijection for Theorem~\ref{thm:RSKD} for sufficiently large values of $k$ and plays a role in the general case as well.

\begin{definition}
  Let $T$ be a generic Kohnert diagram, let
  \[ c = \max_{x \in T} \{\col(x)\} \]
  be the rightmost occupied column of $T$, and let $j \leq n$ be a positive integer. Then the \emph{top insertion map} $\ins_{\infty}$ sends the tuple $(T,j)$ to the diagram
  \begin{equation}
    \ins_{\infty}(T,j) = \rectify(T \sqcup \{(m+1,j)\}) .
    \label{e:rectification}
  \end{equation}
  \label{def:rectification}
\end{definition}

\begin{example}
  Consider the weak composition $\comp{a} = (0,2,1)$ and $k=3$. Refining the right-hand side of Theorem~\ref{thm:RSKD} using Theorem~\ref{thm:k-addable}, we expect a bijection
  \[ \KD{(0,2,1)} \times \KD{(0,0,1)} \stackrel{\sim}{\rightarrow} \KD{(1,2,1)} \cup \KD{(0,3,1)} \cup \KD{(0,2,2)} . \]
  Indeed, Fig.~\ref{fig:insert-top} shows the images of the Kohnert diagrams in $\KD(0,2,1)$ under the top insertion map $\ins_{\infty}$, which is precisely the union on the right.
  \label{ex:top}
\end{example}

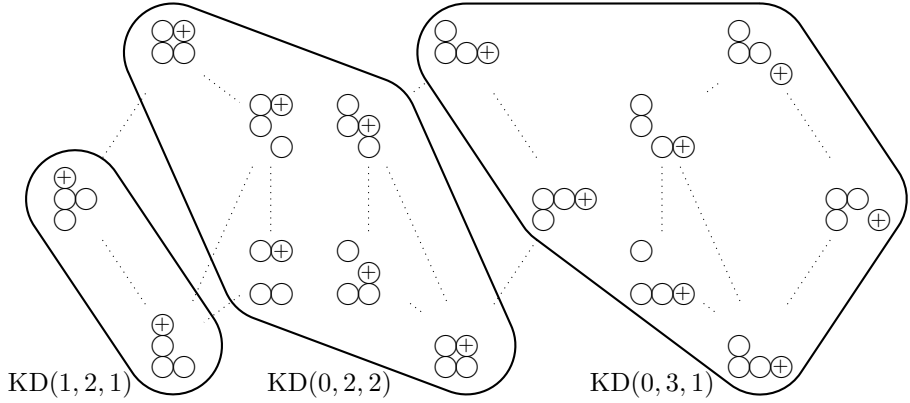
\begin{figure}[ht]
  \begin{center}
    \begin{tikzpicture}[xscale=1.3,yscale=1.3]
      \node at (1,3)   (A3) {$\vline\cirtab{\ & + \\ \ & \ \\ & }$};
      \node at (0,1.5) (B3) {$\vline\cirtab{+ & \\ \ & \ \\ \ & }$};
      \node at (2,2.25)(C3) {$\vline\cirtab{\ & + \\ \ & \\ & \ }$};
      \node at (2,0.75)(D3) {$\vline\cirtab{\ & + \\ & \\ \ & \ }$};
      \node at (1,0)   (E3) {$\vline\cirtab{+ & \\ \ & \\ \ & \ }$};
      \node at (4,3)   (A2) {$\vline\cirtab{\ & \\ \ & \ & + \\ & }$};
      \node at (5,1.5) (B2) {$\vline\cirtab{& \\ \ & \ & + \\ \ & }$};
      \node at (3,2.25)(C2) {$\vline\cirtab{\ & \\ \ & + & \\ & \ }$};
      \node at (3,0.75)(D2) {$\vline\cirtab{\ & \\ & + & \\ \ & \ }$};
      \node at (4,0)   (E2) {$\vline\cirtab{& \\ \ & + & \\ \ & \ }$};
      \node at (7,3)   (A1) {$\vline\cirtab{\ & \\ \ & \ \\ & & + }$};
      \node at (8,1.5) (B1) {$\vline\cirtab{& \\ \ & \ \\ \ & & + }$};
      \node at (6,2.25)(C1) {$\vline\cirtab{\ & \\ \ & \\ & \ & + }$};
      \node at (6,0.75)(D1) {$\vline\cirtab{\ & \\ & \\ \ & \ & + }$};
      \node at (7,0)   (E1) {$\vline\cirtab{& \\ \ & \\ \ & \ & + }$};
      \draw[thick] \convexpath{B3,E3}{5mm};
      \draw[thick] \convexpath{A1,B1,E1,B2,A2}{5mm};
      \draw[thick] \convexpath{A3,C2,E2,D3}{5mm};
      \node at (-0.05,-0.4) (KD121) {$\KD(1,2,1)$};
      \node at (2.6,-0.4) (KD022) {$\KD(0,2,2)$};
      \node at (5.9,-0.4) (KD031) {$\KD(0,3,1)$};
      \draw[dotted] (A3) -- (B3) ;
      \draw[dotted] (A3) -- (C3) ;
      \draw[dotted] (B3) -- (E3) ;
      \draw[dotted] (C3) -- (D3) ;
      \draw[dotted] (C3) -- (E3) ;
      \draw[dotted] (D3) -- (E3) ;
      \draw[dotted] (A2) -- (B2) ;
      \draw[dotted] (A2) -- (C2) ;
      \draw[dotted] (B2) -- (E2) ;
      \draw[dotted] (C2) -- (D2) ;
      \draw[dotted] (C2) -- (E2) ;
      \draw[dotted] (D2) -- (E2) ;
      \draw[dotted] (A1) -- (B1) ;
      \draw[dotted] (A1) -- (C1) ;
      \draw[dotted] (B1) -- (E1) ;
      \draw[dotted] (C1) -- (D1) ;
      \draw[dotted] (C1) -- (E1) ;
      \draw[dotted] (D1) -- (E1) ;
    \end{tikzpicture}
    \caption{\label{fig:insert-top}Three copies of the Kohnert diagrams for $(0,2,1)$ along with the appended cell $(\oplus)$ under the top insertion map $\ins_{\infty}(-,j)$ for $j=3,2,1$, from left to right.}
  \end{center}
\end{figure}

\begin{remark}
  Observe the maps $\ins_1$ and $\ins_{\infty}(-,1)$ differ in general. For instance, we have $\ins_1(\kd_{(0,2,1)}) = \kd_{(1,2,1)}$ whereas $\ins_{\infty}(\kd_{(0,2,1)},1)$ is the rightmost diagram in the top row of Fig.~\ref{fig:insert-top}, which is not a key diagram.
\end{remark}

We devote the remainder of this section to showing the top insertion map $\ins_{\infty}$ induces the bijection of Theorem~\ref{thm:RSKD} for sufficiently large values of $k$. It is important to note, however, that $\ins_{\infty}$ is itself independent of $k$.

\begin{definition}
  A cell $x$ in an arbitrary diagram $T$ is a \emph{removable cell} of $T$ if $T \setminus \{x\}$ is a generic Kohnert diagram. In this case, we say that $T$ is a \emph{weak Kohnert diagram} and the column occupied by $x$ is a \emph{removable column} of $T$.
  \label{def:weak}
\end{definition}

\begin{proposition}
  A weak but not generic Kohnert diagram $T$ has a unique removable column $c > 1$, and $c$ is characterized by
  \begin{itemize}
  \item $\min_r\{\matchable_T(c,r)\} = -1$;
  \item $\matchable_T(i,j) \geq 0$ for every column index $1 < i \neq c$ and every row index $j$. 
  \end{itemize}
  In particular, $T$ has a unique highest and a unique lowest removable cell.
  \label{prop:weak-Kohnert-removable-column}
\end{proposition}

\begin{proof}
  Since $T$ is not a generic Kohnert diagram, there exists a column index $c$ such that for some row index $r$, we have $\matchable_T(c,r) < 0$. For any cell $y$ in $T$ that is not in column $c$, we have
  \[ \matchable_{T \setminus \{y\}}(c,r) \leq \matchable_T(c,r) < 0, \]
  and so $y$ is not a removable cell. Therefore $c$ is the unique removable column of $T$.

  Let $y$ be a removable cell in column $c$. Since $T \setminus \{y\}$ is a generic Kohnert diagram, for every column index $i > 1$ with $i \neq c$ and for every row index $j$, we have
  \[ \matchable_T(i,j) \geq \matchable_{T \setminus \{y\}}(i,j) \geq 0. \]
  Additionally, if there exists a row index $r$ such that $\matchable_T(c,r) < 1$, then
  \[ \matchable_{T \setminus \{y\}}(c,r) \leq \matchable_T(c,r) + 1 < 0, \]
  so that $T \setminus \{y\}$ is not a generic Kohnert diagram, contradicting $y$ being a removable cell. Hence, $\matchable_T$ is bounded from below in column $c$ by $\min_r\{\matchable_T(c,r)\} = -1$.
\end{proof}

Relating this to rectification, we have the following characterization of the unique lowest removable cells for weak Kohnert diagrams.

\begin{lemma}
  For a cell $y$ in position $(c,r)$ of a weak but not generic Kohnert diagram $T$, the following are equivalent:
  \begin{enumerate}
  \item $y$ is the lowest removable cell of $T$;
  \item $y$ is the cell that moves under $T \mapsto \rect(T)$;
  \item $\matchable_T(c,r) = -1$ and for all row-indices $s > r$, $\matchable_T(c,s) \geq 0$.
  \end{enumerate}
  \label{lem:weak-Kohnert-removable-cell-lowest}
\end{lemma}

\begin{proof}
  By Proposition~\ref{prop:weak-Kohnert-removable-column} we have $\matchable_T(i,j) \geq -1$ with equality for only one column index $i>1$. Therefore any cell that satisfies any of conditions (1),(2), or (3) must lie in this column, i.e. $c=i$. In particular, (2) and (3) are equivalent.

  To see (1) implies (3), if $y$ is a removable cell of $T$ and $\matchable_T(c,s) < 0$ for some $s>r$, then $\matchable_{T\setminus\{y\}}(c,s) = \matchable_T(c,s) < 0$ contradicting that $y$ is removable.

  Finally, suppose (3) holds for $y$. If we remove a cell $z$ in column $c$ below $y$, then we still have $\matchable_{T \setminus \{z\}}(c,r) = \matchable_T(c,r) = -1$, so $T \setminus \{z\}$ is not a generic Kohnert diagram and $z$ is not a removable cell. Therefore, $y$ is weakly below the lowest removable cell of $T$. It remains to show $y$ is removable.
  
  For every column index $i > 1$ with $i \neq c, c+1$, we have $\matchable_{T \setminus \{y\}}(i,j) = \matchable_T(i,j) \geq 0$ for every row index $j$. In column $c$, we have $\matchable_{T \setminus \{y\}}(c,s) = \matchable_T(c,s) \geq 0$ for every row index $s>r$, and we have $\matchable_{T \setminus \{y\}}(c,s) = \matchable_T(c,s) + 1 \geq 0$ for every row index $s \leq r$. Then for some row index $s \leq r$, we have $\matchable_{T \setminus \{y\}}(c+1,s) = \matchable_T(c+1,s) - 1 = -1$, so that $\matchable_T(c+1,s) = 0$. 

  If $y$ is not a removable cell of $T$, then there is some $x$ in column $c$ strictly above $y$ that is removable since $T$ is a weak but not generic Kohnert diagram. However, since $x$ is above row $r$ it is also above row $s$, and so $\matchable_{T \setminus \{x\}}(c+1,s) = \matchable_T(c+1,s) - 1 = -1$, contradicting that $T \setminus \{x\}$ is a generic Kohnert diagram. Therefore, $y$ must be removable and hence is the lowest removable cell of $T$. Thus (3) implies (1). 
\end{proof}

For a weak Kohnert diagram is not generic, Lemma~\ref{lem:weak-Kohnert-removable-cell-lowest} characterizes the unique lowest removable cell. We next characterize the unique highest removable cell.

\begin{lemma}
  Let $y$ be the lowest removable cell of a weak but not generic Kohnert diagram $T$, and let $M = \Mtch_{\thd}(T \setminus \{y\})$.
  Then for every cell $x$ in the column of $y$ in $T \setminus \{y\}$, 
  $x$ is above $y$ if and only if $M(x)$ is above $y$. 
  \label{lem:weak-Kohnert-removable-cell-lowest-2}
\end{lemma}

\begin{proof}
  Let $y$ be in position $(c,r)$.
  By Lemma~\ref{lem:weak-Kohnert-removable-cell-lowest}, we must have 
  \[ \matchable_{T \setminus \{y\}}(c,r+1) = \matchable_T(c,r+1) = 0, \]
  so there is the same number of cells above $y$ in columns $c$ and $c-1$. Since $M$ gives a matching from the cells in column $c$ to the cells in column $c-1$, all cells above $y$ in column $c$ must match to the cells above $y$ in column $c-1$.
\end{proof}

To state the next result characterizing the unique highest removable cell, we remark that the threading algorithm described in Definition~\ref{def:thread} is in fact well-defined for any arbitrary diagram, and so it makes sense to refer to the thread decomposition of a weak Kohnert diagram as well. 
With no ambiguity, we extend our notation for matching sequences and write $\Mtch_{\thd}(T)$ for the directed graph on the cells of an arbitrary diagram $T$ induced by the thread decomposition of $T$. 

\begin{lemma}
  Let $x$ be the highest removable cell of a weak but not generic Kohnert diagram $T$. Then the removal of the cell $x$ from $T$ does not change how the other cells are threaded, that is,
  \[ \Mtch_{\thd}(T \setminus \{x\}) \subset \Mtch_{\thd}(T). \]
  Moreover, $x$ is the only removable cell of $T$ with this property and is the unique cell of $T$ not in the first column such that $M(x)$ does not exist.
  \label{lem:weak-Kohnert-removable-cell-highest}
\end{lemma}

\begin{proof}
  Let $M = \Mtch_{\thd}(T)$, and let $M^* = \Mtch_{\thd}(T \setminus \{x\})$. We show $x$ lies in a one cell thread in $M$. Suppose, for contradiction, that $M(y)=x$ for some cell $y$ of $T$. Since the threading algorithm selects the lowest unthreaded cell weakly above $y$ in the next column, we must have $x$ strictly below $M^*(y)$. However, this allows us to construct a matching sequence for $T\setminus\{w\}$ by replacing $w$ with $x$ in $M^*$. By Proposition~\ref{prop:lemma2.2}, this means $w$ is a removable cell for $T$ that lies strictly above $x$, a contradiction. Therefore no cell of $T$ threads into $x$.

  When threading $T$, the cell $x$ must be threaded after all cells to its right and all cells in the same column below it. Suppose, for contradiction, that some cell $z$ above $x$ in the same column has also not yet been threaded. Then again we may construct a matching sequence for $T\setminus\{z\}$ by replacing $z$ with $x$ in $M^*$, arriving at the same contradiction as below. Therefore all cells of $T$ in column $c$ are threaded before $x$. Consider the lowest position, say $(c,r)$, occupied by a remove-able cell of $T$; in particular, $x$ is weakly above row $r$. By Lemma~\ref{lem:weak-Kohnert-removable-cell-lowest-2}, the number of cells weakly above row $r$ in $T\setminus\{x\}$ is the same for columns $c$ and $c-1$, and so all cells in column $c-1$ of $T$ weakly above row $r$ are threaded, leaving no unthreaded cell into which $x$ can thread in $T$.

  By virtue of the threading algorithm selecting the lowest unthreaded cell, if $x$ is never chosen as $M(y)$ for any $y\in T$ and there is no $z\in T$ chosen as $M(x)$, then $M(y) = M^*(y)$ for all $y\in T\setminus\{x\}$ as desired. The latter claims follow.
\end{proof}

We next consider the relationship between the thread decompositions of a diagram with highest or lowest removable cell removed; see Fig.~\ref{fig:daisy}.

\begin{lemma}
  Let $T$ be a weak but not generic Kohnert diagram. Let $x$ (respectively, $y$) be the highest (respectively, lowest) removable cell of $T$. Set $M = \Mtch_{\thd}(T \setminus \{x\})$ and $N = \Mtch_{\thd}(T \setminus \{y\})$. Then there exists an ascending sequence $y = y_0, y_1, \cdots, y_t = x$ of cells in the removable column of $T$ such that 
  \begin{itemize}
  \item $M(y_i) = N(y_{i+1})$ for all $0 \leq i < t$;
  \item if there exists a cell $z_i$ such that $M(z_i) = y_i$, then $N(z_i) = y_{i+1}$;
  \item $M(u) = N(u)$ for all cells $u$ such that $u \neq y_i$ and $M(u) \neq y_i$ for any $i$.
  \end{itemize}
  \label{lem:weak-Kohnert-daisy-chain}
\end{lemma}

\begin{figure}[ht]
  \begin{center}
    \begin{tikzpicture}[xscale=1.25,yscale=1,
        roundnode/.style={circle, draw=black, inner sep=0pt, minimum size=18pt}]
      \node[roundnode] at (0,3.75) (wt1){$\scriptstyle w_{t-1}$};
      \node at (0,3) (wd) {$\vdots$};
      \node[roundnode] at (0,2) (w1) {$\scriptstyle w_{1}$};
      \node[roundnode] at (0,1) (w0) {$\scriptstyle w_{0}$};
      \node[roundnode] at (1,3.75) (yt) {$\scriptstyle x$};
      \node[roundnode] at (1,2.75) (yt1){$\scriptstyle y_{t-1}$};
      \node at (1,2) (yd) {$ \vdots$};
      \node[roundnode] at (1,1) (y1) {$\scriptstyle y_{1}$};
      \node[roundnode] at (1,0) (y0) {$\scriptstyle y$};
      \node[roundnode] at (2,2.5) (zt1){$\scriptstyle z_{t-1}$};
      \node at (2,1.75) (zd) {$ \vdots$};
      \node[roundnode] at (2,0.75) (z1) {$\scriptstyle z_{1}$};
      \node[roundnode] at (2,-.25) (z0) {$\scriptstyle z_{0}$};
      \draw[ultra thick] (z0) -- (y0) -- (w0);
      \draw[ultra thick] (z1) -- (y1) -- (w1);
      \draw[ultra thick] (zt1) -- (yt1) -- (wt1);
      \draw[thick,dashed] (z0) -- (y1) -- (w0);
      \draw[thick,dashed] (z1) -- (yd) -- (w1);
      \draw[thick,dashed] (zt1) -- (yt) -- (wt1);
    \end{tikzpicture}
    \caption{\label{fig:daisy} An illustration of the relation between the matching sequences $\Mtch_{\thd}(T \setminus \{x\})$ (solid) and $\Mtch_{\thd}(T \setminus \{y\})$ (dashed).}
  \end{center}
\end{figure}
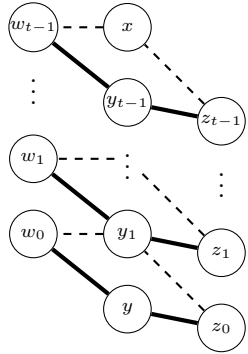

\begin{proof}
  The result is vacuous if $T$ has only one removable cell, so we may assume $x \neq y$. Since $T$ is a weak but not generic Kohnert diagram, $\Mtch_{\thd}(T)$ has at least two threads with one containing only $x$ by Lemma~\ref{lem:weak-Kohnert-removable-cell-highest}. If $\Mtch_{\thd}(T)$ has exactly two threads, then both $M$ and $N$ are single paths related by exchanging $x$ and $y$, and the result holds with $t=1$. We now proceed by induction on the number of threads of $\Mtch_{\thd}(T)$, noting that by Lemma~\ref{lem:weak-Kohnert-removable-cell-highest}, $\Mtch_{\thd}(T)$ is simply $M$ together with a thread consisting solely of $x$.

  Suppose the first thread of $M$ coincides with the first thread of $N$ so that, in particular, neither passes through $x$ or $y$. Let $T'$ denote the diagram obtained by deleting cells along this common thread in $T$, so that $x,y\in T'$. Then any removable cell of $T'$ must also be removable for $T$, ensuring both $x$ and $y$ are removable for $T'$ and remain the highest and lowest, respectively. By induction, the result holds for $T'$, and so too for $T$ once we restore the cells along the common first thread.

  Suppose then the first thread of $M$ differs from the first thread of $N$. Let $c$ denote the index of the removable column of $T$, and let $y_1$ denote the lowest cell in column $c$ strictly above $y$. If $c$ is the rightmost column occupied by $T$, then $y$ (resp. $y_1$) is the first cell of the first thread of $M$ (resp. $N$). Likewise, if $c$ is not the rightmost column occupied by $T$, then there is a cell $z$ in column $c+1$ such that the first thread of $M$ and of $N$ both contain $z$, agree weakly right of $z$, and have $M(z) \neq N(z)$. For this to occur, we must have $M(z)=y$ is the lowest cell in column $c$ weakly above $z$, and so in this case $N(z) = y_1$ since this is the first alternative to thread from $z$ once $y$ is removed. Therefore the first threads for $M$ and $N$ first diverge in column $c$, passing through $y$ and $y_1$, respectively.

  We claim $N(y_1) = M(y)$, as indicated in Fig.~\ref{fig:daisy}. Indeed, $M(y)$ is the lowest cell in column $c-1$ that lies strictly above $y$ (since $y$ has no cell immediately to its left), and $N(y_1)$ is the lowest cell in column $c-1$ that lies weakly above $y_1$ which lies above $y$. Therefore if $N(y_1) \neq M(y)$, then $M(y)$ must be strictly below $y_1$. Since $y_1$ is the lowest cell above $y$, there is no cell above $y$ in column $c$ that can match to $M(y)$ in $T\{y\}$, contradicting Lemma~\ref{lem:weak-Kohnert-removable-cell-lowest-2} and proving the claim.

  As the first thread is a local procedure based on cells of the diagram, and since $T \setminus \{x\}$ and $T \setminus \{y\}$ agree in columns $1$ through $c-1$, once $N(y_1) = M(y)$, the two first threads coincide in columns $1$ through $c-1$. If $y_1 = x$, then the result again holds with $t=1$. Otherwise, let $T'$ denote the diagram obtained by deleting cells along the first thread of $N$ so that $x,y\in T'$. Again, any removable cell of $T'$ must also be removable for $T$, ensuring both $x$ and $y$ are removable for $T'$ and remain the highest and lowest, respectively. By induction, the result holds for $T'$, and so too for $T$ once we restore the cells along the first thread of $N$.
\end{proof}

Observe from the proof of Lemma~\ref{lem:weak-Kohnert-daisy-chain}, if there exists $z_i$ such that $M(z_i) = y_i$ for some $i>0$, then there exists $z_{i-1}$ such that $M(z_{i-1}) = y_{i-1}$.

\begin{corollary}
  Let $T$ be a weak but not generic Kohnert diagram. Let $x$ (respectively, $y$) be the highest (respectively, lowest) removable cell of $T$. Then $\thread(T \setminus \{x\}) = \thread(T \setminus \{y\})$, and $y$ becomes the highest removable cell of $\rect(T)$.
    \label{cor:weak-kohnert-thread-decomposition}
\end{corollary}

\begin{proof}
  The characterization of the thread decompositions for $T \setminus \{x\}$ and $T \setminus \{y\}$ in Lemma~\ref{lem:weak-Kohnert-daisy-chain} ensures $\thread(T \setminus \{x\}) = \thread(T \setminus \{y\})$.
  Additionally, since $\rect(T) \setminus \{y\} = T \setminus \{y\}$, Lemma~\ref{lem:weak-Kohnert-removable-cell-lowest}(3) implies $\matchable_{\rect(T) \setminus \{y\}}(c,r+1) = \matchable_{T \setminus \{y\}}(c,r+1) = 0$. If we remove any cell $w$ above $y$ in $\rect(T)$, then $\matchable_{\rect(T) \setminus \{w\}}(c,r+1) = -1$, and so $\rect(T) \setminus \{w\}$ cannot be a generic Kohnert diagram. Hence, $y$ is the highest removable cell of $\rect(T)$ in its column.
\end{proof}

\begin{lemma}
  Let $U$ be a generic Kohnert diagram, and let $x$ be a removable cell of $U$. If $x$ is the highest removable cell in its column, then
  \begin{enumerate}
  \item $x$ is the rightmost cell of its path in $\Mtch_{\thd}(U)$, and
  \item any cell above $x$ in its column is not the rightmost cell of its path in $\Mtch_{\thd}(U)$.
  \end{enumerate}
  \label{lem:generic-Kohnert-removable-cell-highest-2}
\end{lemma}

\begin{proof}
  Let $M = \Mtch_{\thd}(U)$ and $M^* = \Mtch_{\thd}(U \setminus \{x\})$. Suppose, for contradiction, there exists a cell $y$ of $U$ such that $M(y) = x$. Then there exists some $w$ such that $w = M^*(y)$. Moreover, since $y$ chooses the lowest available cell into which to thread, $w$ must lie above $x$. However, this means replacing $w$ with $x$ in $M^*$ yields a matching sequence on $U \setminus \{w\}$, showing that $w$ is also removable and contradicting that $x$ is the highest removable cell of its column. This proves the first assertion.

  Now suppose, for contradiction, that for some cell $w$ above $x$ in its column, there is no cell $y$ such that $M(y)=w$. Then again we create a matching sequence on $U \setminus \{w\}$ by replacing $w$ with $x$ in $M^*$, and the second assertion follows.
\end{proof}

For the next proof and others to follow, we make use of the lengths of path components of matching sequences for which we now introduce notation.

\begin{definition}
  Let $M$ be a matching sequence on a generic Kohnert diagram $T$. For each cell $x \in T$, the \emph{matching path length for $x$ in $M$}, denoted by $\plength_M(x)$, is the number of cells in the connected component of $M$ containing $x$. 
  \label{def:path-length}
\end{definition}

\begin{lemma}
  Let $U$ be a generic Kohnert diagram, and let $x$ be a removable cell of $U$ that is the highest such in its column. Then $\thread(U) = \thread(U \setminus \{x\}) + \comp{e}_j$, where $j$ is the row index of the cell in the first column on the thread of $x$ in $\Mtch_{\thd}(U)$.
  \label{lem:generic-Kohnert-removable-cell-highest}
\end{lemma}

\begin{proof}
  By Lemma~\ref{lem:generic-Kohnert-removable-cell-highest-2}, $x$ is the last cell to be threaded in its column. Thus, without loss of generality, we may assume $x$ is the only cell in its column and $U$ is empty to the right of $x$. If $x$ lies in the first column, then the result is trivial, so assume $x$ lies in column $c>1$. Let $M = \Mtch_{\thd}(U)$ and $M^* = \Mtch_{\thd}(U \setminus \{x\})$. 
  
  Since $U$ and $U \setminus \{x\}$ have the same number of cells in column $c-1$, the number of threads of length at least $c-1$ must be the same for both $M$ and $M^*$.
  The lemma is equivalent to the equation
  \begin{equation}
    \{z \in U \mid \plength_M(z) \geq c - 1\} = \{z \in U \setminus \{x\} \mid \plength_{M^*}(z) = c - 1\} \cup \{x\},
    \label{e:path-lengths-after-cell-removed}
  \end{equation}
  since removing these cells would reduce both $U$ and $U \setminus \{x\}$, and the thread decompositions of $U$ and $U \setminus \{x\}$ would coincide after those cells have been threaded.

  Suppose, for contradiction, Eq.~\eqref{e:path-lengths-after-cell-removed} fails. Since the sets have the same cardinality, this happens only if there exists some cell of $U$ in the right set that is not in the left set. Let $y$ be the rightmost then highest such cell, so that $\plength_{M^*}(y) = c - 1$ but $\plength_{M}(y) < c - 1$. Since $\plength_{M}(y) < c - 1$, we must have $y$ strictly left of column $c-1$. We will show  $y$ has an infinite descending sequence of cells beneath it, an obvious contradiction to the finiteness of the diagram $U$.

  To begin, since $\plength_{M^*}(y)=c-1$ and $y$ is left of column $c-1$, we may take $z \in U\setminus\{x\}$ such that $M^*(z) = y$. Set $y_1 = M(z) \neq y$. Since threading selects the lowest available cell, we must have $M(z) = y_{1}$ below $M^*(z) = y$. Moreover, $\plength_{M^*}(z) = \plength_{M^*}(y) = c-1$, so by choice of $y$ we have $\plength_{M}(z) \geq c-1$. Therefore $\plength_{M}(y_1) = \plength_{M}(z) \geq c-1$. Again by choice of $y$, since $y_1$ lies below it, we must have $\plength_{M^*}(y_1) = c-1$. Thus there exists $z_1$ such that $M^*(z_1) = y_1$, and $z_1 \neq z$ since $y_1 \neq y$. Since $M^*(z) = y \neq y_{1}$, we have $z_{1} \neq z$. Since $y$ lies above $y_{1}$, we must have $z_{1}$ threads before $z$, ensuring the former is lower. See Fig.~\ref{fig:absurd}.

  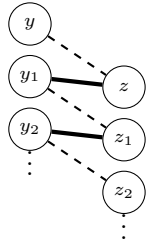
\begin{figure}[ht]
    \begin{center}
      \begin{tikzpicture}[xscale=1.25,yscale=0.7,
          roundnode/.style={circle, draw=black, inner sep=0pt, minimum size=16pt}]
        \node[roundnode] at (1,4) (y0) {$\scriptstyle y$};
        \node[roundnode] at (1,3) (y1){$\scriptstyle y_{1}$};
        \node[roundnode] at (1,2) (y2){$\scriptstyle y_{2}$};
        \node at (1,1.5) (yd) {$ \vdots$};
        \node[roundnode] at (2,2.8) (z0){$\scriptstyle z$};
        \node[roundnode] at (2,1.8) (z1) {$\scriptstyle z_{1}$};
        \node[roundnode] at (2,0.8) (z2) {$\scriptstyle z_{2}$};
        \node at (2,0.3) (zd) {$ \vdots$};
        \draw[ultra thick] (y1) -- (z0);
        \draw[ultra thick] (y2) -- (z1);
        \draw[thick,dashed](y0) -- (z0); 
        \draw[thick,dashed](y1) -- (z1); 
        \draw[thick,dashed](y2) -- (z2); 
      \end{tikzpicture}
      \caption{\label{fig:absurd} An illustration of the infinite descending sequence of cells constructed from the matchings $M = \Mtch_{\thd}(U)$ (thick) and $M^* = \Mtch_{\thd}(U \setminus \{x\})$ (dotted).}
    \end{center}
  \end{figure}

  Continuing, we may set $y_2 = M(z_1) \neq M(z) = y_1$. Since threading selects the lowest available cell, we must have $M(z_1) = y_{2}$ below $M^*(z_1) = y_1$. Since $\plength_{M^*}(z_1) = c-1$ and $z_1$ is to the right of $y$, we must have $\plength_{M}(z_1) \geq c-1$ by choice of $y$. Therefore $\plength_{M}(y_{2}) = \plength_{M}(z_1) \geq c-1$, and so again by the choice of $y$ as the highest in its column and since $y_{2}$ is lower, we must have $\plength_{M^*}(y_{2}) = c-1$. Therefore we may take $z_{2}$ such that $M^*(z_{2}) = y_{2}$. Since $M^*(z_{1}) = y_{1} \neq y_{2}$, we have $z_{2} \neq z_1$. Since $y_1$ lies above $y_{2}$, we must have $z_{2}$ threads before $z_1$, ensuring the former is lower. Iterating, we have an unbounded sequence, contradicting the finiteness of diagrams. Thus, Eq.~\ref{e:path-lengths-after-cell-removed} holds and the lemma is proved.
\end{proof}

Using Corollary~\ref{cor:weak-kohnert-thread-decomposition}, we have the following useful consequence of Lemmas~\ref{lem:generic-Kohnert-removable-cell-highest-2} and \ref{lem:generic-Kohnert-removable-cell-highest} that provides the key to our main result on rectification.

\begin{corollary}
  Let $U$ be a weak Kohnert diagram. Let $x$ be the highest removable cell of $U$ in its column. Define $V = \rectify(U)$, and let $y$ in $V$ be the final cell to move under the rectification of $U$. Then
  \begin{enumerate}
  \item $\thread(V \setminus \{y\}) = \thread(U \setminus \{x\})$;
  \item $\thread(V) = \thread(U \setminus \{x\}) + \comp{e}_j$, where $j = \Lbl_{\Mtch_{\thd}V}(y)$;
  \item $y$ is the rightmost cell of its path in $\Mtch_{\thd}(V)$, and is the highest such cell in its column.
  \end{enumerate}
  \label{cor:insertion-rectify1box}
\end{corollary}



Finally, we can proof Theorem~\ref{thm:insertion-top}, thereby also proving Theorem~\ref{thm:RSKD} via rectification for $k$ sufficiently large.

\begin{theorem}
  Let $\comp{a}$ be a weak composition, and set $\ell = \max_i\{a_i>0\}$. For every positive integer $\ell \leq k \leq n$, the map $\ins_{\infty}$ induces a weight preserving bijection
  \begin{equation}
    \KD(\comp{a}) \times \KD(\comp{e}_k) \stackrel{\sim}{\longrightarrow} \KDs(\comp{a},k).
        \label{e:insert-top}
  \end{equation}
  In particular, Theorem~\ref{thm:RSKD} is proved for $k\geq\ell$.
  \label{thm:insertion-top}
\end{theorem}

\begin{proof}
Let $T\in \KD(\comp{a})$, and let $j \leq k$ be a positive integer. Let $m = \max_i\{a_i\}$ be the rightmost occupied column of $T$. We consider 
$$
\ins_{\infty}(T,j) = \rectify(T \sqcup \{(m+1,j)\}).
$$
Since $T$ has no cell in column $m+1$, the diagram $T \sqcup \{(m+1,j)\}$ and the cell in position $(m+1,j)$ satisfy the hypothesis of Corollary~\ref{cor:insertion-rectify1box}, and so we have
$$
\thread(\ins_{\infty}(T,j)) = \thread(T) + \comp{e}_{j'},
$$
where since rectification paths move down and left, $j'$ is bounded from above by
$$
j' \leq \max\{j, \row(u) \mid u \in T \} \leq \max\{k, i \mid a_i > 0\} = k.
$$
By Lemma~\ref{lem:3.7}, we have $\thread(T) \preceq \comp{a}$, and so $\ins_{\infty}(T,j) \in \KDs(\comp{a},k)$.

Now let $y$ be the final cell to move left under rectification. 
Since Kohnert moves preserve column-weights, it follows from
Corollary~\ref{cor:insertion-rectify1box}(1)
that $\col(y)$ is the added column of $\ins_{\infty}(T,j)$. 
We may uniquely determine $y$ among the cells of $\ins_{\infty}(T,j)$ in its column using Corollary~\ref{cor:insertion-rectify1box}(3). Thus, we may recover $y$ from $\ins_{\infty}(T,j)$.

If $\ins_{\infty}(T,j) = T \sqcup \{(m+1,j)\}$, then $y$ is in position $(m+1,j)$, and we may recover $(T,j)$ from $\ins_{\infty}(T,j)$. Otherwise, suppose we have
$$
T \sqcup \{(m+1,j)\} = U_0 \stackrel{\rect}{\mapsto} U_1 \stackrel{\rect}{\mapsto} \cdots \stackrel{\rect}{\mapsto} U_p = \ins_{\infty}(T,j),
$$
for some $p>0$. At each rectification step, the cell that has moved is uniquely determined by Corollary~\ref{cor:insertion-rectify1box}, provided we know the column to which it belongs. 
Having recovered the cell $y$ from the final step of the rectification, may subsequently invert each step until we arrive at a diagram $U_0 = T \sqcup \{(m+1,j)\}$, at which point we recover the tuple $(T,j)$. Thus this process is reversible, and so the map is injective.

For every diagram $V \in \KDs(\comp{a},k)$, we have $V \in \KD(\comp{b},j)$, for some $\comp{b} \in \lswap(\comp{a})$ and some $j \leq k$. 
Consider the thread decomposition $M = \Mtch_{\thd}(V)$ on $V$, and let $y$ be the cell in column $b_j + 1$ threaded to the cell in position $(j,1)$.
Restricting to $V \setminus \{y\}$ gives a matching sequence with anchor weight $\comp{b}$, and so by Theorem~\ref{thm:TFAE_kohnert}, we have $V \setminus \{y\} \in \KD(\comp{b}) \subseteq \KD(\comp{a})$. Given the pair $(V,y)$, let the $(U,x)$ be the pair obtained from iteratively applying the following procedure.

If $\col(y) = m+1$, then set $U=V$ and $x=y$.

Otherwise, let $x' \in U$ be the highest cell in the column of $y$ with the property that removing it makes the diagram $U \setminus \{x'\}$ a generic Kohnert diagram. The cell $x'$ certainly exists since $y$ has the same property, and the choice of $x'$ insists that the position $(\col(y)+1,\row(x'))$ is not occupied by $U$.
  We then get a new tuple $(U',\comp{y'})$, where $U'$ is obtained from $U$ by pushing the cell $x'$ to the right in position $(\col(y)+1,\row(x'))$ and where $y' \in U'$ is the cell in position $(\col(y)+1,\row(x'))$.

Since $V \in \KDs(\comp{a},k)$ and the above algorithm pushes cells only right, we have 
$$
\row(x) \leq \max\{k,i \mid a_i > 0\} = k.
$$
Note that $x \in U$ is necessarily the only cell in its column, so the diagram $U$ and the cell $x \in U$ satisfy the hypothesis of Corollary~\ref{cor:insertion-rectify1box}. 
In light of Corollary~\ref{cor:insertion-rectify1box}, we have, by the above algorithm's design, $V = \rectify(U) = \ins_{\infty}(U \setminus \{x\},\row(x))$.
By Corollary~\ref{cor:insertion-rectify1box}(1) and Lemma~\ref{lem:3.7}, we have
$$
\thread(U \setminus \{x\}) = \thread(V \setminus \{y\}) \preceq \comp{b} \preceq \comp{a},
$$
and so $U \setminus \{x\} \in \KD(\comp{a})$. 
Thus, $V$ is in the image of $\KD(\comp{a}) \times [k]$ under $\ins_{\infty}$.
By Lemma~\ref{lem:col-added}, we know $\col(y) = b_j + 1$ is the added column of $V$. 
Hence, the map $\ins_{\infty} : \KD(\comp{a}) \times [k] \longrightarrow \KDs(\comp{a},k)$ is surjective.
\end{proof}

%
\section{Key stratification}
%
\label{sec:stratify}

We now complete the proof of the existence of a weight-preserving bijection as asserted in Theorem~\ref{thm:RSKD}. Given a weak composition $\comp{a}$, the target spaces $\KDs(\comp{a},k)$ for each positive integer $k$ form a nested sequence of spaces of Kohnert diagrams,
\[ \cdots \supset \KDs(\comp{a},k+1) \supset \KDs(\comp{a},k) \supset \KDs(\comp{a},k-1) \supset \cdots \supset \KDs(\comp{a},1) . \]
We use this to \emph{stratify} the target space of the desired bijection by
\begin{equation}
  \KDstr = \KDs(\comp{a},k) \setminus \KDs(\comp{a},k-1) .
  \label{e:KDstr}
\end{equation}
In Section~\ref{sec:stratify-define}, for each integer $k > 1$, we define an injective map
\begin{equation}
  \partial_{\comp{a},k} \ : \ \KDstr \longrightarrow \KD(\comp{a})
  \label{e:stratum}
\end{equation}
satisfying $\wt(U) = \wt(\partial_{\comp{a},k}(U)) + \comp{e}_k$, and we show how these maps together with bijections for top and bottom insertion prove Theorem~\ref{thm:RSKD}. In Section~\ref{sec:stratify-range}, we prove the image of $\partial_{\comp{a},k}$ is indeed a Kohnert diagram for $\comp{a}$, and in Section~\ref{sec:stratify-inject} we prove $\partial_{\comp{a},k}$ is injective.

\subsection{Stratum maps}
\label{sec:stratify-define}

Up to this point, given a Kohnert diagram $T$, we have been primarily interested in the matching $\Mtch_{\thread}(T)$ corresponding to the thread decomposition $\thread(T)$. To study the stratum maps, we consider also matchings corresponding to the \emph{Kohnert labeling of $T$ with respect to $\comp{a}$} \cite[Definition~2.5]{AS18}.

\begin{definition}[\cite{AS18}]
  Let $\comp{a}$ be a weak composition and $T \in \KD(\comp{a})$. The \emph{Kohnert labeling of $T$ with respect to $\comp{a}$}, denoted by $\Lbl_{\comp{a}}(T)$, assigns labels to cells of $T$ as follows. Assuming all columns right of column $j$ have been labeled, assign labels $\{i \mid a_i \geq j\}$ to cells of column $j$ from bottom to top by choosing at each cell the smallest label $i$ such that the $i$ in column $j+1$, if it exists, is weakly lower.

  The \emph{Kohnert matching of $T$ with respect to $\comp{a}$}, denoted by $\Mtch_{\comp{a}}(T)$, is the matching sequence on $T$ defined by $x$ matching to $y$ for cells $x \in T$ in column $i+1$ and $y \in T$ in column $i$ if and only if $\Lbl_{\comp{a}}(x) = \Lbl_{\comp{a}}(y)$. 
  \label{def:kohnert-label}
\end{definition}

Assaf and Searles \cite[Theorem~2.8]{AS18} prove this is well-defined and use it to define and establish basic properties of \emph{Kohnert tableaux} \cite[Definition~2.3]{AS18}.

\begin{example}
  Consider the four matching sequences for the generic Kohnert diagram from Example~\ref{ex:matching}. Each is the Kohnert labeling $\Lbl_{\comp{b}}$ for the weak composition $\comp{b}$ given below it corresponding to the weight of the matching sequence.
\end{example}

\begin{proposition}
  For $T \in \KD(\comp{a})$, we have $\wt(\Mtch_{\comp{a}}) \preceq \comp{a}$.
  \label{prop:label-wt}
\end{proposition}

\begin{proof}
  We may regard the Kohnert matching sequence $\Mtch_{\comp{a}}$ as a labeling $L$ where $L(x)=k$ whenever the cell on the component of $x$ in the first column is in row $k$. If this labeling agrees with the Kohnert labeling $\Lbl_{\comp{a}}$, then $\wt(\Mtch_{\comp{a}}) = \comp{a}$. Otherwise, consider the key diagram $\kd_{\wt(\Mtch_{\comp{a}})}$ with the labeling $L'$ where $L'(x)$ is the label of the cell in the first column of $T$ in the same row as $x$. By \cite[Theorem~2.8]{AS18}, this is the Kohnert labeling of $\kd_{\wt(\Mtch_{\comp{a}})}$ with respect to $\comp{a}$, and so $\kd_{\wt(\Mtch_{\comp{a}})} \in\KD(\comp{a})$. Thus by Proposition~\ref{prop:lswap}, $\wt(\Mtch_{\comp{a}}) \preceq \comp{a}$.
\end{proof}

The matching sequence on $T$ corresponding to the thread decomposition of $T$ is a special case of a Kohnert matching on $T$.

\begin{proposition}
  For $T$ a generic Kohnert diagram, $\Mtch_{\thd}(T) = \Mtch_{\thread(T)}(T)$.
  \label{prop:matching-thread-decomposition-specializes-Kohnert-labeling}
\end{proposition}

\begin{proof}
  By Lemma~\ref{lem:3.7}, we have $T \in \KD(\thread(T))$, so we may consider the Kohnert matching of $T$ with respect to $\thread(T)$. Note $\Mtch_{\thd}(T)$ and $\Mtch_{\thread(T)}(T)$ have the same weight $\thread(T)$, and so the same number of threads. Let $i \geq 1$ be the smallest index for which $\thread(T)_i > 0$. We will show that the components for $\Mtch_{\thd}(T)$ and $\Mtch_{\thread(T)}(T)$ anchored in row $i$ coincide, from which the result follows by induction on the number of threads since the base case of one component is trivial.

  Consider the cell in column $\thread(T)_i$ anchored in row $i$. Since $i$ is minimal, this is the lowest cell of $T$ in column $\thread(T)_i$ for $\Mtch_{\thread(T)}(T)$. If for $\Mtch_{\thd}(T)$ there is a cell below this, then that cell must belong to an earlier thread, and so it will be threaded before the cells anchored at row $i$. However, since it is lower in column $\thread(T)_i$, it will always take weakly lower cells, contradicting the minimality of $i$. Thus both threads begin with the lowest cell of $T$ in column $\thread(T)_i$.

  We may assume the components anchored in row $i$ for $\Mtch_{\thd}(T)$ and $\Mtch_{\thread(T)}(T)$ agree weakly right of column $c+1 \leq \thread(T)_i$. Then for $\Mtch_{\thread(T)}(T)$, the chosen cell in column $c$ will be the lowest cell of $T$ that sits weakly above the cell labeled $i$ in column $c+1$. In order for $\Mtch_{\thd}(T)$ not to choose the same cell, there must be an earlier thread that takes the cell chosen by $\Mtch_{\thread(T)}(T)$, but then again this thread will have priority of the one anchored in row $i$ and will continue to select lower cells, once again contradicting the minimality of $i$. Therefore the components anchored in row $i$ for $\Mtch_{\thd}(T)$ and $\Mtch_{\thread(T)}(T)$ coincide, and the result follows.
\end{proof}


Recall Lemma~\ref{lem:col-added} associates to each $U \in \KDs(\comp{a},n)$ a unique added column with respect to $\comp{a}$, called the added column of $U$ (Definition~\ref{def:added-column}).

\begin{lemma}
  For a weak composition $\comp{a}$ and a positive integer $k>1$, let $U \in \KDstr$, and let $c$ be the added column of $U$. Then
  \begin{enumerate}
  \item the key diagram $\kd_{\thread(U)}$ has a cell $y$ in position $(c,k)$, and
  \item the diagram $\kd_{\thread(U)} \setminus \{y\}$ is a Kohnert diagram of $\comp{a}$.
  \end{enumerate}
  \label{lem:excision}
\end{lemma}

\begin{proof}
  By Lemma~\ref{lem:3.7}, since $U \in \KDs(\comp{a},k)$, we have $\thread(U) \preceq \comp{b} + \comp{e}_{k}$ for some weak composition $\comp{b} \in \lswap(\comp{a})$, and, by Lemma~\ref{lem:col-added}, $c = b_k+1$. Thus $\kd_{\thread(U)} \in \KD(\comp{b}+\comp{e}_k)$, and we may consider the Kohnert labeling $L = \Lbl_{\comp{b}+\comp{e}_k}(\kd_{\thread(U)})$. Let $y$ be the cell in column $c$ of $\kd_{\thread(U)}$ with $L(y)=k$. Note $y$ is the rightmost cell with label $k$. Thus we may defined a matching $M$ on $\kd_{\thread(U)} \setminus \{y\}$ by matching cells in adjacent columns if and only if they have the same label under $L$, and $\wt(M) = \comp{b}$. Thus, by Theorem~\ref{thm:TFAE_kohnert} and Proposition~\ref{prop:lswap}, we have $\kd_{\thread(U)} \setminus \{y\} \in \KD(\comp{b}) \subset \KD(\comp{a})$. To prove the lemma, we have only to show  $y$ is in row $k$. 

  Let $\comp{b'} = \thread(\kd_{\thread(U)} \setminus \{y\})$. Since $\kd_{\thread(U)} \setminus \{y\} \in \KD(\comp{a})$, Lemma~\ref{lem:3.7} ensures $\comp{b'} \preceq \comp{a}$. Let $k'$ be the row index of $y$. Since $L(y) = k$, we know by \cite[Theorem~2.8]{AS18} that $k' \leq k$. Since $\kd_{\thread(U)}$ is a key diagram, for every position $(i,j)$ occupied by $\kd_{\thread(U)}$, the position $(i-1,j)$ is occupied by $\kd_{\thread(U)}$ as well. Therefore, considering the thread decomposition of $\kd_{\thread(U)} \setminus \{y\}$, we have $b'_{k'} = c - 1$, where the path in the thread decomposition that is anchored at row $k'$, if it exists, only occupies row $k'$. Hence, either by appending $y$ to the end of the path anchored at row $k'$ (if $c > 1$) or having $y$ form its own path (if $c = 1$), we obtain from $\Mtch_{\thd}(\kd_{\thread(U)} \setminus \{y\})$ a matching sequence on $\kd_{\thread(U)}$ with weight $\comp{b'} + \comp{e}_k$. So by Theorem~\ref{thm:TFAE_kohnert}, we have
  $U \in \KD(\comp{b'} + \comp{e}_{k'})$. If $k' < k$, then $U \in \KDs(\comp{a},k-1)$, which directly contradicts how $U$ is picked from $\KDstr$. Thus, $k' = k$, as desired.
\end{proof}

Lemma~\ref{lem:excision} motivates the following notation.

\begin{definition}
  For $U \in \KDstr$, the \emph{excised weight of $U$} (with respect to $\comp{a}$ and $k$) is the weak composition $\thread(\kd_{\thread(U)} \setminus \{y\})$, where $y$ is the cell of $\kd_{\thread(U)}$ in position $(c,k)$ for $c$ the added column of $U$.
  \label{def:excision}
\end{definition}

\begin{figure}[ht]
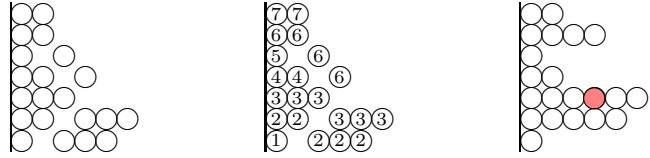

  \begin{displaymath}
    \arraycolsep=3\cellsize
    \begin{array}{ccc}
    \vline\cirtab{%
      \ & \ \\
      \ & \ \\
      \ & & \ \\
      \ & \ & & \ \\
      \ & \ & \ \\
      \ & \ & & \ & \ & \ \\
      \ & & \ & \ & \ } &
    \vline\cirtab{%
      7 & 7 \\
      6 & 6 \\
      5 & & 6 \\
      4 & 4 & & 6 \\
      3 & 3 & 3 \\
      2 & 2 & & 3 & 3 & 3 \\
      1 & & 2 & 2 & 2 } &
    \vline\cirtab{%
      \ & \ \\
      \ & \ & \ & \ \\
      \ \\
      \ & \ \\
      \ & \ & \ & \cball{red} & \ & \ \\
      \ & \ & \ & \ & \ \\
      \ }
    \end{array}
  \end{displaymath}
  \caption{\label{fig:excised-wt}A generic Kohnert diagram $U$ (left) in $\KDstr$ for $\comp{a} = (1,5,2,1,2,6,3)$ and $k=3$ with excised column $c=4$, its thread decomposition (middle), and the key diagram $\kd_{\thread(U)}$ (right) with the cell in position $(c,k)$ indicated.}
\end{figure}

\begin{example}
  Consider the weak composition $\comp{a} = (1,5,2,1,2,6,3)$, and the generic Kohnert diagram $U$ on the left side of Fig.~\ref{fig:excised-wt}. The thread weight of $U$ is $\thread(U) = (1,5,6,2,1,4,2)$, and so $U$ is a Kohnert diagram for $\comp{b}'+\comp{e}_k$ for $\comp{b}=(2,5,3,2,1,6,2) \preceq \comp{a}$ and $k=3$, and so $U \in \KDstr$. Comparing column weights, the added column of $U$ with respect to $\comp{a}$ is $c=4$. The key diagram $\kd_{\thread(U)}$ shown on the right side of Fig.~\ref{fig:excised-wt} indeed has a cell $(\cball{red})$ in position $(4,3)$, and removing it gives a diagram with thread weight $\comp{b} = (1,5,3,2,1,6,2)$, which is the excised weight of $U$.
  \label{ex:excised-wt}
\end{example}

With notation as in Lemma~\ref{lem:excision} and $\comp{b} = \thread(\kd_{\thread(U)} \setminus \{y\})$ the excised weight of $U$, from Lemmas~\ref{lem:excision}(2) and \ref{lem:3.7}, we know $\comp{b} \preceq \comp{a}$ and $\kd_{\thread(U)} \in \KD(\comp{b} + \comp{e}_k)$. Thus, $U \in \KD(\comp{b} + \comp{e}_k)$, and we may consider the Kohnert labeling $\Lbl_{\comp{b} + \comp{e}_k}(U)$.  Through this labeling we shall partition $U$ into sub-diagrams, which we then transform and glue back together to obtain the image of $U$ under the map $\partial_{\comp{a},k}$. 

The simplicity of the statement of the following lemma belies its utility. 

\begin{lemma}
  For $U \in \KDstr$ and $\comp{b}$ the excised weight of $U$, each cell $x$ in the first column of $U$ has $\Lbl_{\comp{b}+\comp{e}_k}(x)$ equal to the row index of $x$. In particular, $\wt(\Mtch_{\comp{b}+\comp{e}_k}(U)) = \comp{b} + \comp{e}_k$.
  \label{lem:L-minimal-labeling}
\end{lemma}

\begin{proof}
  Recall $\comp{b} = \thread(\kd_{\thread(U)} \setminus \{y\})$ for $y$ the cell in position $(c,k)$ of $U$, where $c$ is the added column of $U$. Thus, for all $j \neq k$, we have $b_j > 0$ if and only if $\thread(U)_j > 0$. On the other hand, since $y$ is in row $k$, the key diagram $\kd_{\thread(U)}$ occupies positions in row $k$. Therefore, $\thread(U)_k > 1$, and in particular, $U$ has a cell at position $(1,k)$. By construction, the set of labels used in $\Lbl_{\comp{b}+\comp{e}_k}$ coincides with the row indices of the cells in column $1$ of $U$. Since each cell $x$ has $\Lbl_{\comp{b}+\comp{e}_k}(U)$ at least equal to the row index of $x$, each cell in $U$ in column $1$ must have its row as its label.
\end{proof}

In particular, by Lemma~\ref{lem:L-minimal-labeling}, $U$ has a cell in column $1$ of row $k$, and this cell belongs to a component of $\Mtch_{\comp{b} + \comp{e}_k}(U)$ of length $c$.

\begin{definition}
  For $U \in \KDstr$, let $c$ denote the added column of $U$ and $\comp{b}$ the excised weight of $U$. Define sets
  \begin{eqnarray}
    U^+ &=& \{ x \in U \mid \plength_{\Mtch_{\comp{b} + \comp{e}_k}}(x) \geq c \}, \\
    U^- &=& \{ x \in U \mid \plength_{\Mtch_{\comp{b} + \comp{e}_k}}(x) < c \},
  \end{eqnarray}
  and let $U^+_* = U^+ \setminus \{w\}$, where $w$ is the cell in column $1$ of row $k$.
\end{definition}

\begin{figure}[ht]
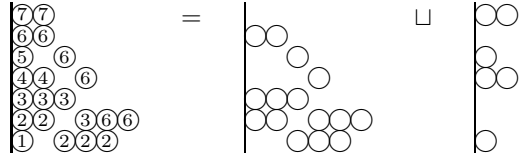

  \begin{displaymath}
    \arraycolsep=\cellsize
    \begin{array}{ccccc}
    \vline\cirtab{%
      7 & 7 \\
      6 & 6 \\
      5 & & 6 \\
      4 & 4 & & 6 \\
      3 & 3 & 3 \\
      2 & 2 & & 3 & 6 & 6 \\
      1 & & 2 & 2 & 2 } & = &
    \vline\cirtab{%
        &   \\
      \ & \ \\
        & & \ \\
        &   & & \ \\
      \ & \ & \ \\
      \ & \ & & \ & \ & \ \\
        & & \ & \ & \ } & \sqcup & 
    \vline\cirtab{%
      \ & \ \\
        &   \\
      \ & &   \\
      \ & \ & &   \\
        &   &   \\
      & \\
      \ & &  &  & } 
    \end{array}
  \end{displaymath}
  \caption{\label{fig:stratum-partition}The partitioning of the generic Kohnert diagram $U$ (left) in $\KDstr$ for $\comp{a} = (1,5,2,1,2,6,3)$ and $k=3$ using the Kohnert labeling to obtain $U^+$ (middle) and $U^-$ (right).}
\end{figure}

\begin{example}
  Continuing with Ex.~\ref{ex:excised-wt}, we take the Kohnert labeling the diagram $U$ with respect to $\comp{b}+\comp{e}_k = (1,5,4,2,1,6,2)$, where $\comp{b}$ is the excised weight of $U$ and $k=3$, as shown on the left side of Fig.~\ref{fig:stratum-partition}. The added column $c=4$ dictates which labels are included for each half of the partitioning, giving the decomposition shown on the right side of Fig.~\ref{fig:stratum-partition}.
  \label{ex:stratum-partition}
\end{example}

Notice $U = U^{+} \sqcup U^{-}$. We can now define the map $\partial_{\comp{a},k}$.

\begin{definition}
  Given a weak composition $\comp{a}$ and an integer $k>1$, the \emph{$k$th stratum map of $\comp{a}$}, denoted by $\partial_{\comp{a},k}$, acts on $U \in \KDstr$ by
  \begin{equation}
    \partial_{\comp{a},k}(U) = \rectify(U^+_*) \cup U^- .
  \end{equation}
\end{definition}

\begin{figure}[ht]
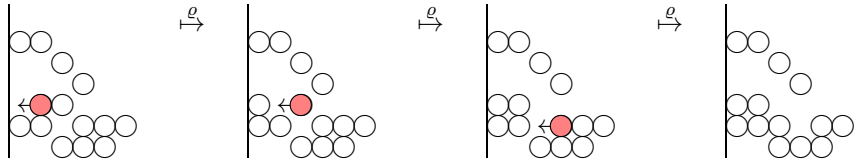

  \begin{displaymath}
    \arraycolsep=\cellsize
    \begin{array}{ccccccc}
    \vline\cirtab{%
        &   \\
      \ & \ \\
        & & \ \\
        &   & & \ \\
        & \mball & \ \\
      \ & \ & & \ & \ & \ \\
        & & \ & \ & \ } & \stackrel{\rect}{\mapsto} &
    \vline\cirtab{%
        &   \\
      \ & \ \\
        & & \ \\
        &   & & \ \\
      \  &  & \mball \\
      \ & \ & & \ & \ & \ \\
        & & \ & \ & \ } & \stackrel{\rect}{\mapsto} &
    \vline\cirtab{%
        &   \\
      \ & \ \\
        & & \ \\
        &   & & \ \\
      \  &  \ \\
      \ & \ & & \mball & \ & \ \\
        & & \ & \ & \ } & \stackrel{\rect}{\mapsto} &
    \vline\cirtab{%
        &   \\
      \ & \ \\
        & & \ \\
        &   & & \ \\
      \  &  \ \\
      \ & \ & \ &  & \ & \ \\
        & & \ & \ & \ } 
    \end{array}
  \end{displaymath}
  \caption{\label{fig:stratum-map}The rectification of the diagram $U^+_*$ obtained from $U^+$ by removing the cell in position $(1,3)$.}
\end{figure}

\begin{example}
  Continuing with Ex.~\ref{ex:stratum-partition}, we remove the cell in position $(1,k)$ of $U^+$ and rectify to obtain the diagram on the right side of Fig.~\ref{fig:stratum-map}. Notice this is disjoint from $U^-$. Their union is the result of applying the stratum map $\partial_{\comp{a},k}$ to the diagram $U$. Observe this is indeed a Kohnert diagram for $\comp{a} = (1,5,2,1,2,6,3)$.
  \label{ex:stratum-map}
\end{example}

The map $\partial_{\comp{a},k}$ is well-defined with $\wt(U) = \wt(\partial_{\comp{a},k}(U)) + \comp{e}_k$ by Lemma~\ref{lem:L-minimal-labeling}. In Section~\ref{sec:stratify-range}, we study properties of rectification to prove the following.

\begin{theorem}
  The diagram $\partial_{\comp{a},k}(U)$ is a Kohnert diagram of $\comp{a}$.
  \label{thm:stratum-image}
\end{theorem}

Therefore $\partial_{\comp{a},k}$ is indeed a map from the stratum $\KDstr$ into the Kohnert space $\KD(\comp{a})$. In Section~\ref{sec:stratify-inject}, we prove this map is reversible and thus injective.

\begin{theorem}
  For $U \in \KDstr$, the diagram $U$ is the unique pre-image of the diagram $\partial_{\comp{a},k}(U)$ under the $k$th stratum map $\partial_{\comp{a},k}$ of $\comp{a}$. That is, the maps
  $$
  \partial_{\comp{a},k}: \KDstr \longrightarrow \KD(\comp{a})
  $$
  are injective.
  \label{thm:stratum-inject}
\end{theorem}


We now prove Theorem~\ref{thm:RSKD}, there exists a weight-preserving bijection
\[ \KD(\comp{a}) \times \KD(\comp{e}_k) \stackrel{\sim}{\longrightarrow} \bigcup_{\substack{ \comp{b} \preceq \comp{a} \\ 1 \leq j \leq k }} \KD(\comp{b} + \comp{e}_j), \]
using the injectivity of the stratum maps together with the bijectivity of the top and bottom insertion maps.

In the proof, we extend our notion of row-weights on sets of Kohnert diagrams to Cartesian products of such sets, e.g. if we have $(S,T) \in \KD(\comp{a}) \times \KD(\comp{b})$ for some weak compositions $\comp{a}$ and $\comp{b}$, then we define $\wt(S,T) = \wt(S) + \wt(T)$.
In this way, we can discuss the row-weight spaces of said sets, e.g. given another weak composition $\comp{d}$, the set $[\KD(\comp{a}) \times \KD(\comp{b})]_{\comp{d}}$
is the set of all tuples $(S,T)$ with $S \in \KD(\comp{a})$ and $T \in \KD(\comp{b})$ such that $\wt(S) + \wt(T) = \comp{d}$.



\begin{proof}[Proof of Theorem~\ref{thm:RSKD}]
Let $\comp{a}$ be a weak composition
It suffices to show
\begin{equation}
    \#[\KD(\comp{a}) \times \KD(\comp{e}_k)]_{\comp{b}} = \#[\KDs(\comp{a},k)]_{\comp{b}},
        \label{e:bijection-per-wt}
\end{equation}
for every positive integer $k \leq n$ and every weak composition $\comp{b}$.
To start, we fix a weak composition $\comp{b}$.
We immediately get 
$$\#[\KD(\comp{a}) \times \KD(\comp{e}_n)]_{\comp{b}} = \#[\KDs(\comp{a},n)]_{\comp{b}}
$$
via the weight-preserving bijection in 
Eq.~\eqref{e:insert-top} of
Theorem~\ref{thm:insertion-top}.
We will now show
$\#[\KDs(\comp{a},k)]_{\comp{b}} = \#[\KD(\comp{a}) \times \KD(\comp{e}_k)]_{\comp{b}}$
for every positive integer $k \leq n$.
By stratification, we have
$$
\KDs(\comp{a},n) = \KDs(\comp{a},1) \sqcup \bigsqcup_{1 < k \leq n} \KDstr.
$$
The weight-preserving bijection in Eq.~\eqref{e:insert-bot} of Theorem~\ref{thm:insertion-bottom} gives us 
\begin{equation}
    \#[\KDs(\comp{a},1)]_{\comp{b}} = \#[\KD(\comp{a}) \times \KD(\comp{e}_1)]_{\comp{b}} = \#[\KD(\comp{a}) \times \{\kd_{\comp{e}_1}\}]_{\comp{b}}.
        \label{e:bijection-per-wt-bot}
\end{equation}
On the other hand, 
for every $1 < k \leq n$,
the 
stratum map 
$\partial_{\comp{a},k}$
which is injective by Theorem~\ref{thm:stratum-inject}
and excises a cell in row $k$, gives us
\begin{equation}
    \#[\KDstr]_{\comp{b}} \leq \#[\KD(\comp{a})]_{\comp{b} - \comp{e}_k} = \#[\KD(\comp{a}) \times \{\kd_{\comp{e}_k}\}]_{\comp{b}}
        \label{e:injection-per-wt-strat}
\end{equation}

Eq.~\eqref{e:bijection-per-wt-bot}~and~\eqref{e:injection-per-wt-strat} together imply
\[
\begin{array}{r@{}l}
    \#[\KDs(\comp{a},n)]_{\comp{b}} &{}= \#[\KDs(\comp{a},1)]_{\comp{b}} + \sum_{1 < k \leq n} \#[\KDstr]_{\comp{b}} \\
     &{}\leq \#[\KD(\comp{a}) \times \{\kd_{\comp{e}_1}\}]_{\comp{b}} + \sum_{1 < k \leq n} \#[\KD(\comp{a}) \times \{\kd_{\comp{e}_k}\}]_{\comp{b}}\\
    &{}\leq \#[\KD(\comp{a}) \times \KD(\comp{e}_n)]_{\comp{b}}.
\end{array}
\]
Thus, since
$
\#[\KD(\comp{a}) \times \KD(\comp{e}_n)]_{\comp{b}} = \#[\KDs(\comp{a},n)]_{\comp{b}},
$
we have for each $1 < k \leq n$, the inequality in Eq.~\eqref{e:injection-per-wt-strat} is in fact an equality.
So for each $1 < k \leq n$, we have
\[
\begin{array}{r@{}l}
    \#[\KDs(\comp{a},k)]_{\comp{b}} &{}= \#[\KDs(\comp{a},1)]_{\comp{b}} + \sum_{1 < j \leq k} \#[\KDstr]_{\comp{b}} \\
     &{}= \#[\KD(\comp{a}) \times \{\kd_{\comp{e}_1}\}]_{\comp{b}}
     + \sum_{1 < j \leq k} \#[\KD(\comp{a}) \times \{\kd_{\comp{e}_j}\}]_{\comp{b}} \\
    &{}= \#[\KD(\comp{a}) \times \KD(\comp{e}_k)]_{\comp{b}}.
\end{array}
\]
In conclusion, Eq.~\eqref{e:bijection-per-wt} holds for every positive integer $k \leq n$ and every weak composition $\comp{b}$, as desired.
\end{proof}

\subsection{Image of the stratum maps}
\label{sec:stratify-range}

To prove the image of $U$ under $\partial_{\comp{a},k}$ is a Kohnert diagram in $\KD(\comp{a})$, we first show $\rectify(U^+_*)$ and $U^-$ are disjoint, from which it follows that $\partial_{\comp{a},k}(U)$ is a generic Kohnert diagram. To this end, we study the rectification process in detail. 

\begin{lemma}
  Let $U \in \KDstr$ with added column $c$. The diagram $\rect^{i-1}(U^+_*)$ is not a generic Kohnert diagram if and only if $i < c$. Moreover, for $y_i$ the cell that moves from column $i+1$ of $\rect^{i-1}(U^+_*)$ to column $i$ of $\rect^i(U^+_*)$, we have 
  \begin{enumerate}
  \item each $y_i$ is weakly below the cell of $\setW$ in column $i+1$, and
  \item $k \geq \row(y_1) \geq \row(y_2) \geq \cdots \geq \row(y_{c-1})$.
  \end{enumerate}
\label{lem:rectification}
\end{lemma}

\begin{proof}
  If $c = 1$, the result is trivial, so we assume $c > 1$. Since $U^+$ consists of all components of $\Mtch_{\comp{b}+\comp{e}_k}(U)$ of length $\geq c$, $U^+$ has the same number of cells for each column weakly left of column $c$, and so for each column $1 < c' \leq c$, we have
  \begin{equation}
    \matchable_{U^+}(c',1) = 0.
    \label{eq:perfect}
  \end{equation}
  Let $\setW = \{w_1,\ldots,w_c\}$ with $w_j$ the cell in column $j$.
  
  We proceed by induction on $i$, with base case of $i = 1$ in which we consider $U^+_*$ itself. By Eq.~\eqref{eq:perfect}, since $U^+_*$ has one fewer cell in the first column than $U^+$, we must have $\matchable_{U^+_*}(2,1) = -1$. In particular, $\rect^0(U^+_*) = U^+_*$ is not a generic Kohnert diagram. Letting $y_1$ denote the cells that moves under $\rect$, it must do so from column $2$ of $U^+_*$ to column $1$ of $\rect(U^+_*)$. Among the cells of $U^+_*$ above $w_2$, $\Mtch_{\comp{b}+\comp{e}_k}$ gives a matching from the cells in column $2$ to the cells in column $1$, and so $\matchable_{U^+_*}(2,j) \geq 0$ for all rows $j$ above the row of $w_2$. By definition of $\rect$, the cell $y_1$ must lie weakly below $w_2$ and so weakly below row $k$ as well. This proves all statements for $i=1$.

  For $1 \leq i < c$, we now assume the result for all $h \leq i$, so that $\rect^{h-1}(U^+_*)$ is not a generic Kohnert diagram, giving a sequence of cells $y_1,\ldots,y_{i}$ where $y_h$ moves from column $h+1$ of $\rect^{h-1}(U^+_*)$ to column $h$ of $\rect(\rect^{h}(U^+_*))$ and satisfies (1) $y_h$ lies weakly below $w_{h+1}$ for $h \leq i$ and (2) $k \geq \row(y_1) \geq \cdots \geq \row(y_{i})$. 

  Since $\rect^{i-1}(U^+_*)$ and $U^+_*$ coincide for all columns strictly right of column $i$, and differ in column $i$ only in the presence of $y_{i-1}$ in $U^+_*$, it follows from Eq.~\eqref{eq:perfect} that $\matchable_{\rect^{i-1}(U^+_*)}(i+1,j) \geq -1$ for all rows $j$ and $\matchable_{\rect^{i-1}(U^+_*)}(i+1,j) \geq 0$ for all $j > \row(y_{i-1})$. Thus when $y_i$ moves from column $i+1$ to column $i$, we have $\matchable_{\rect^{i}(U^+_*)}(i+1,j) \geq 0$ for every row $j$. Therefore $\rect^{i}(U^+_*)$ is a generic Kohnert diagram if and only if $\matchable_{\rect^{i}(U^+_*)}(i+2,j) \geq 0$ for every row $j$. We consider two cases.
  
  First suppose $i=c-1$, so that $i+2=c+1$. Since $c>1$, the Kohnert matching sequence $\Mtch_{\comp{b}+\comp{e}_k}(U)$ matches each cell in column $c+1$ of $U^+_*$ with a cell in column $c$ of $U^+_*$. Using this, we have a matching from each cell in column $c+1$ of $\rect^{i}(U^+_*)$ with a cell in column $c$ of $U^+_*$ \emph{except} for the cell, say $x$, that was matched to $y_{c-1}$ in $U^+_*$. Since $y_{c-1}$ is weakly below $w_c$, we may match $x$ with $w_c$ to complete the matching, thereby showing $\matchable_{\rect^{i}(U^+_*)}(i+2,j) \geq 0$ for every row $j$., and so $\rect^{i}(U^+_*)$ is a generic Kohnert diagram.

  Second suppose $i<c-1$, so that $i+2 \leq c$. By Eq.~\eqref{eq:perfect}, we have $\matchable_{\rect^{i}(U^+_*)}(i+2,1) = -1$, and so $\rect^{i}(U^+_*)$ is not a generic Kohnert diagram. Thus let $y_{i+1}$ be the cell of $\rect^{i}(U^+_*)$ that moves in passing to $\rect(\rect^{i}(U^+_*))$. Now let $x$ be the cell in column $i+1$ of $U^+_*$ such that $\Mtch_{\comp{b}+\comp{e}_k}(x) = y_i$. Since $\Mtch_{\comp{b}+\comp{e}_k}$ gives a matching from column $i+2$ to column $i+1$ in $U^+_*$, it gives a matching for cells weakly above the row of $x$ from column $i+2$ to column $i+1$. In particular, $\matchable_{\rect^{i}(U^+_*)}(i+2,j) \geq 0$ for all rows $j > \row(x)$. In particular, $y_{i+1}$ is weakly below $x$, and so too weakly below $y_i$, proving (1). If $w_{i+2}$ is weakly above $x$, then $y_{i+1}$ is weakly below $w_{i+2}$ as well, proving (2). Otherwise, if $w_{i+2}$ is strictly below $x$, then, similar to the case $i=c-1$, we construct a matching for cells above $\row(w_{i+1})$ from column $i+2$ to column $i+1$ using $\Mtch_{\comp{b}+\comp{e}_k}(U)$ for every cell except $x$ and then matching $x$ to $w_{i+1}$, since $x$ is weakly below $y_i$ which is below $w_{i+1}$. Thus $\matchable_{U_i}(i+2,j) \geq 0$ for all rows $j > \row(w_{i+2})$. So $y_{i+1}$ is again weakly below $w_{i+2}$ proving (2) for this case as well. Therefore all statements hold for $i+1$, and so the result follows by induction.
\end{proof}

Lemma~\ref{lem:rectification} describes the general movement of cells of $U^+_*$ under rectification. It is also helpful to understand their labels under $\Lbl_{\comp{b}+\comp{e}_k}(U)$. The cells with label $k$ under $\Lbl_{\comp{b} + \comp{e}_k}(U)$ are of particular importance, so we introduce the following notation. Recall $b_k + 1 = c$, and so there are precisely $c$ of these cells which form a component of the matching induced by $\Lbl_{\comp{b} + \comp{e}_k}$.

\begin{definition}
  For $U \in \KDstr$ and $\comp{b}$ the excised weight of $U$, define sets
  \begin{eqnarray}
    \setW &=& \{ x \in U \mid \Lbl_{\comp{b} + \comp{e}_k}(x) = k \} , \\
    \setV & = & \{ x \in U \mid \Lbl_{\comp{b} + \comp{e}_k}(x) < k \} \\
    \setVV & = & \{ x \in U \mid \Lbl_{\comp{b} + \comp{e}_k}(x) \leq k \}.
  \end{eqnarray}
  Writing $\setW = \{w_1,\ldots,w_c\}$ with $w_i$ in column $i$, we set $\setW_* = \setW \setminus \{ w_1 \}$.
\end{definition}

Notice we have $\setW \subset U^+$ and $\setW_* \subset U^+_*$.

\begin{lemma}
  For $U \in \KDstr$, $c$ the added column of $U$, and $\comp{b}$ the excised weight of $U$, let $y_i$ be the cell that moves from column $i+1$ of $\rect^{i-1}(U^+_*)$ to column $i$ of $\rect^i(U^+_*)$. Then for $i<c$, we have $\Lbl_{\comp{b}+\comp{e}_k}(y_i) \leq k$.
  \label{lem:moved-cells-in-lower-half}
\end{lemma}

\begin{proof}
  Let $\setW = \{w_1,\ldots,w_c\}$ with $w_i$ the cell in column $i$. We claim each cell $x\in U^+ \setminus \setVV$ in column $i \leq c$ lies strictly above the cell $w_i$. Indeed, in column $c$, if some such $x$ lies below $w_c$, then since $\Lbl_{\comp{b}+\comp{e}_k}(x) > k = \Lbl_{\comp{b}+\comp{e}_k}(w_i)$, the labeling algorithm will assign $k$ to $x$ instead of $w_c$, since $x$ must also have been available and there is no cell labeled $k$ to the right. Continuing left, consider the largest column index $i<c$ such that some cell $x \in U^+ \setminus \setVV$ in column $i$ lies below $w_i$, then in order for $x$ not to have been selected, we must have $w_{i+1}$ in a row strictly above that of $x$. However, since $\plength_{\Mtch_{\comp{b}+\comp{e}_k}}(x) \geq c = \plength_{\Mtch_{\comp{b}+\comp{e}_k}}(w_i)$, there is also a cell $x'$ in column $i+1$ in a row weakly below that of $x$, and hence below $w_{i+1}$, a contradiction to the choice of $i$. Thus each $x$ in column $i \leq c$ must indeed lie above $w_i$.

  By Lemma~\ref{lem:rectification}(1), in $U$ each cell $y_i$ for $i < c$ is weakly below the cell $w_{i+1}$ in its column. Thus, from the claim, $y_i$ cannot be in $U^+ \setminus \setVV$. The result follows.
\end{proof}

One final lemma necessary to prove $\rectify(U^+_*)$ and $U^-$ are disjoint relates the thread decomposition of $\setVV$ with the Kohnert labeling used to define it. 

\begin{lemma}
  For $U \in \KDstr$ and $\comp{b}$ the excised weight of $U$, we have
  \begin{eqnarray}
    \Mtch_{\thd}(\setVV) & = & \Mtch_{\comp{b}+\comp{e}_k}(U)\big|_{\setVV} = \Mtch_{(b_1,\ldots,b_k)+\comp{e}_k} (\setVV) , \label{e:label-thread} \\
    \Mtch_{\thd}(U^+ \cap \setVV) & = & \Mtch_{\comp{b}+\comp{e}_k}(U)\big|_{U^+ \cap \setVV} . \label{e:label-thread-cap}
  \end{eqnarray}
  \label{lem:lower-half-thread-decomposition}
\end{lemma}

\begin{proof}
  Since Definition~\ref{def:kohnert-label} prioritizes the smallest labels first, it follows that $\Lbl_{\comp{b}+\comp{e}_k}(U)$ restricted to labels $\leq k$ is precisely $\Lbl_{(b_1,\ldots,b_k)+\comp{e}_k} (\setVV)$. By Lemma~\ref{lem:L-minimal-labeling}, this means $\wt(\Mtch_{(b_1,\ldots,b_k)+\comp{e}_k} (\setVV)) = (b_1,\ldots,b_k)+\comp{e}_k$. By Proposition~\ref{prop:matching-thread-decomposition-specializes-Kohnert-labeling}, to prove the lemma, it suffices to show $\thread(\setVV) = (b_1,\ldots,b_k)+\comp{e}_k$. 

  Since $\Mtch_{(b_1,\ldots,b_k)+\comp{e}_k} (\setVV)$ is a matching sequence on $\setVV$ of weight $(b_1,\ldots,b_k)+\comp{e}_k$, we have $\setVV \in \KD((b_1,\ldots,b_k)+\comp{e}_k)$ by Theorem~\ref{thm:TFAE_kohnert}. Hence by Lemma~\ref{lem:3.7}, we have $\thread(\setVV) \preceq (b_1,\ldots,b_k)+\comp{e}_k$. 

  For the reverse inequality, we first establish $\thread(U)_k = c$. Since $\thread(\setVV) \preceq (b_1,b_2,\cdots,b_k) + \comp{e}_k$, we have $\thread(\setVV)_k \leq b_k + 1 = c$. On the other hand, 
  $$
  \thread(\setVV) + (0^k,b_{k+1},b_{k+2}, \cdots) \preceq \comp{b} + \comp{e}_k,
  $$
  so by Proposition~\ref{prop:lswap}, we have $\kd_{\thread(\setVV) + (0^k,b_{k+1},b_{k+1},\cdots)} \in \KD(\comp{b} + \comp{e}_k) \subset \KDs(\comp{a},k)$. Combining the labeling $\Lbl_{\comp{b}+\comp{e}_k}(x)$ on cells of $U \setminus \setVV$ with the thread matching on $\setVV$ gives a matching sequence on $U$ of weight $\thread(\setVV) + (0^k,b_{k+1},b_{k+2}, \cdots)$. Thus $U \in \KD(\thread(\setVV) + (0^k,b_{k+1},b_{k+2}, \cdots))$ by Theorem~\ref{thm:TFAE_kohnert}. Since $U \notin \KDs(\comp{a},k-1)$, neither is $\kd_{\thread(\setVV) + (0^k,b_{k+1},b_{k+1},\cdots)}$. Moreover, $\kd_{\thread(\setVV) + (0^k,b_{k+1},b_{k+1},\cdots)}$ has the same added column $c$ as $U$, since they share the same column weight. So by Lemma~\ref{lem:excision}, the diagram $\kd_{\thread(\setVV) + (0^k,b_{k+1},b_{k+1},\cdots)}$ has a cell $y$ at position $(c,k)$, and consequently, $\thread(\setVV)_k \geq c$. Therefore $(\thread(\setVV))_k = c$ as claimed.

  Since $\comp{b} = \thread(\kd_{\thread(U)} \setminus \{y\})$, where $y$ is in row $k$, we necessarily have $b_j = \thread(U)_j$ for all $j < k$. Moreover, since $\thread(U) \preceq \comp{b} + \comp{e}_k$, we have
  $$
  (0^{k-1},\thread(U)_k,\thread(U)_{k+1}, \cdots) \preceq (0^{k-1},b_k + 1, b_{k+1},b_{k+2},\cdots).
  $$
  Meanwhile, since $U\in\KD(\thread(\setVV) + (0^k,b_{k+1},b_{k+2},\cdots))$, by Lemma~\ref{lem:3.7} we have $\thread(U) \preceq \thread(\setVV) + (0^k,b_{k+1},b_{k+2},\cdots)$. Combining these, we have
  \begin{multline*}
    (b_1,\cdots,b_{k-1}, \thread(U)_k, \thread(U)_{k+1}, \cdots) =
    \thread(U) \\
    \preceq (\thread(\setVV)_1,\cdots,\thread(\setVV)_{k-1},b_k + 1,b_{k+1},b_{k+2},\cdots).
  \end{multline*}
  So since $(0^{k-1},(\thread(U))_k,(\thread(U))_{k+1},\cdots) \preceq (0^{k-1},b_k + 1, b_{k+1},b_{k+2},\cdots)$, it follows that $(b_1,b_2,\cdots,b_{k-1}) \preceq (\thread(\setVV)_1,\thread(\setVV)_2,\cdots,\thread(\setVV)_{k-1})$. Combining this with $(\thread(\setVV))_k = c$  proves $(b_1,\ldots,b_k)+\comp{e}_k \preceq \thread(\setVV)$ as desired.

  Eq.~\eqref{e:label-thread-cap} follows from Eq.~\eqref{e:label-thread} since the threading algorithm begins with the longest threads. 
\end{proof}

\begin{theorem}
  For $U \in \KDstr$, the diagrams $\rectify(U^+_*)$ and $U^-$ are disjoint. Consequently, the diagram $\partial_{\comp{a},k}(U)$ is a generic Kohnert diagram.
\end{theorem}

\begin{proof}
  Continuing notation from Lemma~\ref{lem:rectification}, with $c$ the added column of $U$, for $1\leq i < c$, let $y_i$ denote the cell that moves from column $i+1$ of $\rect^{i-1}(U^+_*)$ to column $i$ of $\rect^i(U^+_*)$. Since $U^+_*$ and $U^-$ are disjoint, it suffices for the theorem to show no cell in $U^-$ lies immediately left of any cell $y_i$ in $U^+_*$.

  Suppose, for contradiction, some $x \in U^-$ lies immediately left of some $y_i$. By Lemma~\ref{lem:moved-cells-in-lower-half}, $y_i \in \setVV$, and so by Lemma~\ref{lem:lower-half-thread-decomposition}, the cell to which $y_i$ threads in the thread decomposition of $\setVV$ must be labeled the same as $y_i$ under $\Lbl_{\comp{b}+\comp{e}_k}(U)$, and the thread length of $y$ must be at least $c$ since $y\in U^+$.

  If $x\in\setVV$, then Lemma~\ref{lem:lower-half-thread-decomposition} also applies to $x$, so its thread has length less than $c$ since $x\in U^-$. In particular, $y$ does not thread into $x$. However, since $x$ is immediately left of $y_i$, the threading algorithm would thread $y_i$ into $x$, since $x$ has a shorter thread and so is still available when $y_i$ is threaded. This contradiction means $x\not\in\setVV$, and also $y_i$ does not thread into $x$. Since $y_i$ is immediately right of $x$, it threads into a higher cell. Thus since $x$ has a longer thread than $y_i$, it has priority in the threading algorithm, and so must terminate in a lower cell in column $1$. However, by Lemmas~\ref{lem:moved-cells-in-lower-half} and \ref{lem:L-minimal-labeling}, the thread of $y_i$ terminates in row $k$ forcing the thread of $x$ to terminate in a lower row, contradicting $x\not\in\setVV$. Thus there cannot be a cell immediately left of any $y_i$.
\end{proof}

To prove the diagram $\partial_{\comp{a},k}(U)$ is, in particular, a Kohnert diagram of $\comp{a}$, we observe that $\partial_{\comp{a},k}(U)$ can be described in a more refined way.

\begin{lemma}
  For $U \in \KDstr$, we have 
  \[ \rectify(U^+_*) = (U^+ \setminus \setVV) \sqcup \rectify(U^+_* \cap \setVV). \]
  In particular, for each $i$, the cell that moves from $\rect^{i-1}(U^+_*)$ to $\rect^i(U^+_*)$ is also the cell that moves from $\rect^{i-1}(U^+_* \cap \setVV)$ to $\rect^i(U^+_* \cap \setVV)$.
  \label{lem:same-moved-cells}
\end{lemma}

\begin{proof}
  If $c = 1$, then $\setW_* = \varnothing$, and so $U^+$ is a generic Kohnert diagram leaving nothing to prove. Thus we may assume $c > 1$. We show by induction on $i < c$ that 
  \[  \rect^{i}(U^+_*) = (U^+ \setminus \setVV) \sqcup \rect^i(U^+_* \cap \setVV) \]
  with the same cell moving in both cases. The base case $i=0$ follows from the decomposition $U^+_* = (U^+ \setminus \setVV) \sqcup (U^+_* \cap \setVV)$.

  Assume the result for $h < i$ for some $1 \leq i < c$. Since $\rect^i(U^+_*)$ is not a generic Kohnert diagram by Lemma~\ref{lem:rectification}, and since $U^+ \setminus \setVV$ is, it follows that $\rect^i(U^+_* \cap \setVV)$ is not a generic Kohnert diagram. Moreover, the cell $y$ that moves from $\rect^{i}(U^+_* \cap \setVV)$ to $\rect^{i+1}(U^+_* \cap \setVV)$ must do so from column $i+2$ to column $i+1$, and so must be the same cell as that which moves from $\rect^{i}(U^+_*)$ to $\rect^{i+1}(U^+_*)$ by Lemma~\ref{lem:moved-cells-in-lower-half}. Finally, in the terminal case $i=c-1$, we have 
  \[ \rectify(U^+_*) = U_{c-1} = (U^+ \setminus \setVV) \sqcup \rect^{c-1}(U^+_* \cap \setVV), \]
  with all the moving cells coinciding on each side. Since the last cell to move lies weakly below $w_c$ by Lemma~\ref{lem:rectification}, we may construct a matching from the cells of $\rect^{c-1}(U^+_* \cap \setVV)$ exactly the same as in the proof of Lemma~\ref{lem:rectification}, showing this is also a generic Kohnert diagram and so equals $\rectify(U^+_* \cap \setVV)$. 
\end{proof}

The final lemma needed for Theorem~\ref{thm:stratum-image} equates the thread decompositions of $U^+ \cap \setVV$ when we remove either $\setW_*$ or the cells that move during rectification.

\begin{definition}
  For $U \in \KDstr$, let $y_i$ be the cell that moves from column $i+1$ of $\rect^{i-1}(U^+_*)$ to column $i$ of $\rect^i(U^+_*)$. The \emph{rectification path of $U$ with respect to $\comp{a}$ and $k$} is the subdiagram $\setY = \{y_1,y_2,\cdots,y_{c-1}\} \subset \partial_{\comp{a},k}(U)$.
  \label{def:rect-path}
\end{definition}

\begin{lemma}
  For $U \in \KDstr$, the diagram $\rectify(U^+_* \cap \setVV) \setminus \setY$ is a generic Kohnert diagram with thread weight
  \[\thread(\rectify(U^+_* \cap \setVV) \setminus \setY) = \thread(U^+ \cap \setV).\]
  \label{lem:rectification-thread-weight}
\end{lemma}

\begin{proof}
  As usual, if the added column $c$ of $U$ satisfies $c=1$, then the result is trivial, so we assume $c>1$. Letting $\setW = \{w_1,\ldots,w_c\}$ with $w_i$ in column $i$, we consider successive steps in the rectification of $U^+_* \cap \setVV$. Let $y_i$ denote the cell that moves from column $i+1$ to column $i$ under the $i$th iteration of $\rect$. Define sets $V_i$ by 
  \[ V_i = \rect(V_{i-1} \cup \{w_{i+1}\}\setminus\{y_i\} \]
  for $1\leq i<c$, where $V_0 = U^+ \cap \setVV$. Then we have
  \[ \rect^{i}(U^+_* \cap \setVV) = V_i \sqcup \{y_s, w_t \mid 1 \leq s < i+1 < t \leq c \} . \]
  In particular, $\rect^{c-1}(U^+_* \cap \setVV) = \rectify(U^+_* \cap \setVV) = V_{c-1} \sqcup \setY$. By construction, for each $i$, the cells in columns $i$ and $i+1$ of diagrams $V_i$ and $\rect^{i}(U^+_* \cap \setVV)$ coincide. We prove by induction on $i$ that $V_i$ is a generic Kohnert diagram with thread weight 
  \[ \thread(V_i) = \thread(U^+ \cap \setV). \]
  The base case is immediate since $V_0 = U^+ \cap \setV$, so assume the result for some $0 \leq i < c-1$. Since cells of $V_i \sqcup \{w_{i+2}\}$ and $\rect^{i}(U^+_* \cap \setVV)$  coincide in columns $i+1$ and $i+2$, since $\rect^{i}(U^+_* \cap \setVV)$ is not a generic Kohnert diagram, neither is $V_i \sqcup \{w_{i+2}\}$. Therefore $V_i \sqcup \{w_{i+2}\}$ is a weak but not generic Kohnert diagram with $w_{i+2}$ a removable cell. 

  We claim $w_{i+2}$ is the \emph{highest} removable cell of $V_i \sqcup \{w_{i+2}\}$. The diagrams $V_i$ and $U^+ \cap \setV$ consist of the same cells weakly to the right of column $i+2$. By Lemma~\ref{lem:lower-half-thread-decomposition}, in the thread decomposition of $(U^+ \cap \setV) \sqcup \setW$, the cells of $\setW$ are threaded last. Thus $w_{i+2}$ is the last cell of its column to be threaded in $(U^+ \cap \setV) \sqcup \setW$, and so $w_{i+2}$ is the last cell of its column to be threaded in $V_i \sqcup \{w_{i+2}\}$ as well. Therefore $\Mtch_{\thd}(V_i) \subset \Mtch_{\thd}(V_i \sqcup \{w_{i+2}\})$. By the uniqueness claim in Lemma~\ref{lem:weak-Kohnert-removable-cell-highest}, $w_{i+2}$ is the highest removable cell of $V_i \sqcup \{w_{i+2}\}$. 

  Since $y_{i+1}$ moves from $V_i \sqcup \{w_{i+2}\} \stackrel{\rect}{\mapsto} V_{i+1}$, it is the \emph{lowest} removable cell of $V_i \sqcup \{w_{i+2}\}$. Therefore, by Corollary~\ref{cor:weak-kohnert-thread-decomposition}, $V_{i+1}$ is a generic Kohnert diagram, with thread weight $\thread(V_{i+1}) = \thread(V_i) = \thread(U^+ \cap \setV)$. 
\end{proof}


Finally, we prove Theorem~\ref{thm:stratum-image}, showing $\partial_{\comp{a},k}(U) \in \KD(\comp{a})$ for $U\in\KDstr$.

\begin{proof}[Proof of Theorem~\ref{thm:stratum-image}.]
  By Lemma~\ref{lem:same-moved-cells}, $\rectify(U^+_*) = (U^+ \setminus \setVV) \sqcup \rectify(U^+_* \cap \setVV)$, so there exists a matching sequence $M$ on $\rectify(U^+_*)$ satisfying $\wt(M) = \thread(U^+ \setminus \setVV) + \thread(U^+ \cap \setVV)$. By Lemma~\ref{lem:rectification}, the cells of $\setY$ are weakly below row $k$, and so by Lemma~\ref{lem:rectification-thread-weight}, we have
  \[\wt(M) \preceq \thread(U^+ \setminus \setVV) + \thread(U^+ \cap \setV) + (c-1) \comp{e}_k. \]
  The sub-diagrams $U^+$ and $\setV$ are constructed by taking subsets of path components of $\Mtch_{\comp{b}+\comp{e}_k}$, and so in particular, restricting this gives a matching sequence on $U \setminus \setV$  and on $U \cap \setV$. By Theorem~\ref{thm:TFAE_kohnert}, $U^+ \setminus \setVV$ and $U^+ \cap \setV$ are Kohnert diagrams of $\wt(\Mtch_{\comp{b}+\comp{e}_k}\big|_{U^+ \setminus \setVV})$ and $\wt(\Mtch_{\comp{b}+\comp{e}_k}\big|_{U^+ \cap \setV})$), respectively, and so by Lemma~\ref{lem:3.7}, we have
  \begin{eqnarray*}
    \thread(U^+ \setminus \setVV) & \preceq & \wt(\Mtch_{\comp{b}+\comp{e}_k}\big|_{U^+ \setminus \setVV}) \\
    \thread(U^+ \cap \setVV) & \preceq & \wt(\Mtch_{\comp{b}+\comp{e}_k}\big|_{U^+ \cap \setV}).
  \end{eqnarray*}
  By definition, all labels in $\Lbl_{\comp{b}+\comp{e}_k}\big|_{U^+ \setminus \setVV}$ (respectively, all labels in $\Lbl_{\comp{b}+\comp{e}_k}\big|_{U^+ \cap \setV}$) are strictly greater than (respectively, strictly less than) $k$. Thus there exists a matching on the key diagram $\kd_{\thread(U^+ \setminus \setVV) + \thread(U^+ \cap \setV) + (c-1) \comp{e}_k}$ with weight 
  \begin{multline*}
    \wt(\Mtch_{\comp{b}+\comp{e}_k}\big|_{U^+ \setminus \setVV}) + \wt(\Mtch_{\comp{b}+\comp{e}_k}\big|_{U^+ \cap \setV}) + (c-1) \comp{e}_k \\
    = \wt(\Mtch_{\comp{b}+\comp{e}_k}\big|_{U^+}) + \wt(\Mtch_{\comp{b}+\comp{e}_k}\big|_{\setW}) - \comp{e}_k
    = \wt(\Mtch_{\comp{b}+\comp{e}_k}\big|_{U^+}) - \comp{e}_k.
  \end{multline*}
  So then Theorem~\ref{thm:TFAE_kohnert} implies
  \[ \wt(M) \preceq \thread(U^+ \setminus \setVV) + \thread(U^+ \cap \setV) + (c-1) \comp{e}_k \preceq \wt(\Mtch_{\comp{b}+\comp{e}_k}\big|_{U^+}) - \comp{e}_k. \]
  Therefore the generic Kohnert diagram $\partial_{\comp{a},k}(U) = U^- \sqcup \rectify(U^+_*)$ has a matching sequence $M' = (\Mtch_{\comp{b}+\comp{e}_k}\big|_{U^-}) \sqcup M$. By Lemma~\ref{lem:L-minimal-labeling}, we have $\wt(\Mtch_{\comp{b}+\comp{e}_k}\big|_{U^-}) = \wt(\Mtch_{\comp{b}+\comp{e}_k}\big|_{U^-})$, and so $M'$ has weight 
  \[ \wt(M') = \wt(\Mtch_{\comp{b}+\comp{e}_k}\big|_{U^-}) + \wt(M) \preceq \wt(\Mtch_{\comp{b}+\comp{e}_k}\big|_{U^-}) + \wt(\Mtch_{\comp{b}+\comp{e}_k}\big|_{U^+}) - \comp{e}_k. \]
  Since $\Mtch_{\comp{b}+\comp{e}_k}$ is a matching on $U = U^- \sqcup U^+$ with $\wt(\Mtch_{\comp{b}+\comp{e}_k}) = \comp{b} + \comp{e}_k$, we have
  \[ \wt(M') \preceq \wt(\Mtch_{\comp{b}+\comp{e}_k}) - \comp{e}_k = \comp{b} + \comp{e}_k - \comp{e}_k = \comp{b} \preceq \comp{a}. \]
  Since $M'$ is a matching sequence on $\partial_{\comp{a},k}(U)$, by Theorem~\ref{thm:TFAE_kohnert} we conclude $\partial_{\comp{a},k}(U)$ is a Kohnert diagram for $\comp{a}$.
\end{proof}


\subsection{Stratum maps are injective}
\label{sec:stratify-inject}

We have shown $\partial_{\comp{a},k}$ is a well-defined map from $\KDstr$ to $\KD(\comp{a})$, so we now show $\partial_{\comp{a},k}$ is injective. To do so, we use the rectification path to partition the stratum maps $\partial_{\comp{a},k}$ as follows.

\begin{definition}
  For $U \in \KDstr$, we partition the stratum map by
  \begin{eqnarray*}
    \partial_{\comp{a},k}^+(U) & = & \rectify(U^+_*) \setminus \setY \\
    \partial_{\comp{a},k}^-(U) & = & U^- \sqcup \setY.
  \end{eqnarray*}
\end{definition}

Notice, by construction we have $\partial_{\comp{a},k}(U) = \partial_{\comp{a},k}^+(U) \sqcup \partial_{\comp{a},k}^-(U)$. We will show this partitioning is natural with respect to the thread decomposition.

\begin{proposition}
  For $U \in \KDstr$, the subdiagrams $\partial_{\comp{a},k}^+(U)$ and $\partial_{\comp{a},k}^-(U)$ are generic Kohnert diagrams. For $M^+$ and $M^-$ matching sequences on $\partial_{\comp{a},k}^+(U)$ and $\partial_{\comp{a},k}^-(U)$, respectively, for every cell $x \in \partial_{\comp{a},k}(U)$, we have
  \[\begin{array}{ll}
    \plength_{M^+}(x) \geq c & \mbox{ if } x \in \partial_{\comp{a},k}^+(U) \\
    \plength_{M^-}(x) < c & \mbox{ if } x \in \partial_{\comp{a},k}^-(U).
  \end{array}\]
  \label{prop:image-thread-halves}
\end{proposition}

\begin{proof}
  We first consider the diagram $\partial_{\comp{a},k}^+(U)$. By Lemma~\ref{lem:same-moved-cells}, $\partial_{\comp{a},k}^+(U) = (U^+ \setminus \setVV) \sqcup \rectify(U^+_* \cap \setVV) \setminus \setY$. By its definition, $U^+ \setminus \setVV$ is a generic Kohnert diagram with the same number of cells at each column weakly left of column $c$. On the other hand, $(\rectify(U^+_* \cap \setVV) \setminus \setY)$ is a generic Kohnert diagram by Lemma~\ref{lem:rectification-thread-weight}, with thread weight $\thread(U^+ \cap \setV)$, and hence must also have the same number of cells at each column weakly left of column $c$. Thus, $\partial_{\comp{a},k}^+(U)$ is a generic Kohnert diagram with the same number of cells at each column weakly left of column $c$. It follows that $\plength_{M^+}(x) \geq c$ for every cell $x \in \partial_{\comp{a},k}^+(U)$ and every matching sequence on $\partial_{\comp{a},k}^+(U)$.

  We next consider the diagram $\partial_{\comp{a},k}^-(U)$. By its definition, $U^-$ is a generic Kohnert diagram occupying only positions strictly left of column $c$. By Lemma~\ref{lem:rectification}, the diagram $\setY$ is a generic Kohnert diagram whose cells all lie strictly left of column $c$. Thus, $\partial_{\comp{a},k}^-(U) = U^- \sqcup \setY$ is a generic Kohnert diagram occupying only positions strictly left of column $c$. Thus $\plength_{M^-}(x) < c$ for every cell $x \in \partial_{\comp{a},k}^-(U)$ and every matching sequence on $\partial_{\comp{a},k}^-(U)$.
\end{proof}

The following is helpful in studying thread decompositions for $U$ and for $\partial_{\comp{a},k}(U)$.

\begin{lemma}
  For $U \in \KDstr$, $c$ the added column of $U$, and $\comp{b}$ the excised weight of $U$, for any weak composition $\comp{d}$ such that $\thread(U) \preceq \comp{d} \prec \comp{b} + \comp{e}_k$, we have $d_k > c$. 
  \label{lem:criterion-interval-weak-comp}
\end{lemma}

\begin{proof}
  Since $\comp{b} = \thread(\kd_{\thread(U)} \setminus \{y\})$, where by Lemma~\ref{lem:excision} the cell $y$ is in position $(c,k)$, we have $(\thread(U))_j = b_j$ for all indices $j < k$. Hence, $d_j = b_j$ for all $j < k$ as well. So since $\comp{d} \preceq \comp{b} + \comp{e}_k$, we must have $d_k \geq b_k + 1 = c$. 
  
  Suppose, for contradiction, $d_k = c$. Since $\thread(U) \preceq \comp{d}$, we may consider the Kohnert labeling $L = \Lbl_{\comp{d}}(\kd_{\thread(U)})$. Since $d_k = b_j = (\thread(U))_j$ for all indices $j < k$, we must have $L(x) = \row(x)$ for every cell $x$ below row $k$. Therefore for every cell in position $(i,k)$ for $1 \leq i \leq d_k$, must have label $L(x) = k$. In particular, $L(y) = k$ since $y$ is in position $(c,k)$. Since $d_k = b_k + 1 = c$, $y$ is the rightmost cell in $\kd_{\thread(U)}$ with label $k$ under $L$. Therefore, restricting $L$ to the diagram $\kd_{\thread(U)} \setminus \{y\}$ gives us a labeling with content weight $\comp{d} - \comp{e}_k$. By Lemma~\ref{lem:3.7}, $\comp{b} = \thread(\kd_{\thread(U)} \setminus \{y\}) \preceq \comp{d} - \comp{e}_k$. So since $b_j = d_j$ for all indices $j < k$, and since $b_k = d_k - 1$, we necessarily have $(0^k,b_{k+1},b_{k+2},\cdots) \preceq (0^k,d_{k+1}, d_{k+2},\cdots)$. Hence, since $b_j = d_j$ for all indices $j < k$, and since $b_k + 1 = d_k$, we ultimately get $\comp{b} + \comp{e}_k \preceq \comp{d}$. Since $\comp{d} \preceq \comp{b} + \comp{e}_k$ we must have $\comp{d} = \comp{b} + \comp{e}_k$, which contradicts our assumption that $\comp{d} \neq \comp{b} + \comp{e}_k$.
\end{proof}

\begin{lemma}
  For $U \in \KDstr$ and $\comp{b}$ the excised weight of $U$, we have
  \begin{equation}
    \Mtch_{\thd}(U \setminus \setW) = \Mtch_{\comp{b}+\comp{e}_k}(U)\big|_{U \setminus \setW}.
  \end{equation}
  \label{lem:complement-thread-decomposition}
\end{lemma}

\begin{proof}
  For brevity, we let $M = \Mtch_{\comp{b}+\comp{e}_k}(U)\big|_{U \setminus \setW}$. Since the cells of $\setW$ form the path labeled $k$ in $\Lbl_{\comp{b}+\comp{e}_k}(U)$, restricting to cells of $U\setminus\setW$ is equivalent to $\Lbl_{\comp{b}-b_k\comp{e}_k}(U\setminus\setW)$. By Proposition~\ref{prop:matching-thread-decomposition-specializes-Kohnert-labeling}, in order to prove $M = \Mtch_{\thd}(U \setminus \setW)$, it suffices to show that $\thread(U \setminus \setW) = \wt(M) = \comp{b} - b_k \comp{e}_k$, by Lemma~\ref{lem:L-minimal-labeling}.

    By Theorem~\ref{thm:TFAE_kohnert}, the fact that $(U \setminus \setW)$ is the underlying diagram of the matching sequence $M$ implies that $(U \setminus \setW)$ is a Kohnert diagram of $\wt(M)$. Thus $\thread(U \setminus \setW) \preceq \wt(M)$.
  
  To prove $\wt(M) \preceq \thread(U \setminus \setW)$, note that $M \sqcup (w_1 \leftarrow w_2 \leftarrow \cdots w_c)$ is a matching sequence on $U$, so that by Theorem~\ref{thm:TFAE_kohnert}, 
  we have $\thread(U) \preceq \wt(M) + c \comp{e}_k$. Since $\thread(U \setminus \setW) \preceq \wt(M)$ with $(\thread(U \setminus \setW))_k = 0 = (\wt(M))_k$, we have 
  $$
  \thread(U) \preceq \thread(U \setminus \setW) + c \comp{e}_k \preceq \wt(M) + c \comp{e}_k = \comp{b} - b_k \comp{e}_k + c \comp{e}_k = \comp{b} + \comp{e}_k.
  $$
  If $\thread(U \setminus \setW) \neq \wt(M)$, then $\thread(U \setminus \setW) + c \comp{e}_k \neq \comp{b} + \comp{e}_k$, and it follows from Lemma~\ref{lem:criterion-interval-weak-comp} that 
  $  c = (\thread(U \setminus \setW) + c \comp{e}_k)_k > c  $,
  a contradiction.
\end{proof}

\begin{lemma}
  For $U \in \KDstr$, we have
  \[ \Mtch_{\thd}(\partial_{\comp{a},k}(U)) = \Mtch_{\thd}(\partial_{\comp{a},k}^+(U)) \sqcup \Mtch_{\thd}(\partial_{\comp{a},k}^-(U)). \]
  \label{lem:image-thread-decomposition}
\end{lemma}

\begin{proof}
  By Proposition~\ref{prop:image-thread-halves}, both $\partial_{\comp{a},k}^+(U)$ and $\partial_{\comp{a},k}^-(U)$ are Kohnert diagrams, so we may consider their thread decompositions. Moreover, since all threads of $\Mtch_{\thd}(\partial_{\comp{a},k}^+(U))$ have length at least $c$, where $c$ is the added column of $U$, the result follows by Proposition~\ref{prop:image-thread-halves} from showing each thread of $\Mtch_{\thd}(\partial_{\comp{a},k}^+(U))$ is also a thread of $\Mtch_{\thd}(\partial_{\comp{a},k}(U))$.

  Let $x$ be the first cell in the thread decomposition of $\partial_{\comp{a},k}^+(U)$ such that $x$ threads into some $y \not\in \partial_{\comp{a},k}^+(U)$, and so $y \in U^- \sqcup \setY$. By Lemma~\ref{lem:complement-thread-decomposition}, we cannot have $y \in U^-$ since then it would have a thread of length shorter than $c$. Moreover, by Lemma~\ref{lem:rectification-thread-weight}, if $y \in \setY$ then $y$ must have been in $\setW$ before rectification, but then by Lemma~\ref{lem:complement-thread-decomposition}, it again has to have a thread of length shorter than $c$. Thus $y$ cannot exist, and the result follows. 
\end{proof}

We now have our first major step toward reversing the stratum maps.

\begin{theorem}
  For $U \in \KDstr$ with added column $c$, let $\setY = \{y_1,\ldots,y_{c-1}\}$ with $y_i$ in column $i$ for $i\geq 1$, and let $y_0$ be the position $(1,k)$. be the rectification path of $U$. Then for each $1 < i \leq c-1$, $y_i$ is the highest cell in column $i$ weakly below $y_{i-1}$ with $\plength_{\Mtch_{\thd}(\partial_{\comp{a},k}(U))}(y_i) < c$.
    \label{thm:moved-cells-recoverable-given-added-col}
\end{theorem}

\begin{proof}
  By Lemma~\ref{lem:image-thread-decomposition} and Proposition~\ref{prop:image-thread-halves}, we have $\partial_{\comp{a},k}^-(U) = \{ x \in \partial_{\comp{a},k}(U) \mid \plength_M(x) < c \}$. Thus it suffices to show at every column $1 \leq i < c$, the cells of $U^- = \partial_{\comp{a},k}^-(U) \setminus \setY$ in column $i$ that are below $y_{i-1}$ are also below $y_i$.

  By definition of $\rect$, we have $\matchable_T(i,\row(y_i)) = 0$, and by Lemma~\ref{lem:rectification}, we have $\matchable_T(i,j) > 0$ for all rows $j$ such that $\row(y_i) < j \leq \row(y_{i-1})$. By the greedy choice of the threading algorithm, this means the cells weakly above $\row(y_i)$ and strictly below $y_{i-1}$ in columns $i-1,i$ are matched with lengths at least $c$. Thus by Proposition~\ref{prop:image-thread-halves}, they cannot lie in $U^-$, and the result follows.
\end{proof}

By Theorem~\ref{thm:moved-cells-recoverable-given-added-col}, if we know the added column $c$ of $U$, then we can recover the cells $\{y_1,y_2,\cdots,y_{c-1}\}$ that moved to obtain $\partial_{\comp{a},k}(U)$. Our next and final task for proving the injectivity of $\partial_{\comp{a},k}(U)$ is to show that the added column $c$ is unique.

The following elementary lemma allows us to adjust a matching sequence on a key diagram to pass through a certain cell at the end of its row.

\begin{lemma}
  Let $T$ be a generic Kohnert diagram, and let $c$ and $r$ be positive integers such that $T$ occupies every position weakly left of column $c$ in row $r$.
  Let $z$ be the cell in position $(c,r)$, and let $w$ be a cell in row $r$ weakly to the left of $z$. 
  
  If there exists a matching sequence $M$ on $T$ such that $\plength_M(w) \geq c$, then there exists a matching sequence $M'$ on $T$ with $\wt(M') \preceq \wt(M)$ such that $\plength_{M'}(z) \geq c$.
  \label{lem:matchings-lengthen-row-k}
\end{lemma}

\begin{proof}
  Let $M$ be a matching sequence on $T$ and suppose $w$ is the rightmost cell in row $r$ weakly to the left of $z$ such that $\plength_M(w) \geq c$. If $w=z$, then there is nothing to show, so assume $w \neq z$ in which case we may let $x$ be the cell of $T$ to the immediate right of $w$ in row $k$. By induction, it suffices to show that there exists a matching sequence $M'$ on $T$ such that $\plength_{M'}(x) \geq c$.

  If $M(x)=w$, then $M$ is such a matching, so we may assume this is not the case. Therefore $M(x)$ lies strictly above row $k$ and the cell $y$ for which $M(y)=w$ lies strictly below row $k$. Consider the underlying diagram $S$ of the path components of $w$ and $x$ in $M$, and let $\Mtch_{\thd}(S) = P \sqcup Q$, where $P$ has length $\plength_M(w) \geq c$ and $Q$ has length $\plength_M(x) < c$. Since $M(x)$ is above $w$, the latter is threaded first, and so belongs to $P$. If $x$ also belongs to $P$, then we set $N$ to be the thread matching on $S$. Otherwise, $x$ belongs to $Q$, which is strictly shorter than $P$, and so there must exist some pair of cells $u$ and $v$ in some column $i$ weakly to the right of $x$, with the cell $u$ of $Q$ above the cell $v$ of $P$, such that in column $(i+1)$, either the cell of $Q$ is below the cell of $P$ or there is no cell of $Q$ at all. In either case, we may construct a matching sequence $N$ on $S$ from $\Mtch_{\thd}(S)$ by swapping the cells into which $u$ and $v$ thread, setting $N(x)=w$, and, for $y$ the cell that threads into $w$, setting $N(y)$ to be the cell into which $x$ threads. Then $\wt(N) = \thread(S)$, and $\plength_N(x) = \plength_N(w) \geq c$. Use this to define a matching sequence $M' = (M \mid_{T \setminus S}) \sqcup N$ on $T$, so that $\plength_{M'}(x) \geq c$ and via Lemma~\ref{lem:3.7} $\wt(M') = \wt(M \mid_{T \setminus S}) + \thread(S) \preceq \wt(M)$.
\end{proof}

The next three lemmas provide the foundation for proving the uniqueness of the added column of any pre-image of $\partial_{\comp{a},k}(U)$.

\begin{lemma}
  For any $U \in \KDstr$, the added column $c$ of $U$ satisfies $a_k < c$. Moreover, for $\comp{b}$ the excised weight of $U$, we have
  \begin{equation}
    \# \{ j > k \mid b_j \geq c \} = \# \{ j > k \mid a_j \geq c \}
    \label{eq:col-added-bounds}
  \end{equation}
  \label{lem:col-added-bounds}
\end{lemma}

\begin{proof}
  From the relation $\comp{b} \preceq \comp{a}$, for any value $c$ we have inequalities
  \begin{eqnarray*}
    \# \{ j > k \mid b_j \geq c \} & \leq & \# \{ j > k \mid a_j \geq c \} , \\
    \# \{ j \geq k \mid b_j \geq c \} & \leq & \# \{ j \geq k \mid a_j \geq c \} .
  \end{eqnarray*}
  Further, since $b_k = c-1$ for $c$ the added column of $U$, the left hand sides above must be equal. Therefore if equality holds for the lower expression, then $a_k < c$, and so Eq.~\eqref{eq:col-added-bounds} follows. Thus it suffices to show 
  \begin{equation}
    \# \{ j > k \mid b_j \geq c \} = \# \{ j \geq k \mid a_j \geq c \}.
    \label{eq:col-added-weight-criterion}
  \end{equation}
  Suppose Eq.~\eqref{eq:col-added-weight-criterion} is false. We will show this implies $U \in \KDs(\comp{a},k-1)$, contradicting $U \in \KDstr$, thereby proving Eq.~\eqref{eq:col-added-weight-criterion} and so, too, Eq.~\eqref{eq:col-added-bounds}.

  Let $y$ be the cell in position $(b_k,k)$ of the key diagram $\kd_{\comp{b}}$. We claim there exists a matching sequence $M$ on $\kd_{\comp{b}}$ with $\wt(M) \preceq \comp{a}$ such that $\plength_M(y) \geq c$. The key diagram $\kd_{\comp{b}}$ is a Kohnert diagram of $\comp{a}$ by Proposition~\ref{prop:lswap}, so
  by Definition~\ref{def:kohnert-label}, we may consider the labeling $\Lbl_{\comp{a}}$ on $\kd_{\comp{b}}$, with corresponding matching sequence $\wt(\Mtch_{\comp{a}}) \preceq \comp{a}$.
  If $\plength_{\Mtch_{\comp{a}}}(y) \geq c$, then we may take $M = \Mtch_{\comp{a}}$. Otherwise, by Lemma~\ref{lem:matchings-lengthen-row-k}, we may assume every cell in row $k$ of $\kd_{\comp{b}}$ belongs to a path component in $\Mtch_{\comp{a}}$ strictly shorter than $c$. If Eq.~\eqref{eq:col-added-weight-criterion} is false, then there exists a cell $z$ in column $c$ of $\kd_{\comp{b}}$ such that $z$ is below row $k$ but $L(z) \geq k$. 
  Let $i$ be the leftmost column such that the cell on the path component of $\Mtch_{\comp{a}}$ containing $z$ lies strictly below row $k$. If $i = c$, then we may take $M=\Mtch_{\comp{a}}$ for all cells other than $z$ and set $M(z)=y$. Then $\wt(M) \prec \wt(\Mtch_{\comp{a}}) \preceq \comp{a}$ and $\plength_M(y) = \plength_{M_0}(z) \geq c$. Else if $i<c$, then let $w$ denote the cell in position $(i,k)$. By our assumption about the lengths of paths in $\Mtch_{\comp{a}}$ for cells in row $k$, we have $\plength_{\Mtch_{\comp{a}}}(w) < c \leq \plength_{\Mtch_{\comp{a}}}(z)$. Thus there exists some column maximal column $t \geq i$ such that the cells on the path component of $\Mtch_{\comp{a}}$ containing $w$ lie above the cells on the path component of $\Mtch_{\comp{a}}$ containing $z$ in every column $s$ between $i$ and $t$.
  Define $L'$ to be the labeling derived from $L$ by swapping the labels of the cells in the path components of $\Mtch_{\comp{a}}$ containing $w$ and $z$ for every column $s$ between $i$ and $t$.
  Then $\wt(L') = \wt(L) = \comp{a}$. The same argument used in the proof of Proposition~\ref{prop:label-wt} applies to the underlying matching sequence of $L'$, and so we obtain a matching sequence $\Mtch_{L'}$ on $\kd_{\comp{b}}$ with $\wt(\Mtch_{L'}) \preceq \comp{a}$ such that $\plength_M(w) = \plength_{\Mtch_{\comp{a}}}(z) \geq c$, proving the claim.    

  Let $M$ be any matching sequence on $\kd_{\comp{b}}$ with $\wt(M) \preceq \comp{a}$ and $\plength_M(y) \geq c$, and let $z$ be the cell in column $c$ for which $M(z) = y$. Let $z'$ denote the cell in position $(c,k)$ of the key diagram $\kd_{\comp{b} + \comp{e}_k}$. Define a matching sequence $M'$ on $\kd_{\comp{b} + \comp{e}_k} \setminus \{z\}$ from $M$ by setting $M'(z') = y$, $M'(x) = z'$ if $M(x)=z$, and $M'(x)=M(x)$ otherwise. Then $\wt(M') = \wt(M) \preceq \comp{a}$. Thus, by Theorem~\ref{thm:TFAE_kohnert}, the diagram $\kd_{\comp{b} + \comp{e}_k} \setminus \{z\}$ is a Kohnert diagram of $\comp{a}$, so that by Lemma~\ref{lem:3.7}, we have $\thread(\kd_{\comp{b} + \comp{e}_k} \setminus \{z\}) \preceq \comp{a}$. Define another matching sequence $N$ on $\kd_{\comp{b} + \comp{e}_k}$ with anchor weight $\thread(\kd_{\comp{b} + \comp{e}_k} \setminus \{z\}) + \comp{e}_{\row(z)}$ by taking $\Mtch_{\thd}(\kd_{\comp{b} + \comp{e}_k} \setminus \{z\})$ and then matching $z$ into the cell to its immediate left in $\kd_{\comp{b} + \comp{e}_k}$. By Theorem~\ref{thm:TFAE_kohnert}, $\kd_{\comp{b} + \comp{e}_k}$ is a Kohnert diagram of $\thread(\kd_{\comp{b} + \comp{e}_k} \setminus \{z\}) + \comp{e}_{\row(z)}$, and so by Lemma~\ref{lem:3.7}, 
  \[ \thread(U) \preceq \comp{b} + \comp{e}_k \preceq \thread(\kd_{\comp{b} + \comp{e}_k} \setminus \{z\}) + \comp{e}_{\row(z)}. \]
  Since $\thread(\kd_{\comp{b} + \comp{e}_k} \setminus \{z\}) \preceq \comp{a}$ and $\row(z) < k$, we conclude that $U \in \KDs(\comp{a},k-1)$.
\end{proof}

\begin{lemma}
  For any $U \in \KDstr$, the diagram $\partial_{\comp{a},k}^+(U)$ is a Kohnert diagram of $\thread(U^+)$, and the diagram $\partial_{\comp{a},k}^-(U)$ is a Kohnert diagram of $\thread(U^-) + (c-1) \comp{e}_k$. Thus $\partial_{\comp{a},k}(U) = \partial_{\comp{a},k}^+(U) \sqcup \partial_{\comp{a},k}^-(U)$ is a Kohnert diagram of $\comp{b}$, and $\thread(\partial_{\comp{a},k}(U)) \preceq \comp{b}$.
  \label{lem:image-thread-weight}
\end{lemma}

\begin{proof}
  Recall $\partial_{\comp{a},k}(U) = \partial_{\comp{a},k}^+(U) \sqcup \partial_{\comp{a},k}^-(U)$, where $\partial_{\comp{a},k}^-(U) = U^- \sqcup \setY$ and, by Lemma~\ref{lem:same-moved-cells}, we have $\partial_{\comp{a},k}^+(U) = (U^+ \setminus \setVV) \sqcup \rectify(U^+_* \cap \setVV) \setminus \setY$. By Lemma~\ref{lem:rectification-thread-weight}, we have $\thread(\rectify(U^+_* \cap \setVV) \setminus \setY) = \thread(U^+ \cap \setV)$ giving a matching sequence $M^+ = \Mtch_{\thd}(U^+ \setminus \setVV) \sqcup \Mtch_{\thd}(\rectify(U^+_* \cap \setVV) \setminus \setY)$ on $\partial_{\comp{a},k}^+(U)$ with anchor weight $\wt(M^+) = \thread(U^+ \setminus \setVV) + \thread(U^+ \cap \setV) = \thread(U^+)$. On the other hand, Lemma~\ref{lem:rectification} implies the diagram $\setY$ is a Kohnert diagram of $(c-1) \comp{e}_k$, giving a matching sequence $M^- = \Mtch_{\thd}(U^-) \sqcup \Mtch_{\thd}(\setY)$ with anchor weight $\wt(M^-) \preceq \thread(U^-) + (c-1) \comp{e}_k$. By Theorem~\ref{thm:TFAE_kohnert}, $\partial_{\comp{a},k}^+(U)$ is a Kohnert diagram of $\thread(U^+)$ and $\partial_{\comp{a},k}^-(U)$ is a Kohnert diagram of $\thread(U^-) + (c-1) \comp{e}_k$.
  
  Combining $M^+$ and $M^-$ gives a matching sequence $M$ on $\partial_{\comp{a},k}(U)$ with anchor weight $\wt(M) \preceq \thread(U^-) + \thread(U^+) + (c-1) \comp{e}_k$. By Theorem~\ref{thm:TFAE_kohnert}, we have $U^{\pm} \in \KD(\wt(L \mid_{U^{\pm}}))$, which by Lemma~\ref{lem:3.7} implies $\thread(U^{\pm}) \preceq \wt(L \mid_{U^{\pm}})$.  Hence
  $$
  \thread(U^+) + \thread(U^-) \preceq \wt(L \mid_{U^+}) + \wt(L \mid_{U^-}) = \wt(L \mid_{U^+ \sqcup U^-}) = \comp{b} - b_k \comp{e}_k.
  $$
  Since $b_k = c-1$, it follows that $\wt(M) \preceq \comp{b} - (c-1) \comp{e}_k + (c-1) \comp{e}_k = \comp{b}$. We conclude $\partial_{\comp{a},k}(U) \in \KD(\comp{b})$ by Theorem~\ref{thm:TFAE_kohnert}, and so $\thread(\partial_{\comp{a},k}(U)) \preceq \comp{b}$ by Lemma~\ref{lem:3.7}.
\end{proof}

\begin{lemma}
  For $U \in \KDstr$, $c$ the added column and $\comp{b}$ the excised weight of $U$, the weight of the thread decomposition of $\partial_{\comp{a},k}^+(U)$ is
  \[ \thread(\partial_{\comp{a},k}^+(U))_j =
  \begin{cases}
    b_j & \mbox{ if } b_j \geq c \\
    0 & \mbox{ otherwise.}
  \end{cases} \]
  In particular, we have
  \begin{equation}
    \# \{ j > k \mid b_j \geq c \} = \# \{ j > k \mid (\thread(\partial_{\comp{a},k}(U))_j \geq c \}.
    \label{eq:same-upper-longer-half-column-weight}
  \end{equation}
  \label{lem:same-longer-parts}
\end{lemma}

\begin{proof}
  By construction of $\partial_{\comp{a},k}^+(U)$, we have $\partial_{\comp{a},k}^+(U) \sqcup \setY = U^+_* = U^+ \setminus \{w_1\}$, where we regard $\setY$ as a sub-diagram of $U$ (pre-rectification). So then the diagram $U$ can be partitioned as
  \[ U = (U^+) \sqcup U^- = \partial_{\comp{a},k}^+(U) \sqcup (\setY \sqcup \{w_1\}) \sqcup U^-. \]
  Lemma~\ref{lem:rectification} implies the diagram $\setY \sqcup \{w_1\}$ is a generic Kohnert diagram, where since $w_1$ is in row $k$ by Lemma~\ref{lem:L-minimal-labeling}, we have $\thread(\setY \sqcup \{w_1\}) = c \comp{e}_k$.
  Thus, there exists a matching sequence
  $M = \Mtch_{\thd}(\partial_{\comp{a},k}^+(U)) \sqcup \Mtch_{\thd}(\setY \sqcup \{w_1\}) \sqcup \Mtch_{\thd}(U^-)$ on $U$,
  with anchor weight $\wt(M) = \thread(\partial_{\comp{a},k}^+(U)) + c \comp{e}_k + \thread(U^-)$.
  By Theorem~\ref{thm:TFAE_kohnert}, we obtain
  $\thread(U) \preceq \thread(\partial_{\comp{a},k}^+(U)) + \thread(U^-) + c \comp{e}_k$.
  
  On the other hand, since $\partial_{\comp{a},k}^+(U)$ is a Kohnert diagram of $\thread(U^+)$ by Lemma~\ref{lem:image-thread-weight}, it follows by Lemma~\ref{lem:3.7} that $\thread(\partial_{\comp{a},k}^+(U)) \preceq \thread(U^+)$. By Theorem~\ref{thm:TFAE_kohnert}, we have $U^{\pm} \in \KD(\wt(L \mid_{U^{\pm}}))$, which by Lemma~\ref{lem:3.7} implies $\thread(U^{\pm}) \preceq \wt(L \mid_{U^{\pm}})$. Hence
  \[ \thread(U^+) + \thread(U^-) \preceq \wt(L \mid_{U^+}) + \wt(L \mid_{U^-}) = \wt(L \mid_{U^+ \sqcup U^-}) = \comp{b} - b_k \comp{e}_k. \]
  Therefore we have
  \[ \thread(\partial_{\comp{a},k}^+(U)) + \thread(U^-) + c \comp{e}_k \preceq \thread(U^+) + \thread(U^-) + c \comp{e}_k \preceq \comp{b} + (c - b_k) \comp{e}_k = \comp{b} + \comp{e}_k. \]
  Since $\thread(U) \preceq \thread(\partial_{\comp{a},k}^+(U)) + \thread(U^-) + c \comp{e}_k \preceq \comp{b} + \comp{e}_k$, where the $k$th part is given by precisely $b_k + 1 = c$, Lemma~\ref{lem:criterion-interval-weak-comp} implies we must have
  \[ \thread(\partial_{\comp{a},k}^+(U)) + \thread(U^-) + c \comp{e}_k = \comp{b} + \comp{e}_k. \]
  Now $b_k = c-1 < c$, and by the definition of $U^-$, all parts of $\thread(U^-)$ are necessarily less than $c$. Thus $\thread(\partial_{\comp{a},k}^+(U)) = \comp{b}^+$, and Eq.~\ref{eq:same-upper-longer-half-column-weight} follows.
\end{proof}

Finally, we combine these results to prove for $U \in \KDstr$, the diagram $U$ is the unique pre-image of the diagram $\partial_{\comp{a},k}(U)$ under the $k$th stratum map $\partial_{\comp{a},k}$ of $\comp{a}$. That is, $\partial_{\comp{a},k}$ is injective.

\begin{proof}[Proof of Theorem~\ref{thm:stratum-inject}.]
  Let $U \in \KDstr$. By Lemma~\ref{lem:rectification}, we know exactly how each $y_i \in \setY$ moved in constructing $\partial_{\comp{a},k}(U)$. Moreover, by Lemma~\ref{lem:L-minimal-labeling}, the excised cell $w_1$ of $U$ is in position $(1,k)$. Thus, by Theorem~\ref{thm:moved-cells-recoverable-given-added-col}, we may completely recover the Kohnert diagram $U \in \KDstr$ from its image under $\partial_{\comp{a},k}$ provided we know the added column $c$ of $U$. It follows that if we have a different diagram $U' \in \KDstr$ such that $\partial_{\comp{a},k}(U') = \partial_{\comp{a},k}(U)$, then $U'$ must come equipped with a distinct added column $c' \neq c$ and consequently have an excised weight $\comp{b'} \neq \comp{b}$, since by construction we would get $b'_k = c'-1 \neq c-1 = b_k$.

  Suppose, for contradiction, $U' \in \KDstr$ with added column $c' > c$ and excised weight $\comp{b'} \neq \comp{b}$ satisfies $\partial_{\comp{a},k}(U') = \partial_{\comp{a},k}(U)$. By Eq.~\eqref{eq:same-upper-longer-half-column-weight} of Lemma~\ref{lem:same-longer-parts} and Eq.~\eqref{eq:col-added-bounds} of Lemma~\ref{lem:col-added-bounds}, we have 
  \[\# \{ j > k \mid (\thread(\partial_{\comp{a},k}(U)))_j \geq c \} = \# \{ j > k \mid b_j \geq c \} = \# \{ j > k \mid a_j \geq c \}. \]
  
  By Lemma~\ref{lem:excision}~(2), $\comp{b'} \preceq \comp{a}$, and by Lemma~\ref{lem:col-added-bounds}, we have $c > a_k$, so that by our assumption, we have $b' = c'-1 \geq c > a_k$, and so 
  \[ \# \{ j > k \mid b'_j \geq c \} < \# \{ j > k \mid a_j \geq c \}. \]
  However, since $\thread(\partial_{\comp{a},k}(U)) \preceq \comp{b'}$, by Lemma~\ref{lem:image-thread-weight}, we have
  \[\# \{ j > k \mid (\thread(\partial_{\comp{a},k}(U)))_j \geq c \} \geq \# \{ j > k \mid a_j \geq c \}. \]
  These two inequalities directly contradict the earlier one, and so $c$ is unique.
\end{proof}

%
\section{Key applications}
%
\label{sec:applications}

In its more general form, Pieri's rule \cite{Pie93} gives an elegant formula for the Schur expansion of the product of a Schur polynomial with a \emph{single row} Schur polynomial as the sum over all ways to add $m$ cells successively to the diagram of the partition with no two cells added in the same column. We call these added cells a \emph{horizontal $m$-strip}. The same bijective proof using the Robinson--Schensted--Knuth insertion algorithm applies to this more general case as well, due to properties of the algorithm when inserting successive entries that are weakly decreasing. Analogous to this, in Section~\ref{sec:applications-hor} we present our final main result giving a similar bijection that takes a pair of Kohnert diagrams, the latter being a horizontal strip, and maps it bijectively to another Kohnert diagram. We prove the bijection in Section~\ref{sec:applications-bij} by considering further properties of the bijection for the single cell case. Finally, in Section~\ref{sec:applications-pos}, we look closely at our formula to characterize when the product is nonnegative, giving simpler formulas for certain special cases, including the geometrically important \emph{Schubert polynomials} of Lascoux and Sch{\"u}tzenberger \cite{LS82}.

\subsection{Horizontal strips}
\label{sec:applications-hor}

As with the single cell case, an elegant combinatorial proof uses the RSK insertion algorithm \cite{Rob38,Sch61,Knu70} together with a row-bumping lemma that considers successive insertions.

\begin{theorem}[\cite{Sch61}]
  Given any partition $\lambda$ of length at most $n$ and any positive integer $m$, there exists a weight-preserving bijection
  \begin{equation}
    \SSYT_n(\lambda) \times \SSYT_n(m) \stackrel{\sim}{\longrightarrow} \bigsqcup_{\substack{  \mu \supset \lambda \\ \mu/\lambda \ \mathrm{hor.} \ m-\mathrm{strip} }} \SSYT_n(\mu) ,
    \label{e:RSK-pieri}
  \end{equation}
  where $\mu/\lambda$ is the set-theoretic difference of the diagrams of $\mu$ and $\lambda$.
  \label{thm:RSK-pieri}
\end{theorem}

As with the simpler case, the image of the bijection in Theorem~\ref{thm:RSK-pieri} induced by RSK insertion is \emph{disjoint}, so taking generating polynomials we obtain Pieri's rule for multiplying Schur polynomials \cite{Pie93} as an immediate corollary.

\begin{theorem}[\cite{Pie93}]
  For a partition $\lambda$ and $m>0$ an integer, we have
  \begin{equation}
    s_{\lambda} s_{(m)} = \sum_{\mu/\lambda \ \mathrm{hor.} \ m-\mathrm{strip}} s_{\mu} .
    \label{e:pieri-hor}
  \end{equation}
  \label{thm:pieri-hor}
\end{theorem}

\begin{example}
  Consider the partition $\lambda = (3,2,2)$ and $m=2$. Adding a single cell to $\lambda$ results in one of the partitions $(3,2,2,1)$, $(3,3,2)$, $(4,2,2)$. To each of these we add an additional cell such that the two added cells lie in different columns, giving the five partitions $(3,2,2,2)$, $(3,3,2,1)$, $(4,2,2,1)$, $(4,3,2)$, $(5,2,2)$. Notice we consider only the resulting shapes and not the order in which the cells are added. For instance, we can arrive at the shape $(4,2,2,1)$ by first adding in the fourth column then in the first, or by first adding in the first column then in the fourth. Theorem~\ref{thm:pieri-hor} gives the following expansion illustrated in Fig.~\ref{fig:pieri-hor}
\[ s_{(3,2,2)} s_{(2)} = s_{(3,2,2,2)} + s_{(3,3,2,1)} + s_{(4,2,2,1)} + s_{(4,3,2)} + s_{(5,2,2)}. \]  
\end{example}

\begin{figure}[ht]
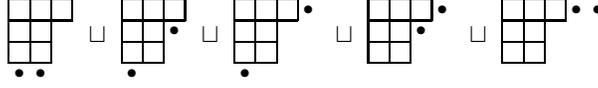

  \begin{displaymath}
    \arraycolsep=3pt
    \begin{array}{ccccccccc}
      \nulltab{ \boxify{\ } & \boxify{\ } & \boxify{\ } \\ \boxify{\ } & \boxify{\ } \\ \boxify{\ } & \boxify{\ } \\ \bullet & \bullet  }
      & \raisebox{-1\cellsize}{$\sqcup$} &
      \nulltab{ \boxify{\ } & \boxify{\ } & \boxify{\ } \\ \boxify{\ } & \boxify{\ } & \bullet \\ \boxify{\ } & \boxify{\ } \\ \bullet & } 
      & \raisebox{-1\cellsize}{$\sqcup$} &
      \nulltab{ \boxify{\ } & \boxify{\ } & \boxify{\ } & \bullet \\ \boxify{\ } & \boxify{\ } \\ \boxify{\ } & \boxify{\ } \\ \bullet & } 
      & \raisebox{-1\cellsize}{$\sqcup$} &
      \nulltab{ \boxify{\ } & \boxify{\ } & \boxify{\ } & \bullet \\ \boxify{\ } & \boxify{\ } & \bullet \\ \boxify{\ } & \boxify{\ } } 
      & \raisebox{-1\cellsize}{$\sqcup$} &
      \nulltab{ \boxify{\ } & \boxify{\ } & \boxify{\ } & \bullet & \bullet \\ \boxify{\ } & \boxify{\ } \\ \boxify{\ } & \boxify{\ } }
    \end{array}
  \end{displaymath}
  \caption{\label{fig:pieri-hor}An illustration of the Pieri rule for computing the Schur expansion of the product $s_{(3,2,2)} s_{(2)}$, where the two marked cells $(\bullet)$ denote the added horizontal $2$-strip.}
\end{figure}

Our key insertion similarly generalizes to a weight-preserving bijection 
onto a space of Kohnert diagrams whose row weights determine the monomial expansion of the product of a key polynomial with a \emph{single row} key polynomial. 

The target space of the full key-Pieri bijection may be stated as the union of the Kohnert spaces of all the ways to add $m$ cells -- whilst performing left-swaps before and after each cell addition --  to the key diagram of the weak composition with no two cells added in the same column.

\begin{definition}
  Given a weak composition $\comp{a}$, and any pair of positive integers $k \leq n$ and $m$, define
  \begin{equation}
    \KDs^{(m)}(\comp{a},k) = \bigcup_{\comp{b}} \KD(\comp{b})
    \label{e:gen-target-space}
  \end{equation}
  where the union is over all weak compositions $\comp{b}$ obtainable from a sequence of weak compositions $\comp{b}^{(0)}, \comp{b}^{(1)},\cdots,\comp{b}^{(m)} = \comp{b}$ satisfying:
  \begin{enumerate}
  \item $\comp{b}^{(0)} \preceq \comp{a}$;
  \item for each $1 \leq i \leq m$, $\comp{b}^{(i)} \preceq \comp{b}^{(i-1)} + \comp{e}_{j_i}$, where $1 \leq j_i \leq k$;
  \item $\comp{b}^{(1)}_{j_1}, \ldots, \comp{b}^{(m)}_{j_m}$ are all distinct.
  \end{enumerate}
\end{definition}

In particular, notice $\KDs^{(0)}(\comp{a},k) = \KD(\comp{a})$ and $\KDs^{(1)}(\comp{a},k) = \KDs(\comp{a},k)$. Using this, we generalize Theorem~\ref{thm:pieri-hor} to key polynomials by generalizing Theorem~\ref{thm:RSKD} to horizontal strips with the following result, proved in Section~\ref{sec:applications-bij}.

\begin{theorem}
  Given any weak composition $\comp{a}$ and any pair of positive integers $k \leq n$ and $m$, there exists a weight-preserving bijection
  \begin{equation}
    \KD(\comp{a}) \times \KD(m\comp{e}_k) \stackrel{\sim}{\longrightarrow}
    \KDs^{(m)}(\comp{a},k).
    \label{e:RSKD-hor}
  \end{equation}
  \label{thm:RSKD-hor}
\end{theorem}

\begin{example}
  Consider the weak composition $\comp{a}=(2,0,3,2)$ with $k=3$ and $m=2$. When adding cells in the key polynomial case, recall that we may move cells down to support the addition in a given row. If we again filter terms so that we avoid the redundancy when one set of Kohnert diagrams is wholly contained in another, then we have the following seven terms illustrated in Fig.~\ref{fig:pierikey-hor},
  \begin{eqnarray*}
    \scriptstyle
    \KD(2,0,3,2) \times \KD(0,0,2) & \scriptstyle = & \scriptstyle \KD(2,2,3,2) \cup \KD(2,3,3,1) \cup \KD(3,1,3,2) \cup \KD(2,1,4,2) \\ & & \scriptstyle \cup \KD(3,0,4,2) \cup \KD(2,3,4) \cup \KD(2,0,5,2).
  \end{eqnarray*}
  \label{ex:pierikey-hor}
\end{example}

\begin{figure}[ht]
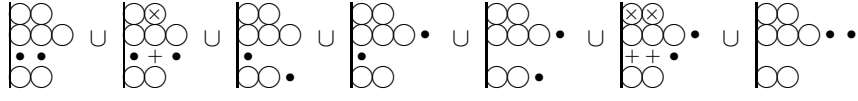

  \begin{displaymath}
    \arraycolsep=3pt
    \begin{array}{ccccccccccccc}
      \vline\nulltab{ \circify{\ } & \circify{\ } \\ \circify{\ } & \circify{\ } & \circify{\ } \\ \bullet & \bullet \\ \circify{\ } & \circify{\ } }
      & \raisebox{-1\cellsize}{$\cup$} &
      \vline\nulltab{ \circify{\ } & \circify{\times} \\ \circify{\ } & \circify{\ } & \circify{\ } \\ \bullet & + & \bullet \\ \circify{\ } & \circify{\ } }
      & \raisebox{-1\cellsize}{$\cup$} &
      \vline\nulltab{ \circify{\ } & \circify{\ } \\ \circify{\ } & \circify{\ } & \circify{\ } \\ \bullet \\ \circify{\ } & \circify{\ } & \bullet }
      & \raisebox{-1\cellsize}{$\cup$} &
      \vline\nulltab{ \circify{\ } & \circify{\ } \\ \circify{\ } & \circify{\ } & \circify{\ } & \bullet \\ \bullet \\ \circify{\ } & \circify{\ } }
      & \raisebox{-1\cellsize}{$\cup$} &
      \vline\nulltab{ \circify{\ } & \circify{\ } \\ \circify{\ } & \circify{\ } & \circify{\ } & \bullet \\ \\ \circify{\ } & \circify{\ } & \bullet }
      & \raisebox{-1\cellsize}{$\cup$} &
      \vline\nulltab{ \circify{\times} & \circify{\times} \\ \circify{\ } & \circify{\ } & \circify{\ } & \bullet \\ + & + & \bullet \\ \circify{\ } & \circify{\ } }
      & \raisebox{-1\cellsize}{$\cup$} &
      \vline\nulltab{ \circify{\ } & \circify{\ } \\ \circify{\ } & \circify{\ } & \circify{\ } & \bullet & \bullet \\ \\ \circify{\ } & \circify{\ } } 
    \end{array}
  \end{displaymath}
  \caption{\label{fig:pierikey-hor}An illustration of the Pieri rule computing the key expansion of the product $\key_{(2,0,3,2)} \key_{(0,0,2)}$, where the two marked cells $(\bullet)$ denote the added horizontal $2$-strip and crossed cells $(\times)$ drop to the marked positions $(+)$.}
\end{figure}

Generalizing Definition~\ref{def:addable}, we have the following notion of horizontal strips for key diagrams parallel to the case for Young diagrams.

\begin{definition}
  Given a weak composition $\comp{a}$ and positive integers $k$ and $m$, a \emph{$k$-addable horizontal $m$-strip for $\comp{a}$} is a sequence of positions
  \[ (c_1,r_1) \ , \ (c_2, r_2) \ , \ \ldots \ , \ (c_m,r_m) \]
  such that the columns indices $c_1,c_2,\ldots,c_m$ are distinct and $(c_i,r_i)$ is a $k$-addable cell for $\comp{a}^{(i-1)}$, where $\comp{a}^{(0)} = \comp{a}$ and $\comp{a}^{(i)} = \supp_{\comp{a}^{(i-1)}}^{(c_{i},r_{i})} + \comp{e}_{r_{i}}$ for $i>0$. 
  \label{def:k-add-hor}
\end{definition}

Note we may always take $c_1 < c_2 < \cdots < c_m$ in Definition~\ref{def:k-add-hor}. This concept of horizontal strips allows us to tighten the expansion on the right hand side of Theorem~\ref{thm:RSKD-hor}, analogous to Theorem~\ref{thm:k-addable}.

\begin{corollary}
  Given a weak composition $\comp{a}$ and positive integers $k \leq n$ and $m$, there exists a weight-preserving bijection
  \begin{equation}
    \KD(\comp{a}) \times \KD(m\comp{e}_k) \stackrel{\sim}{\longrightarrow}
    \bigcup_{\substack{
        (c_1,r_1),\ldots,(c_m,r_m) \\ k-\text{addable $m$-strip}  }}
    \KD(\comp{a}^{(m)}) ,
    \label{e:k-add-hor}
  \end{equation}
  where $\comp{a}^{(0)} = \comp{a}$ and $\comp{a}^{(i)} = \supp_{\comp{a}^{(i-1)}}^{(c_{i},r_{i})} + \comp{e}_{r_{i}}$ for $1 \leq i \leq m$.
  \label{cor:k-add-hor}
\end{corollary}

Revisiting Ex.~\ref{ex:pierikey-hor}, we have the key polynomial expansion
\begin{eqnarray*}
  \key_{(2,0,3,2)} \key_{(0,0,2)} & = & \key_{(2,2,3,2)} + \key_{(2,3,3,1)} + \key_{(3,1,3,2)} - \key_{(2,2,3,1)} + \key_{(2,1,4,2)} \\
  & & + \key_{(3,0,4,2)} + \key_{(2,3,4)} - \key_{(3,2,4)} + \key_{(2,0,5,2)}.
\end{eqnarray*}

\subsection{Iterated bijections}
\label{sec:applications-bij}

  By Lemma~\ref{lem:3.7} and Proposition~\ref{prop:lswap}, for a generic Kohnert diagram $T$, the key diagram $\kd_{\thread(T)}$ belongs to any set $\KD(D)$ to which $T$ belongs. Consequently, for any weak composition $\comp{a}$, and any pair of positive integers $k \leq n$ and $m$, we have $T \in \KDs^{(m)}(\comp{a},k)$ if and only if $\kd_{\thread(T)} \in \KDs^{(m)}(\comp{a},k)$. This useful fact will be invoked ubiquitously, albeit implicitly, throughout this section.

Since left-swaps and, more generally, Kohnert moves do not change column weights, it follows for each diagram $U \in \KDs^{(m)}(\comp{a},k)$, there exists a unique \emph{added column m-set} $\{c_1,c_2,\cdots,c_m\}$ satisfying
$$
\cwt(U) = \cwt(\comp{a}) + \sum_{i=1}^m \comp{e}_{c_i}.
$$
%
%
%
We immediately observe from the definitions that for each $m$, the target space $\KDs^{(m+1)}(\comp{a},k)$ can be obtained from $\KDs^{(m)}(\comp{a},k)$ through the following recursion:
\begin{equation}
\KDs^{(m+1)}(\comp{a},k) =
\bigcup_{\substack{
\kd_{\comp{b}} \in \KDs^{(m)}(\comp{a},k) \\
1 \leq c_1 < c_2 < \cdots < c_m 
}}
\left\{
U \in \KDs(\comp{b},k) 
\mathrel{\Bigg|} \cwt(U) = \cwt(\comp{a}) + \sum_{i=1}^m \comp{e}_{c_i}. 
\right\}
    \label{e:hor-recursion-def}
\end{equation}
When $k=1$, the above recursion simplifies to
\begin{equation}
\KDs^{(m+1)}(\comp{a},1) = \bigcup_{\kd_{\comp{b}} \in \KDs^{(m)}(\comp{a},k)} \KDs(\comp{b},1).
    \label{e:hor-recursion-def-1}
\end{equation}
We use
Recursions~\eqref{e:hor-recursion-def}~and~\eqref{e:hor-recursion-def-1} 
to prove Theorem~\ref{thm:RSKD-hor} in two special cases that parallel our work in Section~\ref{sec:bijection}. 
The first case we explore involves iterating the bottom insertion map $\ins_1$.

\begin{theorem}
  For each weak composition $\comp{a}$ and positive integer $m$, iterating the bottom insertion map $\ins_1$ induces a weight-preserving bijection
  \begin{equation}
    \KD(\comp{a}) \times \KD(m \comp{e}_1) \stackrel{\sim}{\longrightarrow} \KDs^{(m)}(\comp{a},1).
        \label{e:insert-bot-hor}
  \end{equation}
  In particular, Theorem~\ref{thm:RSKD-hor} is proved for $k=1$.
  \label{thm:insertion-bottom-hor}
\end{theorem}

\begin{proof}
We prove the theorem by induction on $m$. For our base cases, $m=0$ is trivial, and $m=1$ follows from Theorem~\ref{thm:insertion-bottom}.

We now assume for our induction hypothesis that for some $m \leq 1$, iterating the bottom insertion map $\ins_{1}$ induces a weight-preserving bijection of the form Eq.~\eqref{e:insert-bot-hor}.
Let $T \in \KD(\comp{a})$, and consider the diagrams $U^* = \ins_1^m(T)$ and $U = \ins_1(U^*) = \ins_1^{m+1}(T)$.
Since each iteration of $\ins_1$ is (row) weight-preserving, a quick induction argument gives us $\wt(U) = \wt(T) + (m+1)\comp{e}_1$. 

We wish to show $U \in \KDs^{(m+1)}(\comp{a},1)$. To do so, we invoke the induction hypothesis to obtain $U^* \in \KDs^{(m)}(\comp{a},1)$, and therefore $\kd_{\thread(U^*)} \in \KDs^{(m)}(\comp{a},k)$ as well. 
In turn, Theorem~\ref{thm:insertion-bottom} implies $U \in \KDs(\thread(U^*),1)$, 
and so following Recursion~\eqref{e:hor-recursion-def-1}, 
we have $U \in \KDs^{(m+1)}(\comp{a},1)$.

Now let $U \in \KDs^{(m+1)}(\comp{a},1)$. We want to show there exists a unique $T \in \KD(\comp{a})$ such that $U = \ins_1^{m+1}(T)$. 
From Recursion~\eqref{e:hor-recursion-def-1}, we know $U \in \KDs(\comp{b},1)$ for some weak composition $\comp{b}$ such that $\kd_{\comp{b}} \in \KDs^{(m)}(\comp{a},1)$. 
So by Theorem~\ref{thm:insertion-bottom}, there exists a unique diagram $U^* \in \KD(\comp{b}) \subset \KDs^{(m)}(\comp{a},k)$ such that $U = \ins_1(U^*)$.
We then invoke our induction hypothesis to conclude there exists a unique diagram $T \in \KD(\comp{a})$ such that $U = \ins_1(\ins_1^m(T)) = \ins_1^{m+1}(T)$.  
\end{proof}

The second special case we explore involves iterating the \emph{top} insertion map $\ins_{\infty}$, where we make use of Theorem~\ref{thm:RSK-rect} to prove the following.

\begin{theorem}
  Let $\comp{a}$ be a weak composition, and set $\ell = \max_i\{a_i > 0\}$. For every positive integer $\ell \leq k \leq n$,
  iterating the top insertion map $\ins_{\infty}$ induces a weight-preserving injection
  \begin{equation}
    \KD(\comp{a}) \times \KD(m \comp{e}_k) \hookrightarrow \KDs^{(m)}(\comp{a},k).
    \label{e:insert-top-hor}
  \end{equation}
  \label{thm:insertion-top-hor}
\end{theorem}

\begin{proof}
We prove the theorem by induction on $m$. 
For our base cases, the result for $m=0$ is trivial, and $m=1$ follows from Theorem~\ref{thm:insertion-top}.

We now assume for our induction hypothesis that for some $m \geq 1$, iterating the top insertion map $\ins_{\infty}$ induces a weight-preserving injection of the form Eq.~\eqref{e:insert-top-hor}.
Let $(T,H) \in \KD(\comp{a}) \times \KD((m+1)\comp{e}_k)$, where
we may write 
$$H = \{(i,j_i) \mid 1 \leq i \leq m+1\}$$
for some $k \geq j_1 \geq \cdots \geq j_{m+1} > 0$.
By iterating the top insertion map $\ins_{\infty}$, we obtain the following diagrams $U^*$ and $U$:
$$
U^* = \ins_{\infty}(\ins_{\infty}( \cdots (\ins_{\infty}(T,j_1),j_2),\cdots),j_{m}),
$$
$$
U = \ins_{\infty}(U^*,j_{m+1}).
$$
Since each iteration of $\ins_{\infty}$ is (row) weight-preserving, a quick induction argument gives us $\wt(U) = \wt(T) + \wt(H)$.

We now want to show $U \in \KDs^{(m+1)}(\comp{a},k)$.
Firstly, observe how the sub-diagram $H \setminus \{(m+1,j_{m+1})\}$ is a Kohnert diagram of $m \comp{e}_j$.
It follows by our induction hypothesis
$U^* \in \KDs^{(m)}(\comp{a},k)$, and therefore
$\kd_{\thread(U^*)} \in \KDs^{(m)}(\comp{a},k)$ as well.
In turn, Theorem~\ref{thm:insertion-top} implies $U \in \KDs(\thread(U^*),k)$. 
In light of Recursion~\eqref{e:hor-recursion-def}, it suffices to show
$$
\cwt(U) - \cwt(\comp{a}) = \sum_{i \in C} \comp{e}_i,
$$
for some column set $C \in {\mathbb{N} \choose m+1}$.

Let $\mathbb{D}$ be the weight-preserving bijection from semistandard Young tableaux to generic Kohnert diagrams obtained by moving cells in column $c$ with entry $r$ to position $(c,n+1-r)$.
From Theorem~\ref{thm:RSK-rect}, we get
$$
U = \mathbb{D}(\mathrm{RSK}(\Tab(T), w)),
$$
where 
$$
w = (n - j_1 + 1), (n - j_2 + 1), \cdots, (n - j_{m+1} + 1)
$$
is precisely the reading word of the tableau $\Tab(H)$, and
where the right hand side of the equation denotes the result of inserting $w$
into the tableau $\Tab(T)$ via $\mathrm{RSK}$.
By Proposition~\ref{prop:sort}, we have $\Tab(T) \in \SSYT_n(\lambda)$, where $\lambda = \mathrm{sort}(\comp{a})$, and $\Tab(H) \in \SSYT_n(m+1)$.
Thus, Theorem~\ref{thm:RSK-pieri} implies $\Tab(U) \in \SSYT_n(\mu)$, for some partition $\mu \supset \lambda$ with $\mu / \lambda$ a horizontal $(m+1)$-strip.
Since $\cwt(U) = \cwt(\mu)$ and $\cwt(\comp{a}) = \cwt(\lambda)$, 
we conclude
$$
\cwt(U) - \cwt(\comp{a}) = \cwt(\mu) - \cwt(\lambda) = \sum_{i=1}^{m+1} \comp{e}_{c_i},
$$
for some column indices $c_1, c_2, \cdots, c_{m+1}$, which all have to be distinct since $\mu / \lambda$ is a horizontal $(m+1)$-strip. 
Therefore, $U \in \KDs^{(m+1)}(\comp{a},k)$.

We have left to show $(T,H)$ is the unique pre-image of $U$ under iterated application of the top insertion $\ins_{\infty}$.
Let $(T',H') \in \KD(\comp{a}) \times \KD((m+1) \comp{e}_k)$ be one such pre-image of $U$.
Following the same arguments as we have done regarding $T$ and $H$, we conclude
$(\Tab(T'),\Tab(H')) \in \SSYT_n(\lambda) \times \SSYT_n(m+1)$ is a pre-image of $\Tab(U)$ under RSK-insertion.
But by Theorem~\ref{thm:RSK-pieri} again, $(\Tab(T),\Tab(H)) $ is the unique pre-image of $\Tab(U)$.
Hence, we must have $(T',H') = (T,H)$.
\end{proof}

We may, if we wish, continue on to show the map described in Theorem~\ref{thm:insertion-top-hor} is a bijection and hence prove Theorem~\ref{thm:RSKD-hor} for $k \geq \ell$. 
However, doing so at this juncture will not be necessary.
Moving forward, we will describe weight-preserving injective maps in the reverse direction for all values of $k>1$, and so
our current result of injectivity in the forward direction for sufficiently large values of $k$ is actually sufficient for proving the desired theorem in its entirety, after which a separate proof of surjectivity will have been rendered redundant. 

Given a weak composition $\comp{a}$ and fixed positive integer $m$,
we now proceed to stratify the target spaces 
$\KDs^{(m)}(\comp{a},k)$ for each positive integer $k$
in a similar fashion as we have done to the target spaces $\KDs(\comp{a},k)$ in Section~\ref{sec:stratify}.
As is the case when $m=1$, these target spaces in general form a nested sequence of spaces of Kohnert diagrams,
$$
\cdots \supset \KDs^{(m)}(\comp{a},k+1) \supset \KDs^{(m)}(\comp{a},k) \supset \cdots \supset \KDs^{(m)}(\comp{a},1).
$$
We use this to \emph{stratify} the target space of our desired bijection by
\begin{equation}
    \KDstrm = \KDs^{(m)}(\comp{a},k) \setminus \KDs^{(m)}(\comp{a},k-1).
\end{equation}
For each pair of positive integers $1 < k \leq n$ and $m$, we define an injective map
\begin{equation}
    \partial^{(m)}_{\comp{a},k}: \KDstrm \longrightarrow \KDs^{(m-1)}(\comp{a},k)
\end{equation}
satisfying $\wt(U) = \wt(\partial^{(m)}_{\comp{a},k}(U)) + \comp{e}_k$, and we show how these maps together with the iterated top and bottom insertions prove Theorem~\ref{thm:RSKD-hor}.

Thankfully, we need not build these \emph{degree-$m$} stratum maps from scratch.
It turns out the stratum maps of any fixed degree $m$ can be described as a degree-$1$ stratum map, which we defined in Section~\ref{sec:stratify}. That is, given a weak composition $\comp{a}$ and positive integers $1 < k \leq n$ and $m$, there exists for each Kohnert diagram $U \in \KDstrm$ a weak composition $\comp{b}$ with $\kd_{\comp{b}} \in \KDs^{(m-1)}(\comp{a},k)$ such that
\begin{equation}
\partial^{(m)}_{\comp{a},k}(U) = \partial_{\comp{b},k}(U).
    \label{e:gen-strat}
\end{equation}
Much of the rest of this section is devoted to identifying an appropriate weak composition $\comp{b}$ for which Eq.~\eqref{e:gen-strat} holds.
That the general stratum maps are injective thereafter follows from the same logic used in Section~\ref{sec:stratify-inject} to prove the stratum maps at $m=1$ are injective.

Recall Definition~\ref{def:excision}, in which the excised weight of a Kohnert diagram $U$ in the $k$th stratum $\KDstr$ is obtained from an appropriate decomposition of the key diagram $\kd_{\thread(U)}$.
We generalize this notion and find for each Kohnert diagram $U$ in the $k$th degree-$m$ stratum $\KDstrm$ a suitable decomposition of the key diagram $\kd_{\thread(U)}$.
To that end, we make the following key observation. 

\begin{lemma}
Let $K$ be a key diagram that admits a nontrivial decomposition $K = T \sqcup Y$, for some generic Kohnert diagram $T$.
If $y \in Y$ is the highest cell in the leftmost column occupied by $Y$, 
then
$T \sqcup \{y\} \in \KD(\thread(T) + \comp{e}_{\row(y)})$.
    \label{lem:key-decomp}
\end{lemma}

\begin{proof}
If $y$ is in the first column, then we are done.
So suppose $y$ is not in the first column. Since $K$ is a key diagram, our choice of $y$ guarantees 
that for every cell of $T \sqcup \{y\}$ weakly to the left of $y$ and not in column $1$, the position to its immediate left is occupied by $T$.
Consequently, the path of the cell $x$ of $T$ to  the immediate left of $y$ begins at $x$ and is anchored at the row of $y$. 
We may then match $y$ into $x$ to obtain a matching sequence $\Mtch_{\thd}(T) \cup (x \leftarrow y)$ on $T \sqcup \{y\}$ with anchor weight $\thread(T) + \comp{e}_{\row(y)}$. 
Therefore, by Theorem~\ref{thm:TFAE_kohnert}, 
the diagram $T \sqcup \{y\}$ is a Kohnert diagram of $\thread(T) + \comp{e}_{\row(y)}$.
\end{proof}

We use Lemma~\ref{lem:key-decomp} in an inductive argument to establish a criterion that every Kohnert diagram in the target space of our desired bijection must satisfy.

\begin{definition}
Let $U$ be a diagram.
Let $\comp{a}$ be a weak composition and let $k \leq n$ and $m$ be positive integers.
Suppose $U$ admits a decomposition $U = T \sqcup Y$, where $T \in \KD(\comp{a})$ and $Y$ is a horizontal $m$-strip whose cells are all weakly below row $k$. 
Then we say that the decomposition $U = T \sqcup Y$ is a \emph{drop decomposition} of $U$ (with respect to $\comp{a}$, $k$ and $m$),
and we refer to the cells of $Y$ as the \emph{added cells} associated to the drop decomposition.
\end{definition}

\begin{theorem}
Let $\comp{a}$ be a weak composition, and let $k \leq n$ and $m$ be positive integers. 
Suppose $U$ is a generic Kohnert diagram.

If 
the key diagram $\kd_{\thread(U)}$ 
admits a drop decomposition with respect to $\comp{a}$, $k$ and $m$, 
then $U \in \KDs^{(m)}(\comp{a},k)$. 
Conversely, if $U \in \KDs^{(m)}(\comp{a},k)$, then $U$
admits a drop decomposition with respect to $\comp{a}$, $k$ and $m$, and each added cell associated to the drop decomposition occupies a column in the added column set of $U$. 
    \label{thm:gen-target-criterion}
\end{theorem}

\begin{proof}
Suppose $K = T \sqcup \{y_1,y_2,\cdots,y_m\}$ drop decomposition with respect to $\comp{a}$, $k$ and $m$, where the cells $y_1,y_2,\cdots,y_m$ are labeled left-to-right in $K$. 
Let $\comp{b}^{(0)} = \thread(T)$, and for each $1 \leq i \leq m$ define $\comp{b}^{(i)} = \thread(T \sqcup \{y_1,\cdots,y_i\})$.
By the definition of a drop decomposition, $T \in \KD(\comp{a})$ which implies $\comp{b}^{(0)} \preceq \comp{a}$ by Lemma~\ref{lem:3.7}.
For each $1 \leq i \leq m$, we have $\comp{b}^{(i)} \preceq \comp{b}^{(i-1)} + \comp{e}_{\row(y_i)}$
by Lemma~\ref{lem:key-decomp}, where by the definition of the drop decomposition, each $y_i$ is weakly below row $k$. 
Note $\thread(U) = \thread(K) = \comp{b}^{(m)}$, so
$U \in \KD(\comp{b}^{(m)})$.
Following the indexing set of the right-hand side of Eq.~\eqref{e:RSKD-hor} of Theorem~\ref{thm:RSKD-hor}, 
we conclude $U \in \KDs^{(m)}(\comp{a},k)$.

On the other hand, suppose $U \in \KDs^{(m)}(\comp{a},k)$.
Following the indexing set of the right-hand side of Eq.~\eqref{e:RSKD-hor} of Theorem~\ref{thm:RSKD-hor}, 
there exists a sequence of weak compositions $\comp{b}^{(0)}, \comp{b}^{(1)}, \cdots, \comp{b}^{(m)}$ satisfying
\begin{enumerate}
    \item $\comp{b}^{(0)} \preceq \comp{a}$,
    \item for each $1 \leq i \leq m$, $\comp{b}^{(i)} \preceq \comp{b}^{(i-1)} + \comp{e}_{j_i}$, where $1 \leq j_i \leq k$,
\end{enumerate}
such that $U \in \KD(\comp{b}^{(m)})$.
by Proposition~\ref{prop:lswap}. 

We now set $T^{(m)} = U$, and we proceed by downward induction to show for every $1 \leq i \leq m$ there exists a decomposition $T^{(i)} = T^{(i-1)} \sqcup \{y_i\}$ such that the
sub-diagram $T^{(i-1)}$ is a Kohnert diagram of $\comp{b}^{(i-1)}$ and the cell $y_i$ is weakly below row $k$.
We assume $T^{(i)}$ is a Kohnert diagram of $\comp{b}^{(i)}$, which already holds for our base case $i = m$.
By condition~(2) above, we have $\comp{b}^{(i)} \preceq \comp{b}^{(i-1)} + \comp{e}_{j_i}$, and so by applying Lemma~\ref{lem:3.7} twice, we infer $T^{(i)} \in \KD(\comp{b}^{(i-1)} + \comp{e}_{j_i})$.
By Definition~\ref{def:kohnert-label}, we may consider the Kohnert labeling $\Lbl_{\comp{b}^{(i-1)} + \comp{e}_{j_i}}$ on $T$. This labeling induces a decomposition $T^{(i)} = T^{(i-1)} \sqcup \{y_i\}$, where $\Lbl_{\comp{b}^{(i-1)} + \comp{e}_{j_i}} \big|_{T^{(i-1)}}$ has content weight $\comp{b}^{(i-1)}$ and where $y_i$ has label $j_i$. Since $j_i \leq k$ by condition~(2), $y_i$ is weakly below row $k$ from the definition of labelings. The same argument used in the proof of Proposition~\ref{prop:label-wt} applies to the underlying matching sequence $M$ for $\Lbl_{\comp{b}^{(i-1)} + \comp{e}_{j_i}}$ restricted to $T^{(i-1)}$, and so $\wt(M) \preceq \comp{b}^{(i-1)}$. Thus by Theorem~\ref{thm:TFAE_kohnert}, $T^{(i-1)}$ is a Kohnert diagram of $\comp{b}^{(i-1)}$, and we are done.

By induction, we obtain a decomposition $U = T^{(0)} \sqcup \{y_1,y_2,\cdots,y_m\}$, where $T^{(0)}$ is a Kohnert diagram of $\comp{b}^{(0)}$ and the cells $y_1,y_2,\cdots,y_m$ are all weakly below row $k$.
Since $\comp{b}^{(0)} \preceq \comp{a}$ by condition~(1), we again apply Lemma~\ref{lem:3.7} twice to infer $T^{(0)} \in \KD(\comp{a})$. 
So in order for $T^{(0)} \sqcup \{y_1,y_2,\cdots,y_m\}$ to be a drop decomposition of $U$ with respect to $\comp{a}$, $k$ and $m$, we have left to show the cells $y_i$ are all in distinct columns. 
Since $U \in \KDs^{(m)}(\comp{a},k)$, we must have $\cwt(U) - \cwt(\comp{a}) = \sum_{i=1}^m \comp{e}_{c_i}$, where $\{c_1,c_2,\cdots,c_m\}$ is the added column set of $U$. 
Since Kohnert moves preserve column weights, we have $\cwt(T^{(0)}) = \cwt(\comp{a})$ and $\cwt(K) = \cwt(U)$. 
Therefore,
$$
\sum_{i=1}^m \comp{e}_{\col(y_i)} = \cwt(K) - \cwt(T^{(0)}) = \cwt(U) - \cwt(\comp{a}) = \sum_{i=1}^m \comp{e}_{c_i}.
$$
For the above equalities to hold, the cells $y_1,y_2,\cdots,y_m$ must each occupy a distinct column in the added column set of $U$. 
This concludes our proof.
\end{proof}

Theorem~\ref{thm:gen-target-criterion} is instrumental in proving a criterion that determines when a Kohnert diagram of $\KDs^{(m)}(\comp{a},k)$ is actually a diagram of $\KDs^{(m)}(\comp{a},k-1)$.
To describe this criterion, the following notation will be convenient.

\begin{definition}
Let $T$ be an arbitrary diagram. For each position $(c,r)$, define
\begin{equation}
    \droppable_{c,r}(T) = \# \{ x \in T \mid \col(x) = c, \row(x) \geq r \}.
        \label{e:droppable-diagram}
\end{equation}
Abusing notation,
given a weak composition $\comp{a}$ and positive integers $c$ and $r$, define
\begin{equation}
    \droppable_{c,r}(\comp{a}) = \droppable_{c,r}(\kd_{\comp{a}}) = \# \{ j \geq r \mid a_j \geq c \}. 
        \label{e:droppable-comp}
\end{equation}
\end{definition}

The above notation is highly compatible with Kohnert moves, which push cells down. Namely, if $S \prec T$ then $\droppable_{c,r}(S) \leq \droppable_{c,r}(T)$ for each position $(c,r)$; by extension, if $T$ is a Kohnert diagram of a weak composition $\comp{a}$, then $\droppable_{c,r}(T) \leq \droppable_{c,r}(\comp{a})$ for each pair of positive integers $c$ and $r$. 
We may additionally state this observation purely in terms of weak compositions, as Proposition~\ref{prop:lswap} implies for each pair of weak compositions $\comp{b} \preceq \comp{a}$, and for each pair of positive integers $c$ and $r$, we have $\droppable_{c,r}(\comp{b}) \leq \droppable_{c,r}(\comp{a})$. 



\begin{lemma}
Let $\comp{a}$ be a weak composition, and let $1 < k \leq n$ and $m$ be positive integers.
Suppose $U \in \KDs^{(m)}(\comp{a},k)$.

If every column $c$ in the added column set of $U$ satisfies
\begin{equation}
\droppable_{c,k}(\thread(U)) \leq \droppable_{c,k}(\comp{a}),
    \label{e:gen-nonstrat-criterion}
\end{equation}
then $U \in \KDs^{(m)}(\comp{a},k-1)$.
    \label{lem:gen-nonstrat-criterion}
\end{lemma}

\begin{proof}
Let $U \in \KDs^{(m)}(\comp{a},k)$ and suppose every column in the added column set of $U$ satisfies Eq.~\eqref{e:gen-nonstrat-criterion}.
In light of Theorem~\ref{thm:gen-target-criterion}, it suffices to show the key diagram $K = \kd_{\thread(U)}$ has a drop decomposition with respect to $\comp{a}$, $(k-1)$ and $m$. 
By way of contradiction, we suppose this is not the case. By Lemma~\ref{lem:gen-nonstrat-criterion}, we may pick a drop decomposition $K = T \sqcup Y$ of $K$ with respect to $\comp{a}$, $k$ and $m$ that has the least number of added cells occupying row $k$. 
Let $y \in Y$ be the leftmost added cell in row $k$ of the drop decomposition,
and let $c$ be the column of $y$.
We claim there exists a matching sequence $M$ on $T$ with $\wt(M) \preceq \comp{a}$ such that for some cell $y'$ of $T$ below $y$ in column $c$, the cell $M(y')$ is weakly above row $k$. 
By the definition of a drop decomposition, $T$ is a Kohnert diagram of $\comp{a}$.
Thus, by Definition~\ref{def:kohnert-label}, we may consider the Kohnert labeling $\Lbl_{\comp{a}}$ on $T$ and the underlying Kohnert matching satisfying $\wt(\Mtch_{\comp{a}}) \preceq \comp{a}$ by Proposition~\ref{prop:label-wt}. 
Let $x'$ be the cell of $T$ to the immediate left of $y$ in row $k$.
If $\plength_{\Mtch_{\comp{a}}}(x') \geq c$,
then we may take $M = \Mtch_{\comp{a}}$. Otherwise, by Lemma~\ref{lem:matchings-lengthen-row-k}, we may assume every cell of $T$ weakly left of $y$ in row $k$ belongs to a path component in $\Mtch_{\comp{a}}$ strictly shorter than $c$. 
Since $c$ is in the added column set of $U$ by Theorem~\ref{thm:gen-target-criterion} and therefore satisfies Eq.~\eqref{e:gen-nonstrat-criterion}, we have $\droppable_{c,k}(T) < \droppable_{c,k}(\comp{a})$.
So there exists a cell $y'$ below $y$ in its column with label less than $k$. 
Let $i$ be the leftmost column such that the cell on the path component of $\Mtch_{\comp{a}}$ containing $y'$ lies strictly below row $k$.
If $i = c$, then we may again take $M = \Mtch_{\comp{a}}$.
Else, if $i < c$, then let $w$ be the cell in position $(i,k)$. 
By our assumption about the lengths of paths in $\Mtch_{\comp{a}}$ for cells left of $y$ in row $k$, we have $\plength_{\Mtch_{\comp{a}}}(w) < c \leq \plength_{\Mtch_{\comp{a}}}(y')$.
Thus there exists some maximal column $t \geq i$ such that the cells on the path component of $\Mtch_{\comp{a}}$ containing $w$ lie above the cells on the path component containing $y'$ for every column $s$ between $i$ and $t$. 
Define $L'$ to be the labeling derived from $\Lbl_{\comp{a}}$ by swapping the labels of the cells in the path components of $\Mtch_{\comp{a}}$ containing $w$ and $y'$ for every column $s$ between $i$ and $t$. The same argument used in the proof of Proposition~\ref{prop:label-wt} applies to the underlying matching sequence of $L'$, and so we obtain a matching sequence $\Mtch_{L'}$ on $T$ with $\wt(\Mtch_{L'}) \preceq \comp{a}$ such that $\plength_{\Mtch_{L'}}(w) = \plength_{\Mtch_{\comp{a}}}(y') \geq c$. 
By Lemma~\ref{lem:matchings-lengthen-row-k}, 
we subsequently get a matching sequence $M$ with $\wt(M) \preceq \wt(\Mtch_{L'}) \preceq \comp{a}$ such that $\plength_M(x') \geq c$, proving the claim.

Let $M$ be any matching sequence on $T$ with $\wt(M) \preceq \comp{a}$ such that for some cell $y'$ of $T$ below $y$ in column $c$, the cell $M(y')$ is weakly above row $k$. 
We will show interchanging $y$ with $y'$ as an added cell gives us a drop decomposition of $K$ (with respect to $\comp{a}$, $k$ and $m$) with one fewer added cell occupying row $k$ than the one we picked, at which point we arrive at a contradiction.
Define a matching sequence $M'$ on the diagram $T' = (T \setminus \{y'\}) \sqcup \{y\}$ from $M$ by setting $M'(y) = M(y')$, $M'(x) = y$ if $M(x) = y'$, and $M'(x) = M(x)$ otherwise. Then $\wt(M) = \wt(M) \preceq \comp{a}$.
Then $T'$ is a Kohnert diagram of $\comp{a}$ by Theorem~\ref{thm:TFAE_kohnert}.
Since $y$ and $y'$ lie in the same column $c$, the diagram $Y' = (Y \setminus \{y\}) \sqcup \{y'\}$ is a horizontal $m$-strip. 
It follows $K = T' \sqcup Y'$ is indeed a drop decomposition with respect to $\comp{a}$, $k$ and $m$, but $Y'$ has one fewer cell in row $k$ than $Y$, which contradicts our minimality assumption about the drop decomposition $U = T \sqcup Y$.
\end{proof}

\begin{lemma}
Let $\comp{a}$ be a weak composition, and let $1 < k \leq n$ and $m$ be positive integers. If $U \in \KDstrm$, then the following conditions hold.
\begin{enumerate}
    \item there exists a column $c$ in the added column set of $U$ such that
    \begin{equation}
        \droppable_{c,k}(\thread(U)) > \droppable_{c,k}(\comp{a}),
            \label{e:added-col-strat}
    \end{equation}
    \item for every column $c$ in the added column set of $U$ satisfying Eq.~\eqref{e:added-col-strat}, and for every drop decomposition of $\kd_{\thread(U)}$ with respect to $\comp{a}$, $k$ and $m$, the added cell occupying column $c$ is in position $(c,k)$.
\end{enumerate}
    \label{lem:added-col-strat}
\end{lemma}

\begin{proof}
Condition~(1) immediately follows from Lemma~\ref{lem:gen-nonstrat-criterion}. Now let $c$ be a column in the added column set of $U$ satisfying Eq.~\eqref{e:added-col-strat}, and let $\kd_{\thread(U)} = T \sqcup Y$ be a drop decomposition of $\kd_{\thread(U)}$ with respect to $\comp{a}$, $k$ and $m$. 
Let $y \in Y$ be the added cell of the drop decomposition in column $c$.
Since $T \in \KD(\comp{a})$, we necessarily have 
$
\droppable_{c,k}(T) \leq \droppable_{c,k}(\comp{a}).
$
If $y$ is not in row $k$, then since $y$ is the only cell in column $c$ of $\kd_{\thread(U)} \setminus T$, it follows
$$
\droppable_{c,k}(\thread(U)) = \droppable_{c,k}(T) \leq \droppable_{c,k}(\comp{a}).
$$
This directly contradicts our assumption of $c$ satisfying Eq.~\eqref{e:added-col-strat}. 
Therefore, $y$ must be in row $k$. 
\end{proof}

\begin{lemma}
Let $\comp{a}$ be a weak composition, and let $1 < k \leq n$ and $m$ be positive integers.
Suppose $U \in \KDstrm$, and let
$c$ be the rightmost column in the added column set of $U$ satisfying Eq.~\eqref{e:added-col-strat} of Lemma~\ref{lem:added-col-strat}.

Then there exists a drop decomposition of $\kd_{\thread(U)}$ with respect to $\comp{a}$, $k$ and $m$ such that the rightmost added cell in row $k$ is at column $c$.
    \label{lem:added-col-strat-rightmost}
\end{lemma}

\begin{proof}
Let $U \in \KDstrm$ and let $c$ be the rightmost column in the added column set of $U$ satisfying Eq.~\eqref{e:added-col-strat} of Lemma~\ref{lem:added-col-strat}. 
Pick any drop decomposition $K = T \sqcup \{y_1,y_2,\cdots,y_m\}$ of $K = \kd_{\thread(U)}$ with respect to $\comp{a}$, $k$ and $m$, where the cells $y_1,y_2,\cdots,y_m$ are labeled in order from left to right. 
Let $y_p$ for some $1 \leq p \leq m$ be the added cell in column $c$.
If $p = m$, then since $y_p$ is necessarily in row $k$ by Lemma~\ref{lem:added-col-strat}~(2),
the lemma is trivially satisfied.
So assume without loss of generality $p < m$. 
So $T \in \KD(\comp{a})$ and every cell $y_1,y_2,\cdots,y_m$ is weakly below row $k$.
Since $T \in \KD(\comp{a})$, we have $\thread(T) \preceq \comp{a}$ by Lemma~\ref{lem:3.7}.
By means of induction, we obtain via Lemma~\ref{lem:key-decomp} for each $1 \leq i \leq p$, $\thread(T \sqcup \{y_1,\cdots,y_i\}) \preceq \thread(T \sqcup \{y_1,\cdots,y_{i-1}\} + \comp{e}_{\row(y_i)}$, where each $y_i$ is weakly below row $k$.
Thus, the diagram $T' = T \sqcup \{y_1,\cdots,y_p\}$ is in $\KDs^{(p)}(\comp{a},k)$.

We now take a slight detour and consider the thread weight $\comp{a'} = \thread(T')$. 
Note $K = T' \sqcup \{y_{p+1},\cdots,y_m\}$ is a drop decomposition with respect to $\comp{a'}$, $k$ and $m$, so 
by Theorem~\ref{thm:gen-target-criterion}, 
$U \in \KDs^{(m-p)}(\comp{a'},k)$, with added column set given by the column indices of $y_{p+1},\cdots,y_m$.
We claim $U \in \KDs^{(m-p)}(\comp{a'},k-1)$,
where in light of Lemma~\ref{lem:gen-nonstrat-criterion},
it suffices to show
$$
\droppable_{\col(y_i),k}(\thread(U)) \leq \droppable_{\col(y_i),k}(\comp{a'}), 
$$
for every $p < i \leq m$.
From our assumption about $c$, we know
$$
\droppable_{\col(y_i),k}(\thread(U)) \leq \droppable_{\col(y_i),k}(\comp{a}),
$$
for every $p < i \leq m$.
It therefore further suffices to show 
\begin{equation*}
\droppable_{\col(y_i),k}(\comp{a'}) \leq \droppable_{\col(y_i),k}(\comp{a}),
\end{equation*}
for every $p < i \leq m$.
We shall in fact prove an even stronger statement. Namely, 
\begin{equation}
\droppable_{i,j}(\comp{a'}) \leq \droppable_{i,j}(\comp{a}), 
    \label{e:gen-target-droppable}
\end{equation}
for every position $(i,j)$ with $i > c$. 


Fix a position $(i,j)$ with $i > c$.
For ease of notation, let $T_0 = T$ and for $1 \leq i' \leq p$, let $T_{i'} = T \sqcup \{y_1,\cdots,y_{i'}\}$.
Firstly, since $\thread(T_0) \preceq \KD(\comp{a})$, we have $\droppable_{i,j}(\thread(T_0)) \leq \droppable_{i,j}(\comp{a})$.
Then, for each $1 \leq i' \leq p$, 
since $\thread(T_{i'}) \preceq \thread(T_{i'-1}) + \comp{e}_{\row(y_{i'})}$, we have $\droppable_{i,j}(\thread(T_{i'}) \leq \droppable_{i,j}(\thread(T_{i'-1}) + \comp{e}_{\row(y_{i'})})$.
We also have 
\begin{eqnarray*}
\cwt(\kd_{\thread(T_{i'-1}) + \comp{e}_{\row(y_{i'})}}) &=&
\cwt(\thread(T_{i'-1}) + \comp{e}_{\row(y_{i'})}) = \cwt(\thread(T_{i'})) \\
&=& \cwt(T_{i'}) = \cwt(T_{i'-1} \sqcup \{y_{i'}\}) \\
&=& \cwt(T_{i'-1}) + \comp{e}_{\col(y_{i'})}
= \cwt(\thread(T_{i'-1})) + \comp{e}_{\col(y_{i'})} \\
&=& \cwt(\kd_{\thread(T_{i'-1})}) + \comp{e}_{\col(y_{i'})},
\end{eqnarray*}
which implies 
$\kd_{\thread(T_{i'-1}) + \comp{e}_{\row(y_{i'})}} = \kd_{\thread(T_{i'-1})} \sqcup \{y_{i'}\}$.
Hence, 
\begin{eqnarray*}
\droppable_{i,j}(\thread(T_{i'-1}) + \comp{e}_{\row(y_{i'})}) 
= \droppable_{i,j}(\kd_{\thread(T_{i'-1}) + \comp{e}_{\row(y_{i'})}}) = \droppable_{i,j}(\kd_{\thread(T_{i'-1})} \sqcup \{y_{i'}\}),
\end{eqnarray*}
where
since $y_{i'}$ is weakly to the left of column $c$, and since $c < i$, 
$$
\droppable_{i,j}(\kd_{\thread(T_{i'-1})} \sqcup \{y_{i'}\}) 
= \droppable_{i,j}(\kd_{\thread(T_{i'-1})})
= \droppable_{i,j}(\thread(T_{i'-1})).
$$
It follows $\droppable_{i,j}(\thread(T_{i'})) \leq \droppable_{i,j}(\thread(T_{i'-1}))$.
By induction, we get
$$
\droppable_{i,j}(\comp{a'}) = \droppable_{i,j}(\thread(T_{p})) \leq \droppable_{i,j}(\comp{a}),
$$
and Eq.~\eqref{e:gen-target-droppable} is satisfied.
Therefore, $U \in \KDs^{(m-p)}(\comp{a'},k-1)$.


We deduce the key diagram $K = \kd_{\thread(U)}$ is also in
$\KDs^{(m-p)}(\comp{a'},k-1)$.
By Theorem~\ref{thm:gen-target-criterion}, there exists a drop decomposition $K = S' \sqcup Y'$ of $K$ with respect to $\comp{a'}$, $(k-1)$ and $(m-p)$, where $Y'$ is a horizontal $(m-p)$-strip whose cells are weakly below row $k-1$ and occupy the same columns as $y_{p+1},\cdots,y_m$.
Now since $T' \in \KDs^{(p)}(\comp{a},k)$, it follows $\kd_{\comp{a'}}$ is in $\KDs^{(p)}(\comp{a},k)$ as well.
Therefore,
since $S' \in \KD(\comp{a'})$,
we must have $S' \in \KDs^{(p)}(\comp{a},k)$. 
Consequently, by Theorem~\ref{thm:gen-target-criterion},
there exists a drop decomposition
$S' = T'' \sqcup Y''$ of $S'$ with respect to $\comp{a}$, $k$ and $p$, where since $\cwt(S') = \cwt(\comp{a'}) = \cwt(T')$, $Y''$ is a horizontal $p$-strip occupying the same columns as $y_1,\cdots,y_p$. 
Thus, we get a drop decomposition $K = T'' \sqcup (Y' \sqcup Y'')$ of $K$ with respect to $\comp{a}$, $k$ and $m$, whose added cells occupy the same columns as $y_1,\cdots,y_m$. 
By Lemma~\ref{lem:added-col-strat}~(2), the added cell at column $c$ of this new drop decomposition of $K$ is still $y_p$.
On the other hand, all the added cells to the right of $y_p$ -- namely, the cells of $Y'$ -- are weakly below row $k-1$. 
Therefore, $y_p$ is the rightmost added cell of the drop decomposition $K = T'' \sqcup (Y' \sqcup Y'')$, and the lemma is satisfied.
\end{proof}

\begin{lemma}
Let $\comp{a}$ be a weak composition, and let $k \leq n$ and $m$ be positive integers.
Suppose $U \in \KDstrm$, and let
$c$ be the rightmost column in the added column set of $U$ satisfying Eq.~\eqref{e:added-col-strat} of Lemma~\ref{lem:added-col-strat}.
Consider the key diagram $K = \kd_{\thread(U)}$, and let $y \in K$ be the cell in position $(c,k)$. 

Then the following conditions hold:
\begin{enumerate}
    \item $K \setminus \{y\} \in \KDs^{(m-1)}(\comp{a},k)$,
    \item $U \in \overline{\KDs}(\thread(K \setminus \{y\}),k)$.
\end{enumerate}
    \label{lem:excision-hor}
\end{lemma}

\begin{proof}
By Lemma~\ref{lem:added-col-strat-rightmost} there exists a drop decomposition $K = T \sqcup Y$ with respect to $\comp{a}$, $k$ and $m$ such that $y \in Y$ is the rightmost added cell in row $k$.
So $T$ is a Kohnert diagram of $\comp{a}$ and by Theorem~\ref{thm:TFAE_kohnert} there exists a matching sequence $M$ on $T$ with $\wt(M) \preceq \comp{a}$. 

Let $S \subset T$ be the sub-diagram of cells weakly northeast of $y$.
Note $M \mid_S$ comprises of matching paths that terminate at cells above $y$ in its column. 
We then define $\hat{S}$ be the diagram obtained from $S$ by pushing
each cell $x \in S$ up to the row of the leftmost cell of the path component of $M \mid_S$ containing $x$. 
In particular, every cell in row $k$ of $S$ gets pushed up, leaving $y$ as the rightmost cell of the key diagram $\hat{K} = (K \setminus S) \sqcup \hat{S}$. 
Theorem~\ref{thm:TFAE_kohnert} implies $S \prec \hat{S}$, and therefore $K \setminus \{y\} \prec \hat{K} \setminus \{y\}$. 
To prove Statement~(1) of the lemma, it suffices to show $\hat{K} \setminus \{y\} \in \KDs^{(m-1)}(\comp{a},k)$.
Since every pair of cells $x$ and $y$ in $S$ with $M(x) = y$ end up in the same row in $\hat{S}$, we may define a matching sequence $\hat{M}$ on the diagram $\hat{T} = (T \setminus S) \sqcup \hat{S}$ from $M$ by taking $\hat{M} = M$ for all cells of $\hat{T}$. 
Since the cells that get pushed up under $S \mapsto \hat{S}$ are all to the right of $y$ and hence not in column $1$, we have $\wt(\hat{M}) = \wt(M) \preceq \comp{a}$.
It follows $\hat{T} \in \KD(\comp{a})$ by Theorem~\ref{thm:TFAE_kohnert},
and the decomposition $\hat{K} \setminus \{y\} = \hat{T} \sqcup (Y \setminus \{y\})$ is a drop decomposition with respect to $\comp{a}$, $k$ and $(m-1)$.
So by Theorem~\ref{thm:gen-target-criterion} we have $\hat{K} \setminus \{y\} \in \KDs^{(m-1)}(\comp{a},k)$.
Statement~(1) of the lemma follows.

Now let $\comp{b} = \thread(K \setminus \{y\})$.
Since $K = (K \setminus \{y\}) \sqcup \{y\}$ is a drop decomposition with respect to $\comp{b}$, $k$ and $1$, we have $U \in \KDs(\comp{b},k)$ by Theorem~\ref{thm:gen-target-criterion}. 
If Statement~(2) is false, then by Theorem~\ref{thm:gen-target-criterion}
for some cell $y'$ of $K$ below $y$ in its column, $K = (K \setminus \{y'\}) \sqcup \{y'\}$ is a drop decomposition with respect to $\comp{b}$, $k-1$ and $1$. 
It follows the diagram $T' = K \setminus \{y'\}$ is a Kohnert diagram of $\comp{b}$. 
But since $y'$ is picked to be strictly below $y$, we must have $\droppable_{\col(y),k}(T') = \droppable_{\col(y),k}(\comp{b}) + 1$, and $T'$ could not be a Kohnert diagram of $\comp{b}$. 
Having arrived at a contradiction, we conclude Statement~(2) must hold.
\end{proof}

\begin{definition}
We refer to the weak composition $\thread(K \setminus \{y\})$ in Lemma~\ref{lem:excision-hor} as the \emph{degree-$(m-1)$ excised weight} of a Kohnert diagram $U \in \KDstrm$. 
That is, given the rightmost column $c$ of the added column $m$-set of $U$ satisfying $\droppable_{c,k}(\thread(U)) > \droppable_{c,k}(\comp{a})$, the degree-$(m-1)$ excised weight of $U$ is given by $\thread(\kd_{\thread(U)} \setminus \{y\})$, where $y$ is the cell in position $(c,k)$.
    \label{def:excision-hor}
\end{definition}

The degree-$(m-1)$ excised weight is precisely the weak composition we need to define the degree-$(m-1)$ stratum maps $\partial^{(m)}_{\comp{a},k}$.

\begin{definition}
Given a weak composition $\comp{a}$ and positive integers $1 < k \leq n$ and $m$, the \emph{$k$th stratum map of degree $m$ of $\comp{a}$}, denoted by $\partial^{(m)}_{\comp{a},k}$, acts on a Kohnert diagram $U \in \KDstrm$ with degree-$(m-1)$ excised weight $\comp{b}$ by
\begin{equation}
    \partial^{(m)}_{\comp{a},k}(U) = \partial_{\comp{b},k}(U).
\end{equation}
\end{definition}

We know from Section~\ref{sec:stratify} the map $\partial^{(m)}_{\comp{a},k}$ is well-defined with $\wt(U) = \wt(\partial_{\comp{a},k}(U)) + \comp{e}_k$.
Moreover, Lemma~\ref{lem:excision-hor}~(1) together with Theorem~\ref{thm:stratum-image} inform us $\partial^{(m)}_{\comp{a},k}(U) \in \KDs^{(m-1)}(\comp{a},k)$. 

\begin{theorem}
Let $\comp{a}$ be a weak composition, and let $1 < k \leq n$ and $m$ be positive integers.
For $U \in \KDstrm$, the diagram $U$ is the unique pre-image of the diagram $\partial^{(m)}_{\comp{a},k}(U)$ under the $k$th degree-$m$ stratum map $\partial^{(m)}_{\comp{a},k}$ of $\comp{a}$. That is, the maps
$$
\partial^{(m)}_{\comp{a},k}: \KDstrm \longrightarrow \KDs^{(m-1)}(\comp{a},k)
$$
are injective.
    \label{thm:stratum-inject-hor}
\end{theorem}

\begin{proof}
Let $U \in \KDstrm$, and
let $V = \partial^{(m)}_{\comp{a},k}(U) \in \KDs^{(m-1)}(\comp{a},k)$.
Given the degree-$(m-1)$ excised weight $\comp{b}$ of $U$, 
we may completely recover $U$ from $V = \partial_{\comp{b},k}(U)$ 
via 
Lemma~\ref{lem:L-minimal-labeling}
and
Theorem~\ref{thm:moved-cells-recoverable-given-added-col}
provided we know the added column $c$ of $U$ with respect to $\comp{b}$. 
It follows if we have a different diagram $U' \in \KDstrm$ such that $\partial^{(m)}_{\comp{a},k}(V) = V$, then $U'$ must come equipped with a distinct added column $c' \neq c$ with respect to its own degree-$(m-1)$ excised weight $\comp{b'}$, where $\comp{b'} \neq \comp{b}$ by Theorem~\ref{thm:stratum-inject}.
Suppose, for contradiction, there exists $U' \in \KDstrm$ with added column $c' > c$ with respect to excised weight $\comp{b'} \neq \comp{b}$ satisfying $\partial^{(m)}_{\comp{a},k}(U') = V$. 

On one hand, since $c$ is the added column of $U$ with respect to $\comp{b}$, by comparing Definitions~\ref{def:excision}~and~\ref{def:excision-hor} we observe the excised weight of $U$ with respect to $\comp{b}$ is itself $\comp{b} = \thread( \kd_{\thread(U)} \setminus \{(c,k)\} )$. 
Hence, 
by Eq.~\eqref{eq:same-upper-longer-half-column-weight} of Lemma~\ref{lem:same-longer-parts}, we have 
$
\droppable_{c,k+1}(\thread(V)) = \droppable_{c,k+1}(\comp{b}).
$
Since $\comp{b} = \thread( \kd_{\thread(U)} \setminus \{(c,k)\} )$, we also have
$$
\droppable_{c,k+1}(\comp{b}) \geq \droppable_{c,k+1}(\kd_{\thread(U)} \setminus \{(c,k)\}) = \droppable_{c,k+1}(\thread(U)).
$$
In accordance with Lemma~\ref{lem:excision-hor}, $c$ satisfies $\droppable_{c,k}(\thread(U)) > \droppable_{c,k}(\comp{a})$.
Therefore,
$$
\droppable_{c,k+1}(\thread(V)) = \droppable_{c,k+1}(\comp{b}) = \droppable_{c,k+1}(\thread(U)) \geq \droppable_{c,k}(\comp{a}).
$$

On the other hand, since $V = \partial_{\comp{b'},k}(U')$, by Lemma~\ref{lem:image-thread-weight} we have $\thread(V) \preceq \comp{b'}$, and so $\droppable_{c,k+1}(\thread(V)) \leq \droppable_{c,k+1}(\comp{b'}$. 
We know $\kd_{\thread(U')} \setminus \{(c',k)\} \in \KDs^{(m-1)}(\comp{a},k)$ by
Lemma~\ref{lem:excision-hor}, so since
$\comp{b'} = \thread(\kd_{\thread(U')} \setminus \{(c',k)\})$,
we have $\kd_{\comp{b'}} \in \KDs^{(m-1)}(\comp{a},k)$ as well.
Thus, by Theorem~\ref{thm:gen-target-criterion}, there exists a drop decomposition $\kd_{\comp{b'}} = T \sqcup Y$ with respect to $\comp{a}$, $k$ and $(m-1)$, where $T \in \KD(\comp{a})$.
Since $\thread(V) \preceq \comp{b'}$, we have $\cwt(\comp{b'}) = \cwt(\thread(V))$.
We similarly have $\cwt(\comp{b}) = \cwt(\thread(V))$, and so $\cwt(\comp{b'}) = \cwt(\comp{b})$.
It follows $c$ is not in the added column $(m-1)$-set of $\kd_{\comp{b'}}$ with respect to $\comp{a}$, for if so, then since $c$ 
is also in the added column $(m-1)$-set of $\comp{b}$, and since $c$
is the added column of $U$ with respect to $\comp{b}$, we would be counting $c$ twice in the added column $m$-set of $U$ with respect to $\comp{a}$.
So by Theorem~\ref{thm:gen-target-criterion} every cell of $\kd_{\comp{b'}}$ in column $c$ is also a cell of $T$, which means 
$$
\droppable_{c,k}(\comp{b'}) = \droppable_{c,k}(T) \leq \droppable_{c,k}(\comp{a}).
$$
Consequently, since $b'_k = c'-1 \geq c$, we get 
$$
\droppable_{c,k+1}(\comp{b'}) = \droppable_{c,k+1}(\comp{b'}) - 1 < \droppable_{c,k+1}(\comp{a}).
$$
Therefore, 
$$
\droppable_{c,k+1}(\thread(V)) \leq \droppable_{c,k+1}(\comp{b'}) < \droppable_{c,k+1}(\comp{a}).
$$

In conclusion, we have two directly contradicting inequalities $\droppable_{c,k+1}(\thread(V)) \geq \droppable_{c,k+1}(\comp{a})$ and $\droppable_{c,k+1}(\thread(V)) < \droppable_{c,k+1}(\comp{a})$, and so $c$ must be unique. 
\end{proof}

We may now proceed with the proof of Theorem~\ref{thm:RSKD-hor}. 

\begin{proof}[Proof of Theorem~\ref{thm:RSKD-hor}]
Let $\comp{a}$ be a weak composition.
Through induction, we will show for any given positive integer $m$ we have
\begin{equation}
\#[\KD(\comp{a}) \times \KD(m \comp{e}_k)]_{\comp{b}} = \#[\KDs^{(m)}(\comp{a},k)]_{\comp{b}},
    \label{e:bijection-per-wt-hor}
\end{equation}
for every positive integer $k \leq n$ and every weak composition $\comp{b}$.
For our base case, since $\KDs^{(1)}(\comp{a},k) = \KD(\comp{a})$ for every positive integer $k \leq n$, 
we know
Eq.~\eqref{e:bijection-per-wt-hor} holds given $m=1$
by way of the weight-preserving bijection in Eq.~\eqref{e:RSKD} of Theorem~\ref{thm:RSKD}.

Suppose, as our induction hypothesis,
for some $m > 1$ Eq.~\eqref{e:bijection-per-wt-hor} holds 
given $m-1$.
We will show Eq.~\eqref{e:bijection-per-wt-hor}
also holds given $m$.
To start, fix a weak composition $\comp{b}$.
We immediately get the
inequality $\#[\KD(\comp{a}) \times \KD(m\comp{e}_n)]_{\comp{b}} \leq \#[\KDs^{(m)}(\comp{a},n)]_{\comp{b}}$
by the weight-preserving injection in Eq.~\eqref{e:insert-top-hor} of Theorem~\ref{thm:insertion-top-hor}.
We want to show that, inversely,
$\#[\KDs^{(m)}(\comp{a},n)]_{\comp{b}} \leq \#[\KD(\comp{a}) \times \KD(m\comp{e}_n)]_{\comp{b}}$.
By stratification, we have
$$
\KDs^{(m)}(\comp{a},n) = \KDs^{(m)}(\comp{a},1) \sqcup \bigsqcup_{1 < k \leq n} \KDstrm.
$$
The weight-preserving bijection in Eq.~\eqref{e:insert-bot-hor} of Theorem~\ref{thm:insertion-bottom-hor} gives us 
\begin{equation}
    \#[\KDs^{(m)}(\comp{a},1)]_{\comp{b}} = \#[\KD(\comp{a}) \times \KD(m \comp{e}_1)]_{\comp{b}}.
        \label{e:bijection-per-wt-bot-hor}
\end{equation}
On the other hand, 
for every $1 < k \leq n$,
the 
degree-$m$
stratum map 
$\partial^{(m)}_{\comp{a},k}$,
which is injective by Theorem~\ref{thm:stratum-inject-hor}
and excises a cell in row $k$, gives us
$$
\#[\KDstrm]_{\comp{b}} \leq \#[\KDs^{(m-1)}(\comp{a},k)]_{\comp{b}- \comp{e}_k}.
$$
where by our induction hypothesis, 
$$
\#[\KDs^{(m-1)}(\comp{a},k)]_{\comp{b}-\comp{e}_k} = \#[\KD(\comp{a}) \times \KD((m-1)\comp{e}_k)]_{\comp{b} - \comp{e}_k}.
$$
It follows 
$$
\#[\KDstrm]_{\comp{b}} \leq \#[\KD(\comp{a}) \times \KD((m-1) \comp{e}_k)]_{\comp{b}-\comp{e}_k}.
$$
Now observe for each $1 < k \leq n$ there exists a bijection
$$
\KD((m-1)\comp{e}_k) \stackrel{\sim}{\longrightarrow} \KD(m \comp{e}_k) \setminus \KD(m \comp{e}_{k-1})
$$
that sends $H \mapsto (\{ \kd_{\comp{e}_k} \} \sqcup H')$, 
where $H'$ is obtained from $H$ by pushing each cell of $H$ to the position in its immediate right. 
Consequently,
$$
\#[\KD((m-1)\comp{e}_k)]_{\comp{b} - \comp{e}_k} = \#[\KD(m \comp{e}_k) \setminus \KD(m \comp{e}_{k-1})]_{\comp{b}},
$$
and we get
\begin{equation}
\#[\KDstrm]_{\comp{b}} \leq \#[\KD(\comp{a}) \times (\KD(m \comp{e}_k) \setminus \KD(m \comp{e}_{k-1}))]_{\comp{b}}
    \label{e:injection-per-wt-strat-hor}
\end{equation}
for each $1 < j \leq k$.

Eq.~\eqref{e:bijection-per-wt-bot-hor}~and~\eqref{e:injection-per-wt-strat-hor} together imply
\[
\begin{array}{r@{}l}
    \#[\KDs^{(m)}(\comp{a},n)]_{\comp{b}} &{}= \#[\KDs^{(m)}(\comp{a},1)]_{\comp{b}} + \sum_{1 < k \leq n} \#[\KDstrm]_{\comp{b}} \\
     &{}\leq \#[\KD(\comp{a}) \times \KD(m \comp{e}_1)]_{\comp{b}} \\
     &{\hspace{0.7cm}} + \sum_{1 < k \leq n} \#[\KD(\comp{a}) \times (\KD(m \comp{e}_k) \setminus \KD(m \comp{e}_{k-1}))]_{\comp{b}} \\
    &{}\leq \#[\KD(\comp{a}) \times \KD(m \comp{e}_n)]_{\comp{b}}.
\end{array}
\]
Thus,
$
\#[\KD(\comp{a}) \times \KD(m \comp{e}_n)]_{\comp{b}} = \#[\KDs^{(m)}(\comp{a},n)]_{\comp{b}},
$
which means for each $1 < k \leq n$, the inequality in Eq.~\eqref{e:injection-per-wt-strat-hor} is in fact an equality.
So for each $1 < k \leq n$, we have
\[
\begin{array}{r@{}l}
    \#[\KDs^{(m)}(\comp{a},k)]_{\comp{b}} &{}= \#[\KDs^{(m)}(\comp{a},1)] + \sum_{1 < j \leq k} \#[KDstrm]_{\comp{b}} \\
    &{}= \#[\KD(\comp{a}) \times \KD(m \comp{e}_1)]_{\comp{b}} \\
    &{\hspace{0.7cm}} + \sum_{1 < j \leq k} \#[\KD(\comp{a}) \times (\KD(m \comp{e}_j) \setminus \KD(m \comp{e}_{j-1}))]_{\comp{b}} \\
    &{}= \#[\KD(\comp{a}) \times \KD(m \comp{e}_k)]_{\comp{b}}.
\end{array}
\]
In conclusion, Eq.~\eqref{e:bijection-per-wt-hor} holds given $m$, as desired.
\end{proof}

\subsection{Positive expansions}
\label{sec:applications-pos}

By Theorem~\ref{thm:monkey}, the key polynomial expansion of $\key_{\comp{a}} \cdot \key_{\comp{e}_k}$ is nonnegative if and only if the $k$-addable row set for $\comp{a}$ for each $k$-addable column $c$ is a singleton. We consider several interesting cases where this holds.

First, consider the subcase in which $k=1$. Here, Theorem~\ref{thm:monkey} follows from our direct bottom insertion algorithm proved in Theorem~\ref{thm:insertion-bottom}. Since the $1$-addable row set is either empty or $\{1\}$, the resulting formula is always nonnegative.

\begin{corollary}
  Let $\comp{a}$ be a weak composition. Then we have
  \begin{equation}
    \key_{\comp{a}} \cdot s_{(m)}(x_1) = \key_{\comp{a}} \cdot (x_1)^m = \sum_{\substack{ a_1 \leq c_1 < \cdots < c_m \\ c_i-1 \in \{a_1,\ldots,a_n,c_{i-1}\} }} \key_{\comp{a}^{(m)}} ,
    \label{e:pieri-key-bottom}
  \end{equation}
  where $\comp{a}^{(0)} = \comp{a}$ and $\comp{a}^{(i)} = \supp_{\comp{a}^{(i-1)}}^{(c_{i},1)} + \comp{e}_{1}$ for $1 \leq i \leq m$.
  \label{cor:pieri-key-bottom}
\end{corollary}

\begin{example}
  Consider the weak composition $\comp{a}=(1,4,0,3)$ and the integer $m=2$. To compute $\key_{\comp{a}} \cdot s_{(m)}(x_1)$, we choose $c_1 \geq 1$ from set obtained by adding $1$ to each part of $\comp{a}$, namely $c_1\in\{2,5,1,4\}$. Then we choose $c_2 > c_1$ from the same set now with the addition of $c_1+1$, namely $c_2\in\{2,5,1,4,c_1+1\}$. This results in five possible choices for $c_1,c_2$, with supports depicted in Fig.~\ref{fig:pieri-key-bottom}, and so we have
  \[ \key_{(1,4,0,3)} \cdot s_{(2)}(x_1) = \key_{(3,4,0,3)} + \key_{(4,4,0,2)} + \key_{(5,2,0,3)} + \key_{(5,4,0,1)} + \key_{(6,1,0,3)} . \]
  \label{ex:pieri-key-bottom}
\end{example}

\begin{figure}[ht]
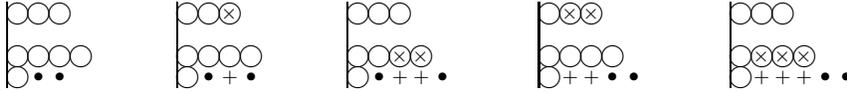

  \begin{displaymath}
    \arraycolsep=2\cellsize
    \begin{array}{ccccc}
      \vline\nulltab{ \circify{\ } & \circify{\ } & \circify{\ } \\ \\ \circify{\ } & \circify{\ } & \circify{\ } & \circify{\ } \\ \circify{\ } & \bullet & \bullet } &
      \vline\nulltab{ \circify{\ } & \circify{\ } & \circify{\times} \\ \\ \circify{\ } & \circify{\ } & \circify{\ } & \circify{\ } \\ \circify{\ } & \bullet & + & \bullet } &
      \vline\nulltab{ \circify{\ } & \circify{\ } & \circify{\ } \\ \\ \circify{\ } & \circify{\ } & \circify{ \times } & \circify{ \times } \\ \circify{\ } & \bullet & + & + & \bullet } &
      \vline\nulltab{ \circify{\ } & \circify{\times} & \circify{\times} \\ \\ \circify{\ } & \circify{\ } & \circify{\ } & \circify{\ } \\ \circify{\ } & + & + & \bullet & \bullet } &
      \vline\nulltab{ \circify{\ } & \circify{\ } & \circify{\ } \\ \\ \circify{\ } & \circify{\times} & \circify{\times} & \circify{ \times } \\ \circify{\ } & + & + & + & \bullet & \bullet }
    \end{array}
  \end{displaymath}
  \caption{\label{fig:pieri-key-bottom}An example of the nonnegative Pieri rule for bottom insertion, computing the key expansion of the product $\key_{(1,4,0,3)} \key_{(2)}$, where the two marked cells $(\bullet)$ denote the added horizontal $2$-strip and crossed cells $(\times)$ drop to the marked positions $(+)$.}
\end{figure}

Second, consider the subcase also considered by Haglund, Luoto, Mason and van Willigenburg \cite{HLMvW11-2} in which $k \geq \ell(\comp{a})$, where $\ell(\comp{a})$ denotes the largest index $i$ for which $a_i > 0$. Here, Theorem~\ref{thm:monkey} follows from our direct insertion algorithm via rectification proved in Theorem~\ref{thm:insertion-top}. Nonnegativity follows from the fact that there are no nonempty rows above row $k$ from which cells can be dropped, so each $k$-addable row set is either empty or consists of the highest row of the given length.

\begin{corollary}
  Let $\comp{a}$ be a weak composition, and set $\ell = \max_i\{a_i>0\}$. For every positive integer $\ell \leq k \leq n$, we have
  \begin{equation}
    \key_{\comp{a}} \cdot s_{(m)}(x_1,\ldots,x_k) =
    \key_{\comp{a}} \cdot (x_1+\cdots +x_k)^m =
    \sum_{\substack{ c_1 < \cdots < c_m \\ c_i-1 \in \{a_1,\ldots,a_n,c_{i-1}\} }} \key_{\comp{a}+\comp{e}_{j_1}+\cdots +\comp{e}_{j_m}} ,
    \label{e:pieri-key-top}
  \end{equation}
  where $j_i$ is the largest index of $\comp{b} = \comp{a}+\comp{e}_{j_1}+\cdots +\comp{e}_{j_{i-1}}$ such that $b_{j_i} = c_i-1$.
  \label{cor:pieri-key-top}
\end{corollary}

\begin{example}
  Consider the weak composition $\comp{a}=(1,4,0,3)$ and the integers $k=4$ and $m=2$. To compute $\key_{\comp{a}} \cdot s_{(m)}(x_1,\ldots,x_4)$, we choose $c_1$ from set obtained by adding $1$ to each part of $\comp{a}$, namely $c_1\in\{2,5,1,4\}$. Then we choose $c_2 > c_1$ from the same set now with the addition of $c_1+1$, namely $c_2\in\{2,5,1,4,c_1+1\}$. This results in eight choices for $c_1,c_2$, depicted in Fig.~\ref{fig:pieri-key-top}, and so we have
  \begin{multline*}
    \key_{(1,4,0,3)} \cdot s_{(2)}(x_1,\ldots,x_4) = \key_{(1,4,2,3)} + \key_{(1,4,1,4)} + \key_{(1,5,1,3)} + \key_{(3,4,0,3)} \\ + \key_{(2,4,0,4)} + \key_{(2,5,0,3)} + \key_{(1,4,0,5)} + \key_{(1,6,0,3)}.
  \end{multline*}
  Notice if we take $k>4$, then the expansion above changes by applying the transposition $t_{3,k}$ to each indexing composition, that is, swapping $b_3$ and $b_k=0$.
  \label{ex:pieri-key-top}
\end{example}

\begin{figure}[ht]
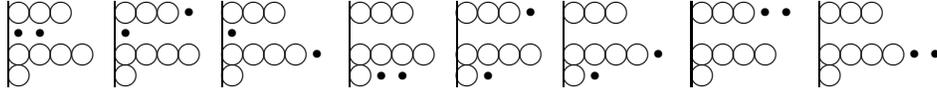

  \begin{displaymath}
    \arraycolsep=4pt
    \begin{array}{cccccccc}
      \vline\nulltab{ \circify{\ } & \circify{\ } & \circify{\ } \\ \bullet & \bullet \\ \circify{\ } & \circify{\ } & \circify{\ } & \circify{\ } \\ \circify{\ } & & } & 
      \vline\nulltab{ \circify{\ } & \circify{\ } & \circify{\ } & \bullet \\ \bullet \\ \circify{\ } & \circify{\ } & \circify{\ } & \circify{\ } \\ \circify{\ } & & } & 
      \vline\nulltab{ \circify{\ } & \circify{\ } & \circify{\ } \\ \bullet \\ \circify{\ } & \circify{\ } & \circify{\ } & \circify{\ }  & \bullet \\ \circify{\ } & & } & 
      \vline\nulltab{ \circify{\ } & \circify{\ } & \circify{\ } \\ \\ \circify{\ } & \circify{\ } & \circify{\ } & \circify{\ } \\ \circify{\ }  & \bullet & \bullet } & 
      \vline\nulltab{ \circify{\ } & \circify{\ } & \circify{\ }  & \bullet \\ \\ \circify{\ } & \circify{\ } & \circify{\ } & \circify{\ } \\ \circify{\ } & \bullet & }  &
      \vline\nulltab{ \circify{\ } & \circify{\ } & \circify{\ } \\ \\ \circify{\ } & \circify{\ } & \circify{\ } & \circify{\ }  & \bullet\\ \circify{\ } & \bullet & } &
      \vline\nulltab{ \circify{\ } & \circify{\ } & \circify{\ } & \bullet & \bullet \\ \\ \circify{\ } & \circify{\ } & \circify{\ } & \circify{\ } \\ \circify{\ } & & } &
      \vline\nulltab{ \circify{\ } & \circify{\ } & \circify{\ } \\ \\ \circify{\ } & \circify{\ } & \circify{\ } & \circify{\ } & \bullet & \bullet \\ \circify{\ } & & } 
    \end{array}
  \end{displaymath}
  \caption{\label{fig:pieri-key-top}An example of the nonnegative Pieri rule for top insertion, computing the key expansion of  $\key_{(1,4,0,3)} \key_{(0,0,0,2)}$, where the two marked cells $(\bullet)$ denote the added horizontal $2$-strip and crossed cells $(\times)$ drop to the marked positions $(+)$.}
\end{figure}

Our third subcase is more involved. For motivation, let $\mathcal{F}_n$ be the complete flag manifold of nested linear subspaces in $\mathbb{C}^n$. Borel \cite{Bor53} showed the cohomology ring $H^{*}(\mathcal{F})$ is the polynomial ring $\mathbb{C}[x_1,\ldots,x_n]$ quotiented by the ideal generated by symmetric polynomials with no constant term, and it has a distinguished linear basis of cosets $[X_w]$ for each permutation $w$. The \emph{Littlewood--Richardson coefficients} $c_{u,v}^{w}$ are the structure constants of this basis, given by
\begin{equation}
  [X_u] \cdot [X_v] = \sum_{w} c_{u,v}^{w} [X_w] .
  \label{e:cuvw}
\end{equation}
As each $[X_w]$ represents the corresponding Schubert variety $X_w$, these Littlewood--Richardson coefficients give intersection multiplicities for the corresponding Schubert varieties and as such are nonnegative. It remains an important open problem to give a direct, nonnegative combinatorial formula for $c_{u,v}^{w}$.

Lascoux and Sch{\"u}tzenberger \cite{LS82} introduced \emph{Schubert polynomials} $\schubert_w$ as polynomial representatives of Schubert classes with nice algebraic and combinatorial properties. Most importantly, Schubert polynomials are a basis for the polynomial ring whose structure constants are precisely the Littlewood--Richardson coefficients,
\begin{equation}
  \schubert_u \cdot \schubert_v = \sum_{w} c_{u,v}^{w} \schubert_w .
  \label{e:cuvw-schub}
\end{equation}
When $w$ is a \emph{Grassmannian permutation}, meaning it has at most one descent, say at position $k$, then Lascoux and Sch{\"u}tzenberger showed the Schubert polynomial $\schubert_w$ is a Schur polynomial in $k$ variables. Generalizing this, they define \emph{vexillary permutations} to be those $w$ such that there do not exist indices $1 \leq a < b < c < d$ for which $w_b < w_a < w_d < w_c$ (for various equivalent characterizations, see \cite{Mac91}). Lascoux and Sch{\"u}tzenberger showed the Schubert polynomial of a vexillary permutation is a key polynomial.

Reversing this characterization, the \emph{Lehmer code} of a permutation $w$, denoted by $\Le(w)$, is the weak composition whose $i$th term is the number of indices $j>i$ for which $w_i > w_j$. This defines a bijection between weak compositions and permutations. Macdonald \cite{Mac91} defined the following.

\begin{definition}[\cite{Mac91}]
  A weak composition $\comp{a}$ is \emph{vexillary} if the following hold:
  \begin{enumerate}
  \item\label{i:vex1} if $i<k$ and $a_i > a_k$, then $\#\{ i<j<k \mid a_j < a_k\} \leq a_i-a_k$;
  \item\label{i:vex2} if $i<k$ and $a_i \leq a_k$, then $a_j \geq a_i$ whenever $i<j<k$.
  \end{enumerate}
  \label{def:vex}
\end{definition}

\begin{example}
  The weak composition $\comp{a}=(1,4,0,3)$ is not vexillary. While condition~\eqref{i:vex1} for all relevant pairs, condition~\ref{i:vex2} fails for $i,j,k = 2,3,4$, respectively, since $a_j < a_i < a_k$. However the weak composition $\comp{b}=(0,1,4,3)$ is vexillary as both conditions are met. For this composition, the corresponding permutation is $w=1376245$ which indeed satisfies $\Le(w)=\comp{b}$. Notice as well $w$ is vexillary.
  \label{ex:vex}
\end{example}

Macdonald proved the vexillary conditions for permutations and weak compositions coincide under the Lehmer correspondence \cite[(1.32)]{Mac91}.

\begin{proposition}[\cite{Mac91}]
  A permutation $w$ is vexillary if and only if $\Le(w)$ is vexillary, and in this case $\schubert_w = \key_{\Le(w)}$.
\end{proposition}

Therefore the product of two key polynomials indexed by vexillary compositions is of particular geometric importance. Condition~\eqref{i:vex2} of Definition~\ref{def:vex} is enough to ensure positivity of the key polynomial expansion in Theorem~\ref{thm:pieri-hor}. In fact, this condition is tight in the following sense.

\begin{theorem}
  Let $\comp{a}$ be a weak composition and $m$ a positive integer. Then the key polynomial expansion of $\key_{\comp{a}} s_{(m)}(x_1,\ldots,x_k)$ is nonnegative for all positive integers $k \leq n$ if and only if $\comp{a}$ satisfies Definition~\ref{def:vex}\eqref{i:vex2}.

  In particular, for $\comp{a}$ vexillary corresponding to $w$ and $v((m),k)$ the grassmannian permutation corresponding to the partition $(m)$ with unique descent at $k$, we have
  \begin{equation}
    \schubert_w \schubert_{v((m),k)} = 
    \key_{\comp{a}} \cdot s_{(m)}(x_1,\ldots,x_k) =
    \sum_{\substack{ \min(a_1,\ldots,a_k) < c_1 < \cdots < c_m \\ c_i-1 \in \{a_1,\ldots,a_n,c_{i-1}\} }}
    \key_{\comp{a}^{(m)}} ,
    \label{e:pieri-key-vex}
  \end{equation}
  where $\comp{a}^{(0)} = \comp{a}$ and $\comp{a}^{(i)} = \supp_{\comp{a}^{(i-1)}}^{(c_{i},r_{i})} + \comp{e}_{r_{i}}$, for $1 \leq i \leq m$, where $r_i$ is the largest row index $r \leq k$ for which $\comp{a}^{(i-1)}_r = c_i-1$, if it exists, or $r_i$ is the largest row index $r \leq k$ for which $\comp{a}^{(i-1)}_r < c_i-1$.
  \label{thm:pieri-key-vex}  
\end{theorem}

\begin{proof}
  Suppose $\comp{a}$ satisfies condition~\eqref{i:vex2} of Definition~\ref{def:vex}. We proceed by induction on $m$. For the base case, take $c_1$ such that $\min(a_1,\ldots,a_k) \leq c_1-1 \in \{a_1,\ldots,a_n\}$, and let $R_1$ denote the $k$-addable row set for $\comp{a}$ in column $c_1$. Let $r_1$ be as in the statement of the theorem.

  We claim $R_1 = \{r_1\}$. If $a_{r_1} = c_1-1$, then $r_1 \in R_1$ by choice of $r_1$ as the highest row of that length lying weakly below row $k$. If there exists some other row index $s \in R_1$, then we must have $r_1 < s \leq k$ with $a_{s} < c_1-1$ and a row index $t > k$ such that $a_t = c_1-1$. However, in this case, condition~\eqref{i:vex2} of Definition~\ref{def:vex} ensures no row strictly between rows $r_1$ and $t$ has length strictly less than $c_1-1$, a contradiction with $r_1 < s < t$ and $a_s < c_1-1$. Thus $R_1 = \{r_1\}$ in this case. If $a_{r_1} < c_1-1$, then $a_i \neq c_1-1$ for any $i \leq k$ by choice of $r_1$. Since $c_1-1 \in \{a_1,\ldots,a_n\}$, this ensures there exists some index $t>k$ such that $a_t = c_1-1$. By choice of $r_1$ as the highest row of length less than $c_1-1$, we have $r_1 \in R_1$, and any other $r \in R_1$ must have $r<r_1$. However, in this case we would have $a_r > a_{r_1}$, but condition~\eqref{i:vex2} of Definition~\ref{def:vex} ensures no row strictly between rows $r$ and $t$ has length strictly less than $a_r$, a contradiction with $r < r_1 < t$ and $a_r > a_{r_1}$. Therefore $R_1 = \{r_1\}$ in this case as well, proving the claim, and the base case follows.

  Continuing the notation, suppose, for induction, $R_i = \{r_i\}$ for all $i< m$. Comparing $\comp{a}^{(m-1)}$ with $\comp{a}$, the only changes weakly below row $k$ occur at rows $r_1,\ldots,r_{m-1}$ which get strictly longer, and the changes for rows strictly above row $k$ result in rows of length at most $c_{m-1}-1$ getting weakly shorter. Therefore while condition~\eqref{i:vex2} of Definition~\ref{def:vex} does not necessarily hold for $\comp{a}^{(m-1)}$ as it did for $\comp{a}$, the condition does hold provided the lower row is weakly below $k$ and the higher row is strictly above $k$. Repeating our arguments from the base case above, this indeed was the only application of the condition needed, and so we conclude $R_m = \{r_m\}$. Therefore nonnegativity for all $k$ follows whenever condition~\eqref{i:vex2} of Definition~\ref{def:vex} holds, as does the explicit formula in Eq.~\eqref{e:pieri-key-vex} for the case when $\comp{a}$ is vexillary. 

  Conversely, suppose condition~\eqref{i:vex2} of Definition~\ref{def:vex} fails for $\comp{a}$. Choose $r$ so that $a_r$ is maximal and then $r$ is maximal such that there exists $s$ and $t$ with $r<s<t$ and $a_s < a_r \leq a_t$. Then choose $s$ maximal among all such indices satisfying that condition. Let $R$ denote the $s$-addable row set for $\comp{a}$ in column $a_t+1$. Then $r \in R$ by choice of $a_r$ as large as possible and $r$ as high as possible, and $s \in R$ as well by choice of $k=s$. Therefore the key polynomial expansion of $\key_{\comp{a}} s_{(m)}(x_1,\ldots,x_s)$ will have terms with negative signs.
\end{proof}

\begin{example}
  Departing from our running example as it is not vexillary, consider the weak composition $\comp{a}=(0,1,4,3)$ which is vexillary, and take $k=3$ (so as to avoid redundancy with prior examples) and $m=2$. To compute $\key_{\comp{a}} \cdot s_{(m)}(x_1,x_2,x_3)$, we choose $c_1 > \min(a_1,a_2,a_3)=0$ from set obtained by adding $1$ to each part of $\comp{a}$, so again $c_1\in\{1,2,5,4\}$. Then we choose $c_2 > c_1$ from the same set now with the addition of $c_1+1$, namely $c_2\in\{1,2,5,4,c_1+1\}$. Thus the choices correspond precisely with those in Ex.~\ref{ex:pieri-key-top}, though as seen in Fig.~\ref{fig:pieri-key-vex}, the support is vastly different, giving the expansion
  \begin{multline*}
    \key_{(0,1,4,3)} \cdot s_{(2)}(x_1,x_2,x_3) = \key_{(1,2,4,3)} + \key_{(1,4,4,1)} + \key_{(1,1,5,3)} + \key_{(0,3,4,3)} \\ + \key_{(0,4,4,2)} + \key_{(0,2,5,3)} + \key_{(0,4,5,1)} + \key_{(0,1,6,3)}.
  \end{multline*}
  \label{ex:pieri-key-vex}
\end{example}

\begin{figure}[ht]
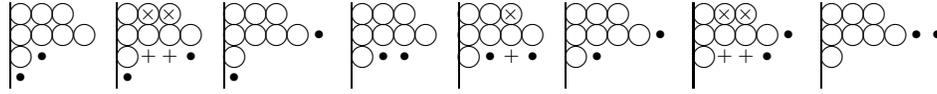

  \begin{displaymath}
    \arraycolsep=4pt
    \begin{array}{cccccccc}
      \vline\nulltab{ \circify{\ } & \circify{\ } & \circify{\ } \\ \circify{\ } & \circify{\ } & \circify{\ } & \circify{\ } \\ \circify{\ } & \bullet \\ \bullet } & 
      \vline\nulltab{ \circify{\ } & \circify{\times} & \circify{\times} \\ \circify{\ } & \circify{\ } & \circify{\ } & \circify{\ } \\ \circify{\ } & + & + & \bullet \\ \bullet } & 
      \vline\nulltab{ \circify{\ } & \circify{\ } & \circify{\ } \\ \circify{\ } & \circify{\ } & \circify{\ } & \circify{\ } & \bullet \\ \circify{\ }  \\ \bullet } & 
      \vline\nulltab{ \circify{\ } & \circify{\ } & \circify{\ } \\ \circify{\ } & \circify{\ } & \circify{\ } & \circify{\ } \\ \circify{\ } & \bullet & \bullet \\ & } & 
      \vline\nulltab{ \circify{\ } & \circify{\ } & \circify{\times} \\ \circify{\ } & \circify{\ } & \circify{\ } & \circify{\ } \\ \circify{\ } & \bullet & + & \bullet \\ & } & 
      \vline\nulltab{ \circify{\ } & \circify{\ } & \circify{\ } \\ \circify{\ } & \circify{\ } & \circify{\ } & \circify{\ } & \bullet \\ \circify{\ } & \bullet \\ & } & 
      \vline\nulltab{ \circify{\ } & \circify{\times} & \circify{\times} \\ \circify{\ } & \circify{\ } & \circify{\ } & \circify{\ } & \bullet \\ \circify{\ } & + & + & \bullet \\ &  } & 
      \vline\nulltab{ \circify{\ } & \circify{\ } & \circify{\ } \\ \circify{\ } & \circify{\ } & \circify{\ } & \circify{\ } & \bullet & \bullet \\ \circify{\ } \\ } 
    \end{array}
  \end{displaymath}
  \caption{\label{fig:pieri-key-vex}An example of the nonnegative Pieri rule in the vexillary case, computing the key expansion of $\key_{(0,1,4,3)} \key_{(0,0,2)}$, where the two marked cells $(\bullet)$ denote the added horizontal $2$-strip and crossed cells $(\times)$ drop to the marked positions $(+)$.}
\end{figure}

Lascoux and Sch{\"u}tzenberger \cite{LS90} stated a nonnegative formula for the key polynomial expansion of a Schubert polynomial, with proof details supplied by Reiner and Shimozono \cite{RS95}. Thus the nonnegativity of the key polynomial expansion follows by observing both terms on the left hand side are Schubert polynomials, and so their product is a nonnegative sum of Schubert polynomials, hence is a nonnegative sum of key polynomials. However, since the key polynomial expansion of a Schubert polynomial is indirect, our formula is indeed new. 

Monk \cite{Mon59} proved a formula for the Schubert structure constants $c_{u,v}^{w}$ when $v$ is a simple transposition, equivalently, when $v$ is grassmannian corresponding to the partition $(1)$. Sottile \cite{Sot96} proved a generalization to the Pieri rule for Schubert polynomials, when $v$ is grassmannian corresponding to the partition $(m)$.

Comparing our formula for the vexillary case with the Monk and Pieri rules for Schubert polynomials might lead to a simplified proof of the Schubert structure constants for the case when $u$ is vexillary. Indeed, one hopes that a full Littlewood--Richardson rule for key polynomials could give important insights into finding a long sought formula for the Schubert structure constants $c_{u,v}^{w}$, even if only in the open case where both $u$ and $v$ are vexillary.

%
%

\bibliographystyle{plain} 
\bibliography{monkey}

\end{document}